\documentclass[11pt]{article}%

\usepackage{amssymb,amsmath,amsfonts,amsthm,array,bm,bbm,color}
\usepackage{amsmath}
\usepackage{graphicx}
\usepackage{enumitem}
\usepackage{appendix}
\usepackage{mathrsfs}
\usepackage[normalem]{ulem}

\providecommand{\U}[1]{\protect\rule{.1in}{.1in}}

\usepackage[
a4paper,
left=22mm,
right=22mm,
top=22mm,
bottom=25mm
]{geometry}

\numberwithin{equation}{section}

\newtheorem{theorem}{Theorem}[section]
\newtheorem{lemma}[theorem]{Lemma}
\newtheorem{corollary}[theorem]{Corollary}
\newtheorem{proposition}[theorem]{Proposition}
\newtheorem{remark}[theorem]{Remark}

\newtheorem{definition}[theorem]{Definition}
\newtheorem{hypotheses}[theorem]{Hypotheses}

\def\<{\langle}
\def\>{\rangle}
\def\E{\mathbb{E}}
\def\P{\mathbb{P}}
\def\R{\mathbb{R}}
\def\T{\mathbb{T}}
\def\Z{\mathbb{Z}}

\RequirePackage[normalem]{ulem}
\RequirePackage{color}\definecolor{RED}{rgb}{1,0,0}\definecolor{BLUE}{rgb}{0,0,1}

\providecommand{\DIFdel}[1]{{\protect\color{red}\ifmmode\text{\sout{\mbox{$\displaystyle #1$}}}\else\sout{#1}\fi}}                      
\providecommand{\DIFaddbegin}{} 
\providecommand{\DIFaddend}{} 
\providecommand{\DIFdelbegin}{} 
\providecommand{\DIFdelend}{} 
\providecommand{\DIFaddbeginFL}{} 
\providecommand{\DIFaddendFL}{} 
\providecommand{\DIFdelbeginFL}{} 
\providecommand{\DIFdelendFL}{} 
\newcommand{\DIFscaledelfig}{0.5}
\RequirePackage{settobox} 
\RequirePackage{letltxmacro} 
\newsavebox{\DIFdelgraphicsbox} 
\newlength{\DIFdelgraphicswidth} 
\newlength{\DIFdelgraphicsheight} 
\LetLtxMacro{\DIFOincludegraphics}{\includegraphics} 
\newcommand{\DIFaddincludegraphics}[2][]{{\color{blue}\fbox{\DIFOincludegraphics[#1]{#2}}}} 
\newcommand{\DIFdelincludegraphics}[2][]{
\sbox{\DIFdelgraphicsbox}{\DIFOincludegraphics[#1]{#2}}
\settoboxwidth{\DIFdelgraphicswidth}{\DIFdelgraphicsbox} 
\settoboxtotalheight{\DIFdelgraphicsheight}{\DIFdelgraphicsbox} 
\scalebox{\DIFscaledelfig}{
\parbox[b]{\DIFdelgraphicswidth}{\usebox{\DIFdelgraphicsbox}\\[-\baselineskip] \rule{\DIFdelgraphicswidth}{0em}}\llap{\resizebox{\DIFdelgraphicswidth}{\DIFdelgraphicsheight}{
\setlength{\unitlength}{\DIFdelgraphicswidth}
\begin{picture}(1,1)
\thicklines\linethickness{2pt} 
{\color[rgb]{1,0,0}\put(0,0){\framebox(1,1){}}}
{\color[rgb]{1,0,0}\put(0,0){\line( 1,1){1}}}
{\color[rgb]{1,0,0}\put(0,1){\line(1,-1){1}}}
\end{picture}
}\hspace*{3pt}}} 
} 
\LetLtxMacro{\DIFOaddbegin}{\DIFaddbegin} 
\LetLtxMacro{\DIFOaddend}{\DIFaddend} 
\LetLtxMacro{\DIFOdelbegin}{\DIFdelbegin} 
\LetLtxMacro{\DIFOdelend}{\DIFdelend} 
\DeclareRobustCommand{\DIFaddbegin}{\DIFOaddbegin \let\includegraphics\DIFaddincludegraphics} 
\DeclareRobustCommand{\DIFaddend}{\DIFOaddend \let\includegraphics\DIFOincludegraphics} 
\DeclareRobustCommand{\DIFdelbegin}{\DIFOdelbegin \let\includegraphics\DIFdelincludegraphics} 
\DeclareRobustCommand{\DIFdelend}{\DIFOaddend \let\includegraphics\DIFOincludegraphics} 
\LetLtxMacro{\DIFOaddbeginFL}{\DIFaddbeginFL} 
\LetLtxMacro{\DIFOaddendFL}{\DIFaddendFL} 
\LetLtxMacro{\DIFOdelbeginFL}{\DIFdelbeginFL} 
\LetLtxMacro{\DIFOdelendFL}{\DIFdelendFL} 
\DeclareRobustCommand{\DIFaddbeginFL}{\DIFOaddbeginFL \let\includegraphics\DIFaddincludegraphics} 
\DeclareRobustCommand{\DIFaddendFL}{\DIFOaddendFL \let\includegraphics\DIFOincludegraphics} 
\DeclareRobustCommand{\DIFdelbeginFL}{\DIFOdelbeginFL \let\includegraphics\DIFdelincludegraphics} 
\DeclareRobustCommand{\DIFdelendFL}{\DIFOaddendFL \let\includegraphics\DIFOincludegraphics} 
\RequirePackage{listings} 
\RequirePackage{color} 
\lstdefinelanguage{DIFcode}{ 
  moredelim=[il][\color{red}\sout]{\%DIF\ <\ }, 
  moredelim=[il][\color{blue}\uwave]{\%DIF\ >\ } 
} 
\lstdefinestyle{DIFverbatimstyle}{ 
	language=DIFcode, 
	basicstyle=\ttfamily,
	columns=fullflexible,
	keepspaces=true
}
\lstnewenvironment{DIFverbatim}{\lstset{style=DIFverbatimstyle}}{}
\lstnewenvironment{DIFverbatim*}{\lstset{style=DIFverbatimstyle,showspaces=true}}{}

\begin{document}

\title{Eddy viscosity by L\'evy transport noises}

\author{Dejun Luo$^{1,2}$\thanks{Email: luodj@amss.ac.cn} \quad Feifan Teng$^{2,3}$\thanks{Email: tengfeifan@amss.ac.cn} \medskip \\
{\scriptsize $^1$SKLMS, Academy of Mathematics and Systems Science, Chinese Academy of Sciences, Beijing 100190, China} \\
{\scriptsize $^2$School of Mathematical Sciences, University of Chinese Academy of Sciences, Beijing 100049, China} \\
{\scriptsize $^3$Academy of Mathematics and Systems Science, Chinese Academy of Sciences, Beijing 100190, China}\\
}

\maketitle

\vspace{-15pt}

\begin{abstract}
We consider
stochastic linear transport equations in dimension $d\geq 2$
and stochastic 2D Euler equations with $L^2$-initial condition, both driven by L\'evy transport noise in the Marcus sense. Under a suitable scaling limit of the noises, we prove that
the solutions to the stochastic linear transport equations
converge to the solution of the deterministic heat equation, while
weak solutions to the stochastic 2D Euler equations
converge weakly to the unique solution of the deterministic 2D Navier--Stokes equation. These results show that small-scale
jump noises generate eddy viscosity, extending the recent studies on It\^o--Stratonovich diffusion limits to the discontinuous setting.
\end{abstract}


\providecommand{\FPdel}[1]{{\color{red}\ifmmode#1\else\sout{#1}\fi}}
\providecommand{\FPdelblock}[1]{{\color{red}#1}}
\providecommand{\FPadd}[1]{{\color{blue}#1}}
\section{Introduction}

We consider the stochastic 2D Euler equations on the torus $\T^2$, driven by L\'evy transport noise:
  \begin{equation}\label{stoch-2D-Euler}
  \mathrm{d}\xi + u\cdot\nabla\xi\, \mathrm{d}t + \diamond\, \mathrm{d}L\cdot\nabla\xi=0, \quad \xi|_{t=0}= \xi_0,
  \end{equation}
where $\xi$ and $u$ are respectively the fluid vorticity and velocity fields, $L= L(t,x)$ stands for a pure jump L\'evy noise, taking values in the space of divergence free vector fields on $\T^2$. The notation $\diamond\, \mathrm{d}$ means that the stochastic integral is understood in the Marcus sense, for which the stochastic calculus satisfies the usual chain rule, see \cite{Marcus, KPP95} and the recent papers \cite{ChePav14, HP23} for more information. We assume the noise $L= L(t,x)$ takes the form
  $$L(t,x)= \sum_k \sigma_k(x) Z^k_t, $$
where $\{\sigma_k\}_{k}$ are divergence free vector fields on $\T^2$ and $\{Z^k\}_k$ are independent real-valued pure jump L\'evy processes, see Section \ref{subs-struc-noise} below for more precise descriptions. Inspired by the recent studies \cite{FGL21, Gal20, Luo21} on the It\^o-Stratonovich diffusion limit for Brownian transport noise, i.e. where the noises $\{Z^k\}_k$ are replaced by independent standard Brownian motions, and also the first work \cite{FPR25} with jump noise, we show that, under a suitable scaling limit of noise, the above equation \eqref{stoch-2D-Euler} is close to the deterministic 2D Navier-Stokes equation
  $$\partial_t\xi + u\cdot\nabla\xi= \kappa \Delta \xi,$$
where $\kappa>0$ is a constant representing the eddy viscosity. Indeed, to explain more clearly the main ideas underlying the scaling arguments, we will first prove such a limit for the stochastic linear transport equations on $d$-dimensional torus $\T^d$ for general $d\ge 2$.

\subsection{Motivations}\label{subs-motivation}

Before going into details, let us briefly discuss the motivations of the fluid model \eqref{stoch-2D-Euler}, following the arguments in \cite[Section 5.5.1]{FlanLuongo23} or \cite[Section II]{FlaHuang23}. The deterministic 2D Euler equation in vorticity form reads as
  $$\partial_t\xi + u\cdot\nabla\xi=0, $$
subjected to suitable initial and boundary conditions. Eddies or whirls in a fluid might be created due to irregularities of the boundary or small obstacles located in the middle of fluid. Although it is possible to include all these information in a deterministic description of the fluid,  the resulting model is too complicated to be analyzed rigorously from a mathematical point of view. One way to reduce the complexity of the fluid model is to introduce noise in the equation, with suitable space-time properties. Here, we idealize the creation of eddies as a jump process, namely, we assume that such eddies are produced in very short time periods, mathematically modeled as jumps of the fluid velocity:
  $$u(t)= u(t-) + \sum_k \sigma_k\, \Delta Z^k_t, $$
where $t-$ represents the left limit, $\Delta Z^k_t= Z^k_t- Z^k_{t-}$ and the fields $\sigma_k$ are imitations of small eddies generated near the irregular boundary or obstacles. Substituting this decomposition into the above 2D Euler equation leads to the model \eqref{stoch-2D-Euler}, provided that the stochastic integral is understood in a proper sense.

In fact, in the heuristic arguments of \cite[Section 5.5.1]{FlanLuongo23}, the authors take a further Brownian limit of the above jump process, by accelerating the speed of creation of eddies and meanwhile decreasing their sizes accordingly. In the limit, what they get is a Brownian transport noise $W(t,x)= \sum_k \sigma_k(x) B^k_t$, where $\{B^k\}_k$ are independent standard Brownian motions. For this noise, motivated by the Wong-Zakai principle, the stochastic integral in fluid equations is often understood in the Stratonovich sense, which preserves many invariants of the equations, such as Casimirs for the 2D Euler equation.

We recall now some recent studies on the It\^o-Stratonovich diffusion limit of Brownian transport noise. For the stochastic linear transport equations, Galeati \cite{Gal20} showed that, under a suitable scaling limit of the noises, the stochastic equations converge to the deterministic heat equation which has an extra dissipation term. This result was soon extended to the nonlinear case: it was shown in \cite{FGL21} that stochastic 2D Euler equations with $L^2$-initial vorticity $\xi_0$ converge weakly to the deterministic 2D Navier-Stokes equation. As a comparison, we mention that in a slightly earlier work \cite{FlanLuo20}, Flandoli and Luo considered the white noise solutions of stochastic 2D Euler equations, and proved that they converge weakly to the unique stationary solution of 2D Navier-Stokes equation driven by space-time white noise. Regarding the transport noise as small-scale perturbations due to turbulence, the results in \cite{Gal20, FGL21} can be interpreted as the emergence of eddy viscosity, meaning that turbulent small fluctuations are dissipative on the mean part of fluid, a fact often referred to as the Boussinesq hypothesis \cite{Boussinesq}. The It\^o-Stratonovich diffusion limit has been widely extended to various settings, see e.g. \cite{Luo21} for stochastic 2D Boussinesq systems, \cite{CarLuon23} for the two-layer quasi-geostrophic model, \cite{FGL21-CPDE, FL21 PTRF, Lange24, Agresti24, Agresti} for suppression of blow-up of nonlinear equations, and \cite{FGL24, GL24, LuoTang23, ZhangHuang24} for quantitative convergence rates and the related Gaussian fluctuations. We also refer to a few recent papers \cite{ButLuon24, BFL24, FlanLuo24} dealing with certain 3D models which involve also a stretching noise, more difficult than pure transport ones. Jiao and Romito \cite{JiaoRomito26} studied scaling limits under pseudo-transport noise and showed that such noise can delay the potential blow-up of the 3D Euler and Navier-Stokes equations with high probability.  In the above works, the space domain is always taken as the torus $\T^d$ where the noise $W(t,x)$ admits an explicit Fourier series expansion, which facilitates various computations.  For some generalizations to other settings, see \cite{FGL22, FlanLuongo22} for bounded domains and infinite channels, \cite{Huang25} for the case of closed manifolds, and finally \cite{LXZ25} for the full space $\R^d$.

With the above developments in mind, a natural question is to extend the It\^o-Stratonovich diffusion limit to other noises. The first choice might be replacing the Brownian motions in $W(t,x)= \sum_k \sigma_k(x) B^k_t$ by fractional Gaussian processes; this has been done by Flandoli and Russo \cite{FlaRus23}, showing a reduced dissipation property for small times and an enhanced diffusion for large times, compared to the standard Brownian case, cf. the paper \cite{CifFla25} for related numerical results on fractional Gaussian transport noise.
 We also mention the works \cite{Pap22, LiuLuo25} where the linear transport equations and 2D Euler equations driven by the Ornstein-Uhlenbeck transport noise are investigated, establishing finite-time mixing and dissipation enhancement by noise. Another natural extension is to study jump noise as mentioned at the beginning of the paper. Indeed, Flandoli et al. \cite{FPR25} considered stochastic linear transport equations driven by L\'evy transport noise:
  \begin{equation}\label{eq-FPR25}
  \mathrm{d}f(t,x)= (\sigma(x)\nabla f(t,x))\diamond \mathrm{d} Z_t, \quad f(0,x)= f_0(x),
  \end{equation}
where $\sigma:\R^d\to \R^{d\times m}$, $f_0:\R^d\to \R$, and the process $Z$ is an $m$-dimensional pure jump L\'evy process corresponding to the L\'evy measure $\nu$ on $\R^m$. The main result of \cite{FPR25} concerns the case $d=1$ with $\sigma\in \R^m$ being a constant vector, and $\nu$ has a radially symmetric density (e.g. the classical $\alpha$-stable density). In this setting, Flandoli et al. proved that the mean $U(t,x):= \E f(t,x)$ satisfies the deterministic equation
  $$\partial_t U(t,x)= \mathcal L_{1,\alpha} U(t,x),$$
where $\mathcal L_{1,\alpha}$ is similar, but not identical, to a fractional Laplacian operator. For related numerical investigations of the diffusion properties
of stochastic transport driven by heavy-tailed L\'evy jump
processes, see \cite{CifFlaMar26}.

The purpose of the current paper is to extend the result in \cite{FPR25}. In particular, for suitably chosen L\'evy vector fields $L(t,x)$, we show that the solutions, rather than their mean, to stochastic transport equations are close, in a weak sense, to the unique solution of the heat equation. This is stronger than the averaged dissipation enhancement. Moreover, we also extend the result to the nonlinear setting, showing that the stochastic 2D Euler equations \eqref{stoch-2D-Euler} converge weakly to the deterministic 2D Navier-Stokes equation. An interesting point of our work is that the extra dissipation term in limit equations is still the standard Laplacian operator, rather than the fractional counterpart as in \cite{FPR25}. This might be a consequence of the special structure of our noise (see the next section for details), mimicking that of Brownian transport noise. In the limit of small-scale noise, the difference between Brownian and L\'evy transport noises vanishes and the same eddy dissipation operator emerges.

In spite of the above heuristic arguments, identifying the extra dissipation term as a Laplacian operator in our setting requires careful analysis. In the case of Brownian transport noise, this operator arises naturally as the It\^o--Stratonovich corrector.
Indeed, for the L\'evy transport noise in \eqref{stoch-2D-Euler} which is understood in the Marcus sense, the counterpart of the It\^o-Stratonovich corrector is given by a nonlocal operator in terms of an integral with respect to the L\'evy measure. To show that the integral operators converge to a Laplacian in a suitable scaling limit, we need to analyze carefully the difference term in the integral which involves the Marcus maps. This term also causes considerable difficulty in the proof of existence of weak solutions to stochastic 2D Euler equations \eqref{stoch-2D-Euler}, since one has to show that the Marcus maps in the Galerkin approximate equations converge strongly in $L^2(\T^2)$ to the infinite dimensional counterpart in the limit equation. For these reasons, we will first consider the simpler stochastic linear transport equation for which the existence of solutions can be deduced easily from existing references, and then we turn to the nonlinear case.
It would be interesting to apply the rough path theory \cite{FrizShekhar17, ChevyrevFriz19} to establish similar results in our setting.

Before moving forward, we introduce some notations used frequently in the sequel. Denote by \(\mathbb{Z}^d_0 := \mathbb{Z}^d \setminus \{0\}\) the set of nonzero lattice points, and $\ell^2= \ell^2(\mathbb{Z}^d_0)$ the collection of square summable sequences indexed by $\mathbb{Z}^d_0$. For $n\ge 1$, $C^n(\T^d)$ is the space of functions on $\T^d$ continuously differentiable up to order $n$. For $p\ge 1$, we write $L^p_x= L^p(\T^d)$ for the Lebesgue spaces on the torus $\T^d$. Next, for $s\in \R$, we denote by $H^s= H^s(\T^d)$ the usual Sobolev space of order $s$, while $H^0$ will be written as $L^2$. The notation $\<\cdot,\cdot\>$ stands for the inner product in
$L^2$ and also the duality between $H^s$ and $H^{-s}$.
We will use the same notations for spaces of vector fields,
which will not cause confusion according to the context. Given a Polish space $(S,\rho)$ with a metric $\rho$, we shall write $D([0,T];S)$ as the Skorohod space endowed with the Skorohod topology, see Section \ref{subs-Skorohod-space} for more details. Finally, we denote by $C$ some generic constant which may change from line to line and, if we want to stress its dependence on the parameters $d,\phi$, we will write it as $C_{d,\phi}$.

\subsection{Structure of noise}\label{subs-struc-noise}

In this part we introduce in detail the L\'evy noise used in the paper. We first recall the standard trigonometric basis \(\{e_k \}_{k \in \mathbb{Z}^d_0}\) of \(L^2(\mathbb{T}^d)\), defined as
\begin{equation}\label{ek}
    e_k(x) = \sqrt{2} \begin{cases}
\cos(2\pi k \cdot x), & k \in \mathbb{Z}^d_+, \\
\sin(2\pi k \cdot x), & k \in \mathbb{Z}^d_-,
\end{cases}
\end{equation}
with \(\{\mathbb{Z}^d_+, \mathbb{Z}^d_-\}\) forming a partition of \(\mathbb{Z}^d_0\) satisfying \(\mathbb{Z}^d_+ = -\mathbb{Z}^d_-\).

We now present the pure jump L\'evy vector fields defined on a filtered probability space $(\Omega,\mathcal{F}, (\mathcal{F}_t)_{t \geq 0}, \P)$:
\begin{equation}\label{noise}
    L(t,x) = \sum_{k \in \Z^d_0} \sum_{i=1}^{d-1} \theta_k \sigma_{k,i}(x) Z^{k,i}_t.
\end{equation}
Here, the coefficients $\theta = \{\theta_k\}_{k \in \Z^d_0}$ belong to $ \ell^2(\Z^d_0)$, the collection of square summable sequences indexed by $\Z^d_0$; furthermore, $\|\theta\|_{\ell^2} = 1$ and $\theta_k$ is radially symmetric in $k$. For simplicity, we assume that only finitely many components of $\theta$ are nonzero, thus the noise in \eqref{noise} is finite dimensional and smooth in $x\in \T^d$. In the sequel, we write $\sum_{k \in \Z^d_0} \sum_{i=1}^{d-1}$ simply as $\sum_{k,i}$.

Next, the family $\{\sigma_{k,i}\}_{k,i}$ are divergence-free vector fields on $\mathbb{T}^d$, defined by
\begin{equation}\label{sigma_k,i}
    \sigma_{k,i}(x) = a_{k,i} e_k(x),
\end{equation}
where  $\{a_{k,1} ,\ldots, a_{k,d-1} \} \subset \mathbb S^{d-1}$ is an orthonormal basis of the $(d-1)$-dimensional subspace $k^\perp= \{x\in \R^d: k\cdot x=0\}$, satisfying $a_{k,i} = a_{-k,i}$ for all $i = 1, \dots, d-1$. By definition, it is easy to see that $\sigma_{k,i}(x) \cdot \nabla \sigma_{k,i}^j(x) = 0$ for each $j=1,\ldots,d$, $k \in \Z^d_0$, and $i=1,\ldots,d-1$. In the 2D case, we will write $\sigma_{k,1}(x) = a_{k,1} e_k(x)$ simply as $\sigma_k(x)= a_k e_k(x)$, where $a_k=\frac{k^\perp}{|k|}$ for $k\in \Z^2_+$ and $k^\perp= (k_2,-k_1)$. We recall an important identity involving the radially symmetric coefficients $\{\theta_k\}_k$ and the vector fields $\{\sigma_{k,i}\}_{k,i}$:
  \begin{equation}\label{eq-important-equality}
  \sum_{k,i} \theta_k^2 (\sigma_{k,i}\otimes \sigma_{k,i})(x) = C_d I_d,
  \end{equation}
where $C_d=(d-1)/d$ and $I_d$ is the $d\times d$ identity matrix; cf. \cite[Lemma 2.6]{FlanLuo20} for a proof in the 2D case and \cite[(9)]{Gal20} for general case.

Finally, let $\nu(\mathrm{d}z)$ be a L\'evy measure on $\R$ satisfying $\nu(\{0\})=0$ and $\int_{|z| \le 1} z^2\,\nu(\mathrm{d}z)<\infty$, and let $\{N_{k,i}\}_{k,i}$ be independent $(\mathcal{F}_t)$-Poisson random measures on $(\Omega,\mathcal{F}, (\mathcal{F}_t)_{t \geq 0}, \P)$ with compensator $\nu(\mathrm{d}z)\mathrm{d}s$. As discussed in Section \ref{subs-motivation}, we regard the noises as relatively small-scale perturbations of the fluid velocity, and thus we consider only the small jump component, that is,
 \begin{equation}\label{Z_t}
     Z^{k,i}_t = \int_0^t \int_{|z| \le 1} z  \tilde N_{k,i}(\mathrm{d}z, \mathrm{d}s),
 \end{equation}
where $\tilde N_{k,i}(\mathrm{d}z, \mathrm{d}s) = N_{k,i}(\mathrm{d}z, \mathrm{d}s) - \nu(\mathrm{d}z)\mathrm{d}s$ is the compensated Poisson random measure. Then $\{Z^{k,i}_t\}_{k,i}$ is a collection of independent, identically distributed (i.i.d.) one dimensional pure jump L\'evy processes. In particular, $\E[Z^{k,i}_t]=0$ and $\E[|Z^{k,i}_t|^p]<\infty$ for every $p>0$ and $t\ge0$.

We end this part by pointing out that our noise can be written in the same form as in \eqref{eq-FPR25}. In fact, letting $\Gamma:= \{k\in \Z^d_0: \theta_k\neq 0\}$ and $m:= \# \Gamma$, we define the matrix-valued function $\sigma:\T^d\to \R^{d\times (d-1)m}$ whose $(k,i)$-column is just the vector $\theta_k \sigma_{k,i}$; similarly, the L\'evy process $Z_t$ has $Z^{k,i}_t$ as its $(k,i)$-component. In this case, the L\'evy measure $\bar\nu$ of $Z_t$ is a singular measure supported on the coordinate axes in $\R^{(d-1)m}$; in other words, the L\'evy noise considered in this paper is cylindrical.

\subsection{Main results}

This section consists of two parts: we first present the scaling limit result for stochastic linear transport equations, then we state a theorem on the existence of weak solutions to the nonlinear equation \eqref{stoch-2D-Euler}, together with the weak convergence to deterministic 2D Navier-Stokes equation.

\subsubsection{Stochastic linear transport equations} \label{subsubsec-transport-eq}

We first consider the stochastic linear transport equation on $\mathbb{T}^d$:
 \begin{equation}\label{STE-Marcus}
  \quad
  \mathrm{d}\xi+ \sum_{k,i} \theta_k \sigma_{k,i}\cdot \nabla \xi \diamond \mathrm{d}Z^{k,i}_t=0, \quad \xi_0\in L^2(\T^d).
  \end{equation}
Following \cite{HP23}, we understand \eqref{STE-Marcus} as the following integral equation:
\begin{equation}\label{STE-Marcus.1-1}
  \begin{aligned}
  \xi(t) &= \xi_0 + \sum_{k,i} \int_0^t\! \int_{|z|\le 1} \big[ {\rm e}^{-z \theta_k \sigma_{k,i}\cdot\nabla} \xi(s-) - \xi(s-) \big]  \tilde N_{k,i}(\mathrm{d}z,\mathrm{d}s) \\
  &\quad + \sum_{k,i} \int_0^t\! \int_{|z|\le 1} \big[ {\rm e}^{-z \theta_k \sigma_{k,i}\cdot\nabla} \xi(s) - \xi(s) + z \theta_k \sigma_{k,i} \cdot \nabla \xi(s) \big]\,\nu(\mathrm{d}z) \mathrm{d}s.
  \end{aligned}
 \end{equation}
To give the meaning of the operator
${\rm e}^{-z\theta_k\sigma_{k,i}\cdot\nabla}$, we introduce the
flow $\{\varphi_t^{k,i}\}_{t\in\mathbb R}$ generated by
$\sigma_{k,i}$:
\begin{equation}\label{characteristics}
\partial_t\varphi_t^{k,i}(x)
=
\sigma_{k,i}\big(\varphi_t^{k,i}(x)\big),
\quad
\varphi_0^{k,i}(x)=x\in\T^d,
\quad
t\in\mathbb R.
\end{equation}
Since $\sigma_{k,i}$ is smooth, autonomous, and divergence-free,
$\{\varphi_t^{k,i}\}_{t\in\mathbb R}$ is a group of measure-preserving
diffeomorphisms, with
$\varphi_{-t}^{k,i}=(\varphi_t^{k,i})^{-1}$.
Moreover, owing to the particular structure of $\sigma_{k,i}$, the flow
admits the explicit representation
\[
\varphi_t^{k,i}(x)=x+t\sigma_{k,i}(x),
\quad t\in\mathbb R,
\]
see Section~\ref{Marcus} for further discussion. For $t\in\mathbb R$ and $g,f\in L^2(\mathbb T^d)$, we write
\begin{equation}\label{define e}
\begin{aligned}
{\rm e}^{t\sigma_{k,i}\cdot\nabla}g
=
g\circ\varphi_t^{k,i},
\quad
\big\langle {\rm e}^{t\sigma_{k,i}\cdot\nabla}g,f\big\rangle
=
\big\langle g,{\rm e}^{-t\sigma_{k,i}\cdot\nabla}f\big\rangle,
\end{aligned}
\end{equation}
where the second identity follows from the preservation of Lebesgue measure.
In particular,
${\rm e}^{-z\theta_k\sigma_{k,i}\cdot\nabla}g
=
g\circ\varphi_{-z\theta_k}^{k,i}$.

Using these identities, \eqref{STE-Marcus.1-1} can be written as
\begin{equation*}
  \begin{aligned}
  \xi(t) &= \xi_0 + \sum_{k,i} \int_0^t\! \int_{|z|\le 1} \big[ \xi(s-,\varphi_{-z\theta_k}^{k,i}) - \xi(s-) \big]  \tilde N_{k,i}(\mathrm{d}z,\mathrm{d}s) \\
  &\quad + \sum_{k,i} \int_0^t\! \int_{|z|\le 1} \big[ \xi(s,\varphi_{-z\theta_k}^{k,i}) - \xi(s) + z \theta_k \sigma_{k,i} \cdot \nabla \xi(s) \big]\,\nu(\mathrm{d}z) \mathrm{d}s.
  \end{aligned}
\end{equation*}
We shall further interpret the above equation in the weak sense: for any $\phi\in C^\infty(\T^d)$, testing the equation against $\phi$, using the group property and preservation of Lebesgue measure, and integrating by parts, we arrive at
  \begin{equation}\label{STE-Marcus-2}
  \begin{aligned}
  \langle \xi(t),\phi\rangle &= \langle \xi_0,\phi\rangle+ \sum_{k,i} \int_0^t\! \int_{|z|\le 1} \big\langle \xi(s-), \phi\big(\varphi_{z\theta_k}^{k,i} \big) -\phi \big\rangle   \, \tilde N_{k,i}(\mathrm{d}z,\mathrm{d}s) \\
  &\quad + \sum_{k,i}\int_0^t\! \int_{|z|\le 1} \big\langle \xi(s) , \phi\big(\varphi_{z\theta_k}^{k,i} \big) -\phi - z \theta_k\sigma_{k,i} \cdot\nabla \phi \big\rangle  \, \nu(\mathrm{d}z)\mathrm{d}s.
  \end{aligned}
  \end{equation}

If $\xi\in L^2\big(\Omega, L^2\big([0,T];L^2(\mathbb{T}^d)\big)\big)$ is an $(\mathcal{F}_t)$-adapted c\`adl\`ag $L^2(\mathbb{T}^d)$-valued process, then $\xi(s-)$ is predictable, and one can show that all the terms in the above equation make sense, see the beginning of Section \ref{sect:ste} for the computations.

Thanks to the well-posedness results in \cite{HP23} for \eqref{STE-Marcus} with smooth initial data, for any $\xi_0\in L^2(\mathbb{T}^d)$, we shall prove in Section \ref{subs-STE-existence} existence of solutions to the above equation, preserving the $L^2(\mathbb{T}^d)$-norm, see Theorem \ref{cunzaixing} below.

Our main interest is studying the asymptotic behavior of stochastic transport equations under a certain scaling limit of the noise. To this end, we take a sequence of noise coefficients $\{\theta^n\}_n \subset \ell^2(\Z^d_0)$ satisfying
\begin{hypotheses}\label{hypo-coefficients}
\begin{enumerate}
\item[\rm (1)] For every $n\geq1$, the support
$\Gamma_n:=\{k\in\Z_0^d:\theta_k^n\neq0\}$
is finite;
\item[\rm (2)] $\theta_k^n=\theta_j^n$ whenever $|k|=|j|$;
\item[\rm (3)] $\|\theta^n\|_{\ell^2}=1$ for every $n\geq1$;
\item[\rm (4)] $\|\theta^n\|_{\ell^\infty}\to0$ as $n\to\infty$.
\end{enumerate}
\end{hypotheses}
Conditions {\rm (3)}--{\rm (4)} imply that
\[
1
=
\sum_{k\in\Gamma_n}|\theta_k^n|^2
\leq
\#\Gamma_n\,\|\theta^n\|_{\ell^\infty}^2,
\]
and therefore
$\#\Gamma_n\to\infty.$ Here is an example of such sequences: for $a\in (0,d/2)$ and $n\ge 1$, let $\tilde\theta^n_k= \frac1{|k|^a} \textbf{1}_{\{1\le |k|\le n\}}$ and define $\theta^n_k:= \tilde\theta^n_k /\|\tilde\theta^n\|_{\ell^2}$, $k\in \Z^d_0$. Then the family of coefficients $\{\theta^n \}_{n\ge 1} \subset \ell^2$ satisfies the above conditions.

Given such a noise coefficient $\theta^n$, we consider the stochastic transport equation
  \begin{equation}\label{STE-Marcus-n}
  \begin{aligned}
  &\mathrm{d}\xi^n+ \sum_{k,i} \theta^n_k \sigma_{k,i}\cdot \nabla \xi^n \diamond \mathrm{d}Z^{k,i}_t=0, \quad \xi^n_0=\xi_0\in L^2(\T^d).
  \end{aligned}
  \end{equation}
The weak formulation of \eqref{STE-Marcus-n} is similar to \eqref{STE-Marcus-2} with $\theta_k$ and $\xi$ replaced by $\theta_k^n$ and $\xi^n$, respectively.
Here the underlying flow is always $\varphi^{k,i}$ from \eqref{characteristics}; changing $n$, $z$, or $\theta_k^n$ only changes the evaluation time $z\theta_k^n$. Thanks to Theorem \ref{cunzaixing}, for any $n\ge 1$, \eqref{STE-Marcus-n} admits a weak solution $\xi^n$ which is strong in the probabilistic sense, satisfying $\P$-a.s.  $\|\xi^n(t) \|_{L^2}=\|\xi_0 \|_{L^2}$ for all $t\ge 0$.

Now we can state the first main result of our paper.

\begin{theorem}\label{main sca}
Let $\{\theta^n\}_{n\geq1}$ satisfy Hypotheses~\ref{hypo-coefficients}, and let $\xi^n$ be the solution of \eqref{STE-Marcus-n} constructed in Theorem~\ref{cunzaixing}. Then, as $n\to\infty$, for every $\varepsilon>0$, $\xi^n$ converges in $L^2(\Omega;L^\infty([0,T];H^{-\varepsilon}(\mathbb T^d)))$  to the unique solution of
  \begin{equation}\label{limit-heat-eq}
  \partial_t \xi  = \kappa \Delta \xi , \quad \xi|_{t=0} =\xi_0,
  \end{equation}
where $\kappa=\frac{1}{2}C_d\int_{|z|\le1}z^2\,\nu(\mathrm{d}z)$ and $C_d=(d-1)/d$ is the constant in \eqref{eq-important-equality}.
\end{theorem}
This result will be proved in Section \ref{subs-scaling-SLTE}. We remark that we assume for simplicity the initial conditions of \eqref{STE-Marcus-n} are the same. A similar result also holds if one considers a sequence of initial data $\{\xi^n_0 \}_n \subset L^2(\T^d)$ such that $\xi^n_0$ converges weakly to some $\xi_0\in L^2(\T^d)$.

Before concluding this part, we give a probabilistic explanation for the appearance of the Laplacian operator in the limiting equation.

\begin{remark}
Fixing $x\in\T^d$, we consider  the stochastic process
\[
  L^n(t,x):=\sum_{k,i}\theta_k^n\sigma_{k,i}(x)Z_t^{k,i}.
\]
 Recall that $Z^{k,i}$ is the compensated small-jump process in \eqref{Z_t}, which is centered and has finite moments of every order. For $t>0$ and $\eta\in\mathbb R^d$ with $|\eta|=1$, \eqref{eq-important-equality} gives $\operatorname{Var}(\eta\cdot L^n(t,x))=2\kappa t$, and the Berry--Esseen theorem for independent triangular arrays (see \cite[Theorem~2.2.17]{Stroock11}) yields
\[
\begin{aligned}
&\sup_{y\in\mathbb R}\left|
\mathbb P\big(\eta\cdot L^n(t,x)\leq y\big)
-\mathbb P\big(\sqrt{2\kappa t}\,N\leq y\big)
\right|\\
&\quad\leq
C\frac{\mathbb E|Z_t|^3}{(2\kappa t)^{3/2}}
\sum_{k,i}|\theta_k^n|^3|\eta\cdot\sigma_{k,i}(x)|^3
\leq C_{d,t,\nu}\|\theta^n\|_{\ell^\infty}
\to0,
\end{aligned}
\]
where $Z_t$ is a copy of $Z_t^{k,i}$ and $N\sim\mathcal N(0,1)$. With more efforts, one can try to study the  L\'evy characteristics and show that $L^n(\cdot,x)\Rightarrow\sqrt{2\kappa}\,B$ in $D([0,T];\mathbb R^d)$, where $B$ is a standard $d$-dimensional Brownian motion. This one-point central-limit behavior partially explains the emergence of the eddy-viscosity operator $\kappa\Delta$.
\end{remark}

\subsubsection{Stochastic 2D Euler equations}

Now we consider the stochastic 2D Euler equation \eqref{stoch-2D-Euler} which, using the L\'evy noise in 2D, can be written more precisely as
\begin{equation}\label{stoch-2D-Euler-eq}
\begin{cases}
 \mathrm{d}\xi+ u \cdot \nabla \xi \, \mathrm{d}t + \sum_{k \in \mathbb{Z}^2_0} \theta_k\sigma_k  \cdot \nabla \xi \diamond \mathrm{d}Z^k_t = 0,\quad t\in[0,T], \\
\xi(0, \cdot) = \xi_0.
\end{cases}
\end{equation}
Recall that the velocity field $u$ is related to vorticity $\xi$ via the Biot-Savart law: $u = K\ast \xi= - \nabla^\perp (-\Delta)^{-1} \xi$, where $K$ is the Biot-Savart kernel on $\mathbb T^2$. We write $\Gamma=\{k\in \mathbb{Z}^2_0:\theta_k\neq0\}$ for the support of $\theta=\{\theta_k\}_k$ which is a finite set.

To give the meaning of weak solutions to the above equation, and also to fix notations in this setting, we follow the heuristic arguments of the last subsection and consider the following integral form:
\begin{equation}\label{2D-Euler-form-2}
\begin{aligned}
\xi(t) &= \xi_0  - \int_0^t u(s) \cdot \nabla \xi(s) \, \mathrm{d}s  + \sum_{k \in \Gamma} \int_0^t\! \int_{|z|\le 1} \big[{\rm e}^{-z \theta_k\sigma_k\cdot\nabla}\xi(s-) - \xi(s-)\big] \tilde{N}_k(\mathrm{d}z, \mathrm{d}s) \\
&\quad + \sum_{k \in \Gamma} \int_0^t\! \int_{|z|\le 1} \big[{\rm e}^{-z \theta_k\sigma_k\cdot\nabla}\xi(s) - \xi(s) + z \theta_k \sigma_k\cdot  \nabla \xi(s) \big] \,\nu(\mathrm{d}z)\mathrm{d}s.
\end{aligned}
\end{equation}
In dimension two, the flow $\varphi_t^k:=\varphi_t^{k,1}$ is defined by
\begin{equation}\label{flow}
\partial_t\varphi_t^k(x)
=\sigma_k\bigl(\varphi_t^k(x)\bigr),
\quad
\varphi_0^k(x)=x\in\T^2,
\quad t\in\mathbb R.
\end{equation}
By \eqref{define e}, the transport operator can be written
in terms of the inverse flow as
\begin{equation}\label{guanxi}
\quad
{\rm e}^{-z\theta_k\sigma_k\cdot\nabla}\xi(s)
=\xi\bigl(s,\varphi_{-z\theta_k}^k\bigr).
\end{equation}

We now define weak solutions to \eqref{stoch-2D-Euler-eq}
by testing \eqref{2D-Euler-form-2}  against smooth functions  and using \eqref{define e} and \eqref{guanxi}.

\begin{definition}\label{def-weak-solution}
Let $(\Omega,\mathcal{F}, (\mathcal{F}_t)_{t \geq 0}, \P)$ be a stochastic basis on which are defined a sequence of i.i.d. pure jump $(\mathcal{F}_t)$-L\'evy processes $\{Z^k_t\}_k$. We say that an $(\mathcal{F}_t)$-adapted c\`adl\`ag process $\xi$, with trajectories in $L^\infty([0,T];L^2(\mathbb{T}^2)) \cap D([0,T];H^{-\varepsilon}(\mathbb{T}^2))$, is a weak solution to \eqref{stoch-2D-Euler-eq} if $\P$-a.s. $\|\xi(t) \|_{L^2} \le \|\xi_0\|_{L^2}$, and for any $\phi\in C^\infty(\T^2)$ and $t\in [0,T]$, it holds
\begin{equation}\label{def-weak-solution.1}
\begin{aligned}
\langle\xi(t),\phi\rangle
&= \langle\xi_0,\phi\rangle
+\int_0^t \langle\xi(s),u(s)\cdot\nabla\phi\rangle\,\mathrm{d}s\\
&\quad +\sum_{k\in\Gamma}\int_0^t\!\int_{|z|\le1}
\bigl\langle \xi(s-), \phi\bigl(\varphi_{z\theta_k}^k\bigr)-\phi \bigr\rangle\,\tilde N_k(\mathrm{d}z,\mathrm{d}s)\\
&\quad +\sum_{k\in\Gamma}\int_0^t\!\int_{|z|\le1} \bigl\langle \xi(s), \phi\bigl(\varphi_{z\theta_k}^k\bigr)-\phi
-z\theta_k\sigma_k\cdot\nabla\phi
\bigr\rangle\,\nu(\mathrm{d}z)\,\mathrm{d}s.
\end{aligned}
\end{equation}
\end{definition}

Here is an existence result for \eqref{stoch-2D-Euler-eq} which will be proved in Section \ref{sect:see} by the classical methods of Galerkin approximation and compactness arguments, but we have to overcome the difficulties related to the Marcus mappings.

\begin{theorem}\label{thm:weak_solution}
 For any initial vorticity $\xi_0 \in L^2(\mathbb{T}^2)$, \eqref{stoch-2D-Euler-eq} admits at least one weak solution $\xi$ in the sense of Definition \ref{def-weak-solution}.
\end{theorem}

As in Subsection \ref{subsubsec-transport-eq}, we are concerned with the scaling limit of a sequence of stochastic 2D Euler equations, where the coefficients $\theta^n$ satisfy Hypotheses \ref{hypo-coefficients}. That is, we consider
\begin{equation}\label{yilie2D}
\begin{cases}
\mathrm{d}{\xi^n} + u^n \cdot \nabla \xi^n\, \mathrm{d}t + \sum_{k \in \Gamma_n} \theta_k^n \sigma_k \cdot \nabla \xi^n \diamond \mathrm{d}Z_t^k = 0, \\
\xi^n(0) = \xi_0.
\end{cases}
\end{equation}
By Theorem \ref{thm:weak_solution}, for each $n \geq 1$, there exists a weak solution $\xi^n$ to  \eqref{yilie2D} defined on some probability space $(\Omega,\mathcal{F},\mathbb{P})$ such that
\begin{equation}\label{estimate-2}
    \mathbb{P}\text{-a.s.}\quad \sup_{t \in [0,T]} \|\xi^n(t)\|_{L^2} \leq \|\xi_0\|_{L^2}.
\end{equation}
We remark that the solutions are possibly defined on different probability spaces, but for simplicity of notation, we will not distinguish $\Omega, \P$ and $\E$.

Now we can state the second main result of the paper.

\begin{theorem}\label{2d-scal}
Let $\{\theta^n\}_{n\ge1}$ satisfy Hypotheses~\ref{hypo-coefficients} and let $\xi^n$ be a weak solution to \eqref{yilie2D}. For any $\varepsilon>0$, let $Q^n$ denote the law of $\xi^n$.  Then $\{Q^n\}_{n\ge1}$ is tight in
$D\bigl([0,T];H^{-\varepsilon}(\mathbb T^2)\bigr)$ and converges weakly to $Q=\delta_\xi$, where $\xi$ is the unique solution to the 2D Navier-Stokes equation in vorticity form:
$$\partial_t\xi + u\cdot\nabla\xi= \kappa \Delta \xi, \quad \xi|_{t=0}=\xi_0,$$
where $\kappa$ is the constant as in \eqref{limit-heat-eq} in the case $d=2$.
\end{theorem}

We finish this section by describing the organization of the paper. We will make some preparations in Section \ref{sec-prelim} by recalling useful properties of the flow $\varphi^{k,i}_t$ and Sobolev spaces, as well as basic facts of the Skorohod space. Section 3 is devoted to proving Theorem \ref{main sca}, while Sections 4 and 5 contain the proofs of Theorems \ref{thm:weak_solution} and \ref{2d-scal}, respectively. Finally, we collect some useful materials in the appendix.

\section{Preliminaries}\label{sec-prelim}
This section consists of some preparatory materials. Section \ref{Marcus} presents  useful estimates for the flows defined in \eqref{characteristics} and \eqref{flow}, as well as two frequently used properties of Sobolev spaces. Then Section \ref{space D} recalls some basic facts about the Skorohod space $D([0,T];S)$, including the characterization of compact subsets.

\subsection{Flow estimates and Sobolev inequalities}
\label{Marcus}
\label{subsec-jump-estimates}
Let $\{\varphi^{k,i}_t\}_{t\in \R}$ be the flow of measure-preserving diffeomorphisms generated by \eqref{characteristics}. It verifies the following nice property:
\begin{equation}\label{important-property}
\sigma_{k,i}(\varphi^{k,i}_t(x))= \sigma_{k,i}(x), \quad  x\in \T^d,\, k\in \Z^d_0, i\in \{1,\ldots,d-1 \}.
\end{equation}
Indeed, recalling \eqref{ek} and \eqref{sigma_k,i}, we see that $\sigma_{k,i}(x) \cdot \nabla \sigma_{k,i}^j(x)=0$ holds for any $j=1,\ldots, d$ and $k\in \Z^d_0,\, i=1,\ldots, d-1$. Then,
$$\frac{\mathrm{d}}{\mathrm{d}t} \sigma_{k,i}(\varphi^{k,i}_t(x))= \frac{\mathrm{d}}{\mathrm{d}t}\varphi^{k,i}_t(x) \cdot (\nabla\sigma_{k,i}) (\varphi^{k,i}_t(x))= (\sigma_{k,i} \cdot \nabla\sigma_{k,i})(\varphi^{k,i}_t(x)) =0, $$
which leads to \eqref{important-property}. Thanks to \eqref{important-property}, one has
\begin{equation}\label{flow-expression}
\varphi^{k,i}_t(x) = x+ t\sigma_{k,i}(x), \quad x\in \T^d, k\in \Z^d_0, i\in \{1,\ldots, d-1\}.
\end{equation}
 In dimension two, we write $\varphi_t^k=\varphi_t^{k,1}$ and $\sigma_k=\sigma_{k,1}$, so that $\varphi_t^k(x)=x+t\sigma_k(x)$. We remark that \eqref{flow-expression} holds for more general shear flows.

 We now use \eqref{flow-expression} to derive some useful estimates.

\begin{lemma}\label{lem-flow-increments}
Let $\phi\in C^\infty(\mathbb T^d)$, $k\in\mathbb Z_0^d$, $i\in\{1,\ldots,d-1\}$ and $t\in\mathbb R$. Then
\begin{align}
\bigl\|\phi\bigl(\varphi_t^{k,i}\bigr)-\phi\bigr\|_{L^2}
&\le |t|\,\|\nabla\phi\|_{L^\infty},\label{prelim-first-order}\\
\bigl\|\phi\bigl(\varphi_t^{k,i}\bigr)-\phi-t\sigma_{k,i}\cdot\nabla\phi\bigr\|_{L^2}
&\le Ct^2\|\nabla^2\phi\|_{L^\infty}.\label{prelim-second-order}
\end{align}
The constant $C$ is independent of $k,i,t$ and $\phi$.
\end{lemma}
\begin{proof}
By \eqref{flow-expression} and the fundamental theorem of calculus,
\begin{equation}\label{flow-first-order-formula}
\phi\bigl(\varphi_t^{k,i}(x)\bigr)-\phi(x)=t\int_0^1\nabla\phi\bigl(\varphi_{rt}^{k,i}(x)\bigr)\cdot\sigma_{k,i}(x)\,\mathrm{d}r.
\end{equation}
By \eqref{ek} and \eqref{sigma_k,i}, $\|\sigma_{k,i}\|_{L^2}=1$ and $\|\sigma_{k,i}\|_{L^\infty}\le\sqrt2$. Taking the $L^2$-norm in \eqref{flow-first-order-formula} gives \eqref{prelim-first-order}. Taylor's formula with integral remainder also gives
\begin{equation}\label{2-taylor}
\phi\bigl(\varphi_t^{k,i}(x)\bigr)-\phi(x)-t\sigma_{k,i}(x)\cdot\nabla\phi(x)=t^2\int_0^1(1-r)\nabla^2\phi\bigl(\varphi_{rt}^{k,i}(x)\bigr):\bigl(\sigma_{k,i}(x)\otimes\sigma_{k,i}(x)\bigr)\,\mathrm{d}r,
\end{equation}
where $:$ denotes the inner product of matrices. Since
$\|\sigma_{k,i}\otimes\sigma_{k,i}\|_{L^2}
\le \|\sigma_{k,i}\|_{L^\infty}\|\sigma_{k,i}\|_{L^2}
\le \sqrt{2}$,
taking the $L^2$-norm in \eqref{2-taylor} proves
\eqref{prelim-second-order}.
\end{proof}
\begin{remark}
The same formulae give estimates in terms of the $L^2$-norms of the derivatives, which will be used in \eqref{R-n-prime}. Indeed, by \eqref{flow-first-order-formula} and \eqref{2-taylor}, using the measure-preserving property of $\varphi_{rt}^{k,i}$ and $\|\sigma_{k,i}\|_{L^\infty}\le\sqrt2$, we obtain
\begin{equation}\label{flow-L2-increments}
\begin{aligned}
\bigl\|\phi\bigl(\varphi_t^{k,i}\bigr)-\phi\bigr\|_{L^2}
&\le\sqrt2\,|t|\,\|\nabla\phi\|_{L^2},\\
\bigl\|\phi\bigl(\varphi_t^{k,i}\bigr)-\phi-t\sigma_{k,i}\cdot\nabla\phi\bigr\|_{L^2}
&\le t^2\|\nabla^2\phi\|_{L^2}.
\end{aligned}
\end{equation}
\end{remark}
By duality, the estimates in \eqref{flow-L2-increments} imply the following bounds in negative Sobolev norms.
\begin{corollary}
Let $\xi\in L^2(\mathbb T^d)$, $k\in\mathbb Z_0^d$,
$i\in\{1,\ldots,d-1\}$ and $t\in\mathbb R$. Then
\begin{equation}\label{flow-negative-Sobolev-increments}
\begin{aligned}
\bigl\|\xi\bigl(\varphi_{-t}^{k,i}\bigr)-\xi\bigr\|_{H^{-1}}
&\le C|t|\,\|\xi\|_{L^2},\\
\bigl\|\xi\bigl(\varphi_{-t}^{k,i}\bigr)-\xi
+t\sigma_{k,i}\cdot\nabla\xi\bigr\|_{H^{-2}}
&\le Ct^2\|\xi\|_{L^2},
\end{aligned}
\end{equation}
where $C$ depends only on $d$, and
$\sigma_{k,i}\cdot\nabla\xi$ is understood in the distributional sense.
\end{corollary}

\begin{proof}
For every $\phi\in C^\infty(\mathbb T^d)$,
$\langle\sigma_{k,i}\cdot\nabla\xi,\phi\rangle
=-\langle\xi,\sigma_{k,i}\cdot\nabla\phi\rangle$ since $\operatorname{div}\sigma_{k,i}=0$.
Together with \eqref{define e}, this gives
\begin{equation}\label{flow-duality-identities}
\begin{aligned}
\bigl\langle\xi\bigl(\varphi_{-t}^{k,i}\bigr)-\xi,\phi\bigr\rangle
&=\bigl\langle\xi,\phi\bigl(\varphi_t^{k,i}\bigr)-\phi\bigr\rangle,\\
\bigl\langle\xi\bigl(\varphi_{-t}^{k,i}\bigr)-\xi
+t\sigma_{k,i}\cdot\nabla\xi,\phi\bigr\rangle
&=\bigl\langle\xi,\phi\bigl(\varphi_t^{k,i}\bigr)-\phi
-t\sigma_{k,i}\cdot\nabla\phi\bigr\rangle.
\end{aligned}
\end{equation}
Applying the Cauchy--Schwarz inequality and
\eqref{flow-L2-increments}, we obtain
\[
\begin{aligned}
\bigl|\bigl\langle
\xi\bigl(\varphi_{-t}^{k,i}\bigr)-\xi,\phi
\bigr\rangle\bigr|
&\le C|t|\,\|\xi\|_{L^2}\|\phi\|_{H^1},\\
\bigl|\bigl\langle
\xi\bigl(\varphi_{-t}^{k,i}\bigr)-\xi
+t\sigma_{k,i}\cdot\nabla\xi,\phi
\bigr\rangle\bigr|
&\le Ct^2\|\xi\|_{L^2}\|\phi\|_{H^2}.
\end{aligned}
\]
Taking the respective suprema over smooth test functions with
$\|\phi\|_{H^1}\le1$ and $\|\phi\|_{H^2}\le1$
proves \eqref{flow-negative-Sobolev-increments}.
\end{proof}
We next use \eqref{prelim-first-order} and \eqref{prelim-second-order} with $t=z\theta_k$ to estimate sums over the noise indices.
\begin{corollary}\label{lem-paired-jump-estimates}
For $\xi\in L^2(\mathbb T^d)$ and $\phi\in C^\infty(\mathbb T^d)$, we have
\begin{align}
\sum_{k,i}\int_{|z|\le1}\bigl|\bigl\langle\xi,\phi\bigl(\varphi_{z\theta_k}^{k,i}\bigr)-\phi\bigr\rangle\bigr|^2\,\nu(\mathrm{d}z)
&\le \overline{C}_d\|\xi\|_{L^2}^2\|\nabla\phi\|_{L^\infty}^2\int_{|z|\le1}z^2\,\nu(\mathrm{d}z),\label{prelim-summed-noise}\\
\sum_{k,i}\int_{|z|\le1}\bigl|\bigl\langle\xi,\phi\bigl(\varphi_{z\theta_k}^{k,i}\bigr)-\phi-z\theta_k\sigma_{k,i}\cdot\nabla\phi\bigr\rangle\bigr|\,\nu(\mathrm{d}z)
&\le \overline{C}_d\|\xi\|_{L^2}\|\nabla^2\phi\|_{L^\infty}\int_{|z|\le1}z^2\,\nu(\mathrm{d}z).\label{prelim-summed-compensator}
\end{align}
The constant $\overline{C}_d$ depends only on $d$.
\end{corollary}
\begin{proof}
Apply the Cauchy--Schwarz inequality to each pairing, followed by \eqref{prelim-first-order} and \eqref{prelim-second-order} with $t=z\theta_k$, respectively. Integrating in $z$,  summing over $k,i$, and using $\sum_{k,i}\theta_k^2=d-1$, we obtain \eqref{prelim-summed-noise} and \eqref{prelim-summed-compensator}.
\end{proof}
The constants in Lemma~\ref{lem-flow-increments} and Corollary~\ref{lem-paired-jump-estimates}, as well as those in \eqref{flow-L2-increments}, are independent of $\theta$ and its support. Thus these estimates remain valid with $\theta^n$ in place of $\theta$ when evaluated at $t=z\theta_k^n$, with constants independent of $n$.

We finally recall two Sobolev inequalities. The first is an interpolation inequality.
\begin{lemma}\label{neicha}
Let $s,t\in\mathbb R$ and $\delta\in(0,1)$. For every $f\in H^s(\mathbb T^d)\cap H^t(\mathbb T^d)$,
\[
\|f\|_{H^{\delta s+(1-\delta)t}(\mathbb T^d)}\le\|f\|_{H^s(\mathbb T^d)}^\delta\|f\|_{H^t(\mathbb T^d)}^{1-\delta}.
\]
\end{lemma}
The following product estimate is the periodic analogue of
\cite[Corollary~2.55, p.~90]{fourier}. Here the space $H^{s_1+s_2-d/2}(\mathbb T^d)$ for the product $fg$ is not restricted to functions with zero mean.
\begin{lemma}\label{fenkai}
Let $s_1,s_2<d/2$ with $s_1+s_2>0$. For every $f\in H^{s_1}(\mathbb T^d)$ and $g\in H^{s_2}(\mathbb T^d)$, the product $fg$ belongs to $H^{s_1+s_2-d/2}(\mathbb T^d)$ and satisfies
\[
\|fg\|_{H^{s_1+s_2-d/2}(\mathbb T^d)}\le C\|f\|_{H^{s_1}(\mathbb T^d)}\|g\|_{H^{s_2}(\mathbb T^d)},
\]
where $C$ depends only on $d,s_1$ and $s_2$.
\end{lemma}

\subsection{Skorohod spaces $D([0,T];S)$}\label{space D} \label{subs-Skorohod-space}
This section focuses on the Skorohod space $D([0,T];S)$, where $S$ is a Polish space equipped with a metric $\rho$. To begin with, we introduce the basic settings in the Skorohod space $D([0,T];S)$, following main reference \cite{Dspace}.
Define $\Lambda'$ as the collection of strictly increasing continuous functions $\lambda: [0,T] \to [0,T]$ satisfying $\lambda(0) = 0$, $\lambda(T) = T$. Let $\Lambda$ denote the subset of Lipschitz continuous $\lambda \in \Lambda'$ such that
\begin{equation*}
\gamma(\lambda) := \underset{t\in [0,T]}{\mathrm{ess\,sup}}\, |\log \lambda'(t)|  = \sup_{s > t} \left| \log \frac{\lambda(s) - \lambda(t)}{s - t} \right| < \infty.
\end{equation*}
For $x, y \in D([0,T];S)$, the metric is given by
\[
d(x, y) = \inf_{\lambda \in \Lambda} \left[ \gamma(\lambda) \vee d(x, y, \lambda, T) \right],
\]
where
\[
d(x, y, \lambda, u) = \sup_{t \in [0,T]} \bigl[1\wedge\rho\bigl(x(t \wedge u), y(\lambda(t) \wedge u)\bigr)\bigr].
\]
The topology on \( D([0,T];S) \) induced by the metric \( d \) is called  Skorohod $J_1$ topology. The following proposition gives a characterization of convergence in $D([0,T];S)$ equipped with the Skorohod topology, as shown in \cite[Chapter 3]{Dspace}.
\begin{proposition}\label{D-con-0}
Let \( \{x_n\} \subset D([0,T];S) \) and \( x \in D([0,T];S) \). Then \( \lim_{n \to \infty} d(x_n, x) = 0 \) if and only if there exists \( \{\lambda_n\} \subset \Lambda \) such that $\lim_{n \to \infty} \gamma(\lambda_n) = 0$ holds and
\[
\lim_{n \to \infty} d(x_n, x, \lambda_n, u) = 0 \quad \text{for all continuity points } u \text{ of } x.
\]
In particular, \( \lim_{n \to \infty} d(x_n, x) = 0 \) implies that \( \lim_{n \to \infty} x_n(u) = x(u) \) for all continuity points \( u\in(0,T] \) of \( x \). In addition, $x_n(0)\to x(0)$ and $x_n(T)\to x(T)$.
\end{proposition}
The above characterization yields equivalent conditions for convergence in the Skorohod topology, see \cite[Proposition 5.3]{Dspace}.
\begin{proposition}\label{conD}
Let \( \{x_n\} \subset D([0,T];S) \) and \( x \in D([0,T];S) \), the following are equivalent:
\begin{enumerate}
\item[\rm (a)] \( \lim_{n \to \infty} d(x_n, x) = 0 \).
\item[\rm (b)] There exists \( \{\lambda_n\} \subset \Lambda \) such that
$\lim_{n \to \infty} \gamma(\lambda_n) = 0 $ holds and
\begin{equation}\label{r-converge}
\lim_{n \to \infty} \sup_{0 \leq t \leq T} \rho(x_n(t), x(\lambda_n(t))) = 0.
\end{equation}
\item[\rm (c)] There exists \( \{\lambda_n\} \subset \Lambda' \) such that $\lim_{n \to \infty} \sup_{0 \leq t \leq T} |\lambda_n(t) - t| = 0 $ and \eqref{r-converge} holds.
\end{enumerate}
\end{proposition}

It is worth mentioning that condition (c) is typically more convenient for verifying convergence of sequences in $D([0,T];S)$. Next we record a pathwise consequence of joint Skorohod convergence.

\begin{lemma}\label{lem-jump-matching}
Suppose that, as $n\to \infty$,
\[
(x_n,z_n)\to(x,z)\quad\text{in }D([0,T];S\times\mathbb R)
\]
with the joint $J_1$ topology, and let $\lambda_n\in\Lambda'$ be any sequence of time changes realizing this convergence. Fix $\delta>0$ such that $|\Delta z(s)|\ne\delta$ for every $s\in(0,T]$, where $\Delta z(s)=z(s)-z(s-)$. Denote by $\tau_1<\cdots<\tau_{J^\delta}$ and $\tau_1^n<\cdots<\tau_{J_n^\delta}^n$ the jump times in $(0,T]$ of $z$ and $z_n$, respectively, whose absolute jump sizes are greater than $\delta$. Then, for all sufficiently large $n$,
\begin{equation}\label{jump-matching-threshold}
|\Delta z_n(\lambda_n(s))|>\delta\Longleftrightarrow|\Delta z(s)|>\delta,\quad s\in(0,T],
\end{equation}
and consequently
\begin{equation}\label{jump-matching-times}
J_n^\delta=J^\delta,\quad \tau_j^n=\lambda_n(\tau_j),\quad j=1,\ldots,J^\delta.
\end{equation}
Moreover, for every $j=1,\ldots,J^\delta$, $\tau_j^n\to\tau_j$ and
\begin{equation}\label{jump-matching-values}
\Delta z_n(\tau_j^n)\to\Delta z(\tau_j),\quad \rho\bigl(x_n(\tau_j^n-),x(\tau_j-)\bigr)\to0.
\end{equation}
\end{lemma}

\begin{proof}
By Proposition~\ref{conD}-(c), applied to $S\times\mathbb R$ and with the time changes inverted if necessary, the common time changes satisfy $\sup_{t\in[0,T]}|\lambda_n(t)-t|\to0$ and
\[
\sup_{t\in[0,T]}\rho\bigl(x_n(\lambda_n(t)),x(t)\bigr)\to0,\quad \sup_{t\in[0,T]}|z_n(\lambda_n(t))-z(t)|\to0.
\]
Fix $n$ and $s\in(0,T]$. Since $\lambda_n$ is continuous and strictly increasing, $\lambda_n(r)\uparrow\lambda_n(s)$ as $r\uparrow s$. Taking left limits in the uniform estimate gives
\[
|z_n(\lambda_n(s)-)-z(s-)|\le\sup_{u\in[0,T]}|z_n(\lambda_n(u))-z(u)|.
\]
Subtracting the left limits from the values at the corresponding times, we obtain
\begin{equation}\label{jump-matching-error}
\sup_{0<s\le T}|\Delta z_n(\lambda_n(s))-\Delta z(s)|\le2\sup_{u\in[0,T]}|z_n(\lambda_n(u))-z(u)|\to0.
\end{equation}

We first claim that, for all sufficiently large $n$,
\begin{equation}\label{jump-matching-size-gap}
\bigl||\Delta z_n(\lambda_n(s))|-|\Delta z(s)|\bigr|<\bigl||\Delta z(s)|-\delta\bigr|,\quad s\in(0,T].
\end{equation}
To prove \eqref{jump-matching-size-gap}, we first observe that
\[
\inf_{0<s\le T}\bigl||\Delta z(s)|-\delta\bigr|>0.
\]
Indeed, if this infimum were zero, the assumption $|\Delta z(s)|\ne\delta$ would imply the existence of infinitely many distinct jumps whose absolute sizes converge to $\delta>0$. This contradicts the fact that a c\`adl\`ag path on $[0,T]$ has only finitely many jumps whose absolute sizes exceed any prescribed positive constant.
By the reverse triangle inequality and \eqref{jump-matching-error},
\[
\sup_{0<s\le T}\bigl||\Delta z_n(\lambda_n(s))|-|\Delta z(s)|\bigr|\le\sup_{0<s\le T}|\Delta z_n(\lambda_n(s))-\Delta z(s)|\to0.
\]
Consequently, for all sufficiently large $n$,
\[
\sup_{0<s\le T}\bigl||\Delta z_n(\lambda_n(s))|-|\Delta z(s)|\bigr|<\inf_{0<s\le T}\bigl||\Delta z(s)|-\delta\bigr|.
\]
In particular, \eqref{jump-matching-size-gap} holds for every $s\in(0,T]$.

If $|\Delta z(s)|>\delta$, \eqref{jump-matching-size-gap} gives
\[
|\Delta z_n(\lambda_n(s))|>|\Delta z(s)|-\bigl(|\Delta z(s)|-\delta\bigr)=\delta.
\]
If $|\Delta z(s)|<\delta$, \eqref{jump-matching-size-gap} gives
\[
|\Delta z_n(\lambda_n(s))|<|\Delta z(s)|+\bigl(\delta-|\Delta z(s)|\bigr)=\delta.
\]
Since $|\Delta z(s)|\ne\delta$ for every $s\in(0,T]$, these two cases prove \eqref{jump-matching-threshold}. As $\lambda_n$ is a strictly increasing bijection, the jump times of $z_n$ with absolute jump sizes greater than $\delta$ are exactly $\lambda_n(\tau_1),\ldots,\lambda_n(\tau_{J^\delta})$, in increasing order. This proves \eqref{jump-matching-times}, including the case $J^\delta=0$.

It follows that $|\tau_j^n-\tau_j|\le\sup_{t\in[0,T]}|\lambda_n(t)-t|\to0$, and \eqref{jump-matching-error} gives the convergence of the corresponding jump sizes. Finally, applying the same left-limit argument to $x_n$ and using \eqref{jump-matching-times}, we obtain
\[
\rho\bigl(x_n(\tau_j^n-),x(\tau_j-)\bigr)\le\sup_{u\in[0,T]}\rho\bigl(x_n(\lambda_n(u)),x(u)\bigr)\to0,
\]
which completes the proof.
\end{proof}

The following proposition collects two more consequences of convergence in Skorohod spaces.

\begin{proposition}\label{Skorohod-interpolation}
\begin{enumerate}
\item[\rm (i)] Let $E$ be a separable Banach space. If $x_n\to x$ in $D([0,T];E)$, then
\[
\sup_n\sup_{t\in[0,T]}\|x_n(t)\|_E<\infty,\quad x_n\to x\text{ in }L^p([0,T];E),\quad 1\le p<\infty.
\]
\item[\rm (ii)] Let $\alpha\in(0,1)$ and $E_0\hookrightarrow E_\alpha\hookrightarrow E_1$ be continuous embeddings of separable Banach spaces, with $E_0$ reflexive. Assume that, for some $C>0$,
\begin{equation}\label{skorohod-interpolation-inequality}
\|v\|_{E_\alpha}
\le C\|v\|_{E_0}^{1-\alpha}\|v\|_{E_1}^{\alpha},
\quad \text{for every } v\in E_0.
\end{equation}
If $\sup_{n\ge1}\sup_{t\in[0,T]}\|x_n(t)\|_{E_0}\le M<\infty$ and $x_n\to x$ in $D([0,T];E_1)$, then $x(t)\in E_0$ for every $t\in[0,T]$,
$\sup_{t\in[0,T]}\|x(t)\|_{E_0}\le M$, and
\[
x_n\to x\quad\text{in }
D([0,T];E_\alpha)\cap L^p([0,T];E_\alpha),
\quad 1\le p<\infty.
\]
\end{enumerate}
\end{proposition}
\begin{proof}
For (i), choose $\lambda_n\in\Lambda'$ as in condition (c) in Proposition \ref{conD}. Since every c\`adl\`ag path is bounded on $[0,T]$, \eqref{r-converge} gives
\[
\begin{aligned}
\sup_n\sup_{t\le T}\|x_n(t)\|_E
\le\sup_{t\le T}\|x(t)\|_E+\sup_n\sup_{t\le T}\|x_n(t)-x(\lambda_n(t))\|_E<\infty.
\end{aligned}
\]
Since $x$ has at most countably many discontinuities, Proposition \ref{D-con-0} implies $$x_n(t)\to x(t)\quad\text{in }E\text{ for a.e. }t\in[0,T];$$
by the dominated convergence theorem,
\[
\int_0^T\|x_n(t)-x(t)\|_E^p\,\mathrm{d}t\to0,\quad 1\le p<\infty.
\]
For (ii), choose $\lambda_n\in\Lambda'$ as in condition (c) for the convergence in $D([0,T];E_1)$. For every $s\in[0,T]$,
\[
\begin{gathered}
x_n\bigl(\lambda_n^{-1}(s)\bigr)\to x(s)\quad\text{in }E_1,\quad
\sup_n\bigl\|x_n\bigl(\lambda_n^{-1}(s)\bigr)\bigr\|_{E_0}\le M.
\end{gathered}
\]
By reflexivity of $E_0$ and uniqueness of the limit in $E_1$, $x(s)\in E_0$ and
\[
\begin{aligned}
x_n\bigl(\lambda_n^{-1}(s)\bigr)\rightharpoonup x(s)\quad\text{in }E_0,\quad
\|x(s)\|_{E_0}\le\liminf_{n\to\infty}\bigl\|x_n\bigl(\lambda_n^{-1}(s)\bigr)\bigr\|_{E_0}\le M.
\end{aligned}
\]
By \eqref{skorohod-interpolation-inequality},
\[
\|x(s)-x(t)\|_{E_\alpha}
\le C(2M)^{1-\alpha}\|x(s)-x(t)\|_{E_1}^{\alpha},
\quad s,t\in[0,T].
\]
Since $x\in D([0,T];E_1)$ and $E_\alpha$ is complete,
this estimate implies $x\in D([0,T];E_\alpha)$.
The same argument applies to each $x_n$.
Moreover,
\[
\begin{aligned}
\sup_{t\le T}\|x_n(t)-x(\lambda_n(t))\|_{E_\alpha}
&\le C(2M)^{1-\alpha}\sup_{t\le T}\|x_n(t)-x(\lambda_n(t))\|_{E_1}^{\alpha}\to0.
\end{aligned}
\]
This proves convergence in $D([0,T];E_\alpha)$; the $L^p([0,T];E_\alpha)$ convergence follows from (i).
\end{proof}

Now we return to general case of Polish space $(S,\rho)$.
Our analysis of tightness in $D([0,T];S)$ begins with the concept of modulus of continuity, which serves as a key tool for characterizing compact subsets.
\begin{definition}
Let $f \in D([0,T];S)$ and let $\delta > 0$ be given. A modulus of $f$ is defined by
\[
\omega_{[0,T],S}(f, \delta) := \inf_{\Pi_{\delta}} \max_{i} \sup_{t_i \leq t < s < t_{i+1} \leq T} \rho(f(t), f(s)),
\]
where $\Pi_{\delta}$ is the set of all increasing sequences $ \{0 = t_0 < t_1 < \dots < t_n = T\}$ with the property $t_{i+1} - t_i \geq \delta$, $i = 0,1,\dots, n - 1$.
\end{definition}
Building upon this definition, we now state a complete characterization of relatively compact sets in the Skorohod topology, see \cite[Lemma 7.2]{Metivier}.
\begin{lemma}\label{compact1}
A set $\mathcal{K} \subset D([0,T];S)$ has compact closure if and only if it satisfies the following conditions:
\begin{enumerate}
\item[\rm (1)] There exists a dense subset $I \subset [0,T]$ containing $0$ and $T$ such that for every $t \in I$ the set $\{f(t): f \in \mathcal{K}\}$ has compact closure in $S$;
\item[\rm (2)] $\lim_{\delta \to 0} \sup_{f \in \mathcal{K}} \omega_{[0,T],S}(f, \delta) = 0$.
\end{enumerate}
\end{lemma}
While Lemma \ref{compact1} provides a theoretical basis for proving tightness of stochastic processes with jumps, we also need the following Aldous condition introduced in \cite{Adlous} to control the modulus.
\begin{definition}
We say a sequence of $S$-valued c\`adl\`ag $(\mathcal{F}_t)$-adapted stochastic processes $(X_n)_{n \in \mathbb{N}}$ satisfies the \textbf{Aldous condition} if, for every $\varepsilon > 0$,
\[
\lim_{\delta\downarrow0}\sup_{n\in\mathbb N}\sup_{0\le\theta\le\delta\wedge T}\sup_{\tau_n+\theta\le T}
\mathbb{P}\Bigl\{\rho\bigl(X_n(\tau_n+\theta),X_n(\tau_n)\bigr)\ge\varepsilon\Bigr\}=0,
\]
where $\theta$ is deterministic and $\tau_n$ ranges over all nonnegative $(\mathcal{F}_t)$-stopping times.
\end{definition}
To verify the Aldous condition in practice, we present the following sufficient condition adapted from \cite[Lemma 9]{NS}.
\begin{lemma}\label{adlous condition1}
Let $(X_n)_{n \in \mathbb{N}}$ be a sequence of $S$-valued c\`adl\`ag $(\mathcal{F}_t)$-adapted stochastic processes. The sequence $(X_n)_{n \in \mathbb{N}}$ satisfies the Aldous condition in the space $S$ if there exist constants $\alpha, \beta > 0$ and $C > 0$ such that for all $n \in \mathbb{N}$, all $\theta \geq 0$, and all sequences $(\tau_n)_{n \in \mathbb{N}}$ of $\mathcal{F}_t$-stopping times with $\tau_n + \theta \leq T$, it holds
\[
\mathbb{E} \big[ \rho( X_n(\tau_n + \theta), X_n(\tau_n) )^{\alpha} \big] \leq C \theta^{\beta}.
\]
\end{lemma}
The following lemma (cf. \cite[Lemma 7]{NS}) captures the essential connection between the Aldous condition and the modulus of continuity, providing the key to tightness for c\`adl\`ag processes in Section~\ref{com-tight}.
\begin{lemma}\label{compact-Aldous-1}
Assume that a sequence of $S$-valued c\`adl\`ag stochastic processes $(X_n)_{n \in \mathbb{N}}$ satisfies the Aldous condition. Let $\mathbb{P}_n$ be the law of $X_n$ on $D([0,T];S)$, $n\in \mathbb{N}$. Then for every $\varepsilon > 0$ there exists a subset $A_{\varepsilon} \subset D([0,T];S)$ such that
$$\inf_{n \in \mathbb{N}} \mathbb{P}_n(A_{\varepsilon}) \geq 1 - \varepsilon\quad
\text{and}\quad\lim_{\delta \to 0} \sup_{f \in A_{\varepsilon}} \omega_{[0,T],S}(f, \delta) = 0.$$
\end{lemma}

\section{Stochastic transport equations driven by L\'evy noises}\label{sect:ste}

In this part we consider the stochastic transport equation \eqref{STE-Marcus} understood in the Marcus sense. We shall first provide an existence theorem for $L^2(\mathbb T^d)$ initial data, then prove the scaling limit result stated in Theorem~\ref{main sca}.

We use the weak formulation \eqref{STE-Marcus-2} and first verify that all its terms are well-defined. Fix $\phi\in C^\infty(\mathbb T^d)$ and assume that $\xi\in L^2\bigl(\Omega;L^2([0,T];L^2(\mathbb T^d))\bigr)$ is an $(\mathcal F_t)$-adapted c\`adl\`ag $L^2(\mathbb T^d)$-valued process. Then $\mathbb P$-a.s. $\xi_-=\{\xi(s-)\}_{s\in(0,T]}$ is predictable, and $\xi(s-)=\xi(s)$ for a.e. $s$.  Using the compensated-Poisson isometry (see \cite[Theorem~1.14]{Eberle15}), orthogonality of the integrals for independent noises, and \eqref{prelim-summed-noise}, we obtain
\begin{equation}\label{I1}
\begin{aligned}
&\mathbb E\biggl|\sum_{k,i}\int_0^t\!\int_{|z|\le1}
\bigl\langle\xi(s-),\phi\bigl(\varphi_{z\theta_k}^{k,i}\bigr)-\phi\bigr\rangle\,\tilde N_{k,i}(\mathrm{d}z,\mathrm{d}s)\biggr|^2\\
&\quad=\mathbb E\int_0^t\sum_{k,i}\int_{|z|\le1}
\bigl|\bigl\langle\xi(s),\phi\bigl(\varphi_{z\theta_k}^{k,i}\bigr)-\phi\bigr\rangle\bigr|^2\,\nu(\mathrm{d}z)\,\mathrm{d}s\\
&\quad\le  \overline{C}_d\|\nabla\phi\|_{L^\infty}^2
\biggl(\int_{|z|\le1}z^2\,\nu(\mathrm{d}z)\biggr)
\mathbb E\int_0^T\|\xi(s)\|_{L^2}^2\,\mathrm{d}s<\infty.
\end{aligned}
\end{equation}
Thus the stochastic term is a square-integrable martingale. Similarly, by \eqref{prelim-summed-compensator} and the Cauchy--Schwarz inequality,
\[
\begin{aligned}
&\mathbb E\sum_{k,i}\int_0^T\!\int_{|z|\le1}
\bigl|\bigl\langle\xi(s),\phi\bigl(\varphi_{z\theta_k}^{k,i}\bigr)-\phi-z\theta_k\sigma_{k,i}\cdot\nabla\phi\bigr\rangle\bigr|\,\nu(\mathrm{d}z)\,\mathrm{d}s\\
&\quad\le \overline{C}_d\|\nabla^2\phi\|_{L^\infty}
\biggl(\int_{|z|\le1}z^2\,\nu(\mathrm{d}z)\biggr)
\mathbb E\int_0^T\|\xi(s)\|_{L^2}\,\mathrm{d}s\\
&\quad\le \overline{C}_d\sqrt T\,\|\nabla^2\phi\|_{L^\infty}
\biggl(\int_{|z|\le1}z^2\,\nu(\mathrm{d}z)\biggr)
\biggl(\mathbb E\int_0^T\|\xi(s)\|_{L^2}^2\,\mathrm{d}s\biggr)^{1/2}<\infty.
\end{aligned}
\]
Hence the last term in \eqref{STE-Marcus-2} is absolutely integrable, and all terms in the weak formulation are well-defined.

Now we can give the definition of solutions to \eqref{STE-Marcus-2}.

\begin{definition}\label{def1}
    Let $(\Omega,\mathcal{F},\mathcal{F}_t,\mathbb{P})$ be the filtered probability space defined as in Section \ref{subs-struc-noise}. We say that \eqref{STE-Marcus-2} has a probabilistically strong solution if there exists an $(\mathcal{F}_t)$-adapted c\`adl\`ag $L^2(\mathbb T^d)$-valued process $\xi\in L^2\big(\Omega,L^2\big([0,T];L^2(\mathbb{T}^d)\big)\big)$  such that for any $\phi\in C^\infty(\mathbb{T}^d)$, the equality \eqref{STE-Marcus-2} holds $\mathbb{P}$-a.s. for all $t\in[0,T]$.
\end{definition}

Note that the solution is strong in the probabilistic sense but weak in the analytic sense.

\subsection{Existence of solutions}\label{subs-STE-existence}

This section is devoted to proving the existence of solutions
to \eqref{STE-Marcus-2} in the sense of Definition~\ref{def1}.
Throughout this subsection, we fix a coefficient sequence $\theta$
with finite support
$\Gamma:=\{k\in\mathbb Z_0^d:\theta_k\neq0\}$.
All noise sums below are over $k\in\Gamma$ and
$i\in\{1,\ldots,d-1\}$.
We first establish the properties of the stochastic flow by applying
\cite[Theorem~2.1]{LM15}; see also \cite{BP23, BPZZ25}
for other recent works on the existence of stochastic flows.

\begin{proposition}\label{prop-linear-flow}
The Marcus equation
\begin{equation}\label{Marcus flow}
\mathrm{d}X_t(x)
=
\sum_{k\in\Gamma,i}
\theta_k\sigma_{k,i}\bigl(X_{t-}(x)\bigr)
\diamond\mathrm{d}Z_t^{k,i},
\quad X_0(x)=x\in\T^d,
\end{equation}
generates an $(\mathcal F_t)$-adapted stochastic flow $X$.
There exists an event $\Omega_0\in\mathcal F$ with
$\mathbb P(\Omega_0)=1$ such that the following properties
hold on $\Omega_0$:
\begin{enumerate}
\item[\textnormal{(i)}]
$X_t$ is a $C^2$-diffeomorphism of $\T^d$
for every $t\in[0,T]$.

\item[\textnormal{(ii)}]
The flow and its inverse satisfy
\[
X_\cdot,\ X_\cdot^{-1}
\in D\bigl([0,T];C(\T^d;\T^d)\bigr),
\]
where $C(\T^d;\T^d)$ is equipped with the uniform metric
$d_\infty(f,g)
:=\sup_{x\in\T^d}
d_{\T^d}\bigl(f(x),g(x)\bigr),$
and $d_{\T^d}$ denotes the usual distance on $\T^d$.
\item[\textnormal{(iii)}]
$X_t$ and $X_t^{-1}$ preserve Lebesgue measure
for every $t\in[0,T]$.
\end{enumerate}
\end{proposition}

\begin{proof}
\textbf{Step 1. Existence and regularity of the flow.} Since the independent driving processes $\mathbb P$-a.s. have no simultaneous jumps  and \eqref{flow-expression} gives
$\varphi_{z\theta_k}^{k,i}(x)-x-z\theta_k\sigma_{k,i}(x)=0,$
the Marcus equation \eqref{Marcus flow} is equivalent to

\begin{equation}\label{linear-characteristic-ito}
X_t(x)=x+\sum_{k\in\Gamma,i}\int_0^t\!\int_{|z|\le1}
z\theta_k\sigma_{k,i}\bigl(X_{s-}(x)\bigr)\,\tilde N_{k,i}(\mathrm{d}z,\mathrm{d}s).
\end{equation}

We first check that \eqref{linear-characteristic-ito} satisfies \cite[Assumption~2.1]{LM15}. Note that we work on $\mathbb T^d$ and use the same notation for periodic lifts when needed.
Combine the independent Poisson random measures on $\Gamma\times\{1,\ldots,d-1\}\times\{|z|\le1\}$, with intensity $\pi(\{(k,i)\}\times\mathrm{d}z)=\nu(\mathrm{d}z)$. The drift and diffusion coefficients are zero, and $H_t(x,(k,i,z))=z\theta_k\sigma_{k,i}(x)$. By smooth periodicity and the definitions \eqref{ek} and \eqref{sigma_k,i}, the spatial bounds in \cite[Assumption~2.1-(1)]{LM15}, with $\beta=3$, hold with

\[
K_t(k,i,z):=C|k|^3|z\theta_k|,
\]
where $C$ depends only on $d$. Finiteness of $\Gamma$ and the Lévy measure condition give
\[
\begin{aligned}
\sup_{\substack{k\in\Gamma,\ 1\le i\le d-1\\|z|\le1}}
K_t(k,i,z)
&\le C\max_{k\in\Gamma}|k|^3|\theta_k|<\infty,\\
\sum_{k\in\Gamma}\sum_{i=1}^{d-1}
\int_{|z|\le1}K_t(k,i,z)^2\,\nu(\mathrm{d}z)
&=C^2(d-1)\sum_{k\in\Gamma}|k|^6\theta_k^2
\int_{|z|\le1}z^2\,\nu(\mathrm{d}z)<\infty.
\end{aligned}
\]
Thus \cite[Assumption~2.1-(1)]{LM15} holds. Moreover, $D\sigma_{k,i}=a_{k,i}\otimes\nabla e_k$ and $a_{k,i}\cdot\nabla e_k=0$ imply

\begin{equation}\label{linear-jump-inverse}
\begin{aligned}
\bigl(I_d+z\theta_kD\sigma_{k,i}(x)\bigr)^{-1}
=I_d-z\theta_kD\sigma_{k,i}(x),\quad
\det\bigl(I_d+z\theta_kD\sigma_{k,i}(x)\bigr)
=1.
\end{aligned}
\end{equation}
Consequently,
\[
\begin{aligned}
&\sup_{\substack{k\in\Gamma,\ 1\le i\le d-1\\
|z|\le1,\ x\in\mathbb R^d}}
\bigl\|\bigl(I_d+z\theta_kD\sigma_{k,i}(x)\bigr)^{-1}\bigr\|\le
1+\max_{\substack{k\in\Gamma\\1\le i\le d-1}}
|\theta_k|\,\|D\sigma_{k,i}\|_{L^\infty(\mathbb R^d)}
<\infty,
\end{aligned}
\]
which verifies \cite[Assumption~2.1-(2)]{LM15}.

By \cite[Theorem~2.1-(1)]{LM15}, with $(\tau,\beta,\beta')=(0,3,5/2)$, we may choose a modification of the strong solution and an event $\Omega_1\in\mathcal F$ with $\mathbb P(\Omega_1)=1$ such that, on $\Omega_1$, $X_t$ is a $C_{\mathrm{loc}}^{5/2}$-diffeomorphism of $\mathbb R^d$ for every $t\in[0,T]$, and
\begin{equation}\label{linear-flow-cadlag-Rd}
\begin{gathered}
X_\cdot,\ X_\cdot^{-1}
\in D\bigl([0,T];C_{\mathrm{loc}}^{5/2}(\mathbb R^d;\mathbb R^d)\bigr).
\end{gathered}
\end{equation}

By periodicity of the coefficients and pathwise uniqueness,
there exists an event $\Omega_a\subset\Omega_1$ of probability
one on which $X_t(x+q)=X_t(x)+q$ for all $t\in[0,T]$,
$x\in\mathbb R^d$ and $q\in\mathbb Z^d$.
Thus, with the same notation, $X_t$ is a $C^2$-diffeomorphism of $\mathbb T^d$ with inverse $X_t^{-1}$. By \eqref{linear-flow-cadlag-Rd}, on $\Omega_a$,
\begin{equation}\label{linear-flow-cadlag-torus}
X_\cdot,\ X_\cdot^{-1}
\in D\bigl([0,T];C(\mathbb T^d;\mathbb T^d)\bigr),
\end{equation}
where $C(\mathbb T^d;\mathbb T^d)$ is equipped with the
uniform metric. So we have shown assertions (i) and (ii) hold.

\textbf{Step 2. Preservation of Lebesgue measure.}
It remains to prove (iii) for which we use the Wong--Zakai approximation established in \cite[Theorem~6.5]{KPP95}. For $n\ge1$, define
\[
Z_t^{k,i,n}:=n\int_{t-1/n}^t Z_r^{k,i}\,\mathrm dr,\quad t\in[0,T],
\]
where $Z_r^{k,i}=0$ for $r<0$. Consider the ordinary differential equation
\[
\mathrm dX_t^n(x)
= \sum_{k\in\Gamma,i}
\theta_k\sigma_{k,i}\bigl(X_t^n(x)\bigr)
\,\mathrm dZ_t^{k,i,n},
\quad X_0^n(x)=x,\quad 0\le t\le T.
\]
Each path $t\mapsto Z_t^{k,i,n}$ is absolutely continuous. Since $\{\sigma_{k,i} \}_{k\in \Gamma,i}$ are divergence-free, the  field $x\mapsto\sum_{k\in\Gamma,i}\theta_k\sigma_{k,i}(x) \frac{\mathrm dZ_t^{k,i,n}}{\mathrm dt}$ is divergence-free for almost every $t\in[0,T]$. Hence $X_t^n$ preserves Lebesgue measure for every $t\in[0,T]$. Moreover, by \cite[Theorem~6.5]{KPP95}, for almost every $t\in[0,T]$,
\[
X_t^n(x)\to X_t(x)\quad\text{in probability}
\]
for all $x\in\mathbb T^d$. Note that the exceptional set $\mathcal N_T$ of time is independent of $x$.

Fix $t\in [0,T]\setminus \mathcal N_T$ and $\phi\in C(\mathbb T^d)$. By continuity and boundedness of $\phi$,
\[
\mathbb E\bigl|\phi(X_t^n(x))-\phi(X_t(x))\bigr|\to0,\quad
\mathbb E\bigl|\phi(X_t^n(x))-\phi(X_t(x))\bigr|
\le2\|\phi\|_\infty.
\]
Since $X_t^n$ preserves Lebesgue measure,  dominated convergence yields
\[
\begin{aligned}
\mathbb E\left|
\int_{\mathbb T^d}\phi\bigl(X_t(x)\bigr)\,\mathrm dx
-\int_{\mathbb T^d}\phi(x)\,\mathrm dx
\right|
&=\mathbb E\left|
\int_{\mathbb T^d}
\bigl[\phi(X_t(x))-\phi(X_t^n(x))\bigr]\,\mathrm dx
\right|\\
&\le\int_{\mathbb T^d}
\mathbb E\bigl|\phi(X_t(x))-\phi(X_t^n(x))\bigr|\,\mathrm dx
\to0.
\end{aligned}
\]
The left-hand side is independent of $n$, so it is zero. Therefore, for every fixed $t\in [0,T]\setminus \mathcal N_T$ and $\phi$,
\begin{equation}\label{linear-flow-measure}
\mathbb P\text{-a.s.}\quad
\int_{\mathbb T^d}\phi\bigl(X_t(x)\bigr)\,\mathrm dx
=
\int_{\mathbb T^d}\phi(x)\,\mathrm dx.
\end{equation}
Next, by right-continuity of paths of $X_t(x)$, we can extend the above identity to all $t\in [0,T]$.

Choose a countable uniformly dense family $\{\phi_m\}_{m\ge1}\subset C(\mathbb T^d)$ and define
$\{\phi_m\}_{m\ge1}\subset C(\mathbb T^d)$ and set
\[
\Omega_0:=\Omega_a\cap
\bigcap_{\substack{m\ge1}}
\bigg\{
\int_{\mathbb T^d}\phi_m(X_q(x))\,\mathrm dx
=
\int_{\mathbb T^d}\phi_m(x)\,\mathrm dx, \quad\text{for} \,\, q\in\mathbb Q\cap[0,T],\quad\forall t\in [0,T]
\bigg\}.
\]
Then $\mathbb P(\Omega_0)=1$. By density and right continuity, \eqref{linear-flow-measure} holds on $\Omega_0$ for every $t\in[0,T]$ and $\phi\in C(\mathbb T^d)$. Taking $\phi\circ X_t^{-1}$ as a test function shows that $X_t^{-1}$ also preserves Lebesgue measure.
\end{proof}

Inspired by \cite{HP23} and using Proposition~\ref{prop-linear-flow}, we obtain the following existence result.

\begin{theorem}\label{cunzaixing}
For any $\xi_0\in L^2(\mathbb{T}^d)$, there exists at least one probabilistically strong solution to \eqref{STE-Marcus} with trajectories in $ D([0,T];L^2(\mathbb T^d))$ which is given explicitly by
$$ \xi(t) = \xi_0\big(X_t^{-1}\big), $$
where $X_t: \mathbb{T}^d \to \mathbb{T}^d$ is the stochastic flow generated by \eqref{Marcus flow}. Moreover, $\xi$ preserves the $L^2(\mathbb{T}^d)$-norm:  $\|\xi(t)\|_{L^2} = \|\xi_0\|_{L^2}$ $\mathbb{P}$-a.s. for all $t \in [0,T]$.
\end{theorem}

Recalling that the noise is assumed to be smooth in space variable, it is possible to prove the pathwise uniqueness of \eqref{STE-Marcus} with some additional efforts. We do not want to pursue this here since the uniqueness of solutions is not needed in the scaling limit result of Theorem \ref{main sca}.
\begin{proof}
According to \cite[Theorem 2.1]{HP23}, the result holds when $\xi_0 \in C^2(\mathbb{T}^d)$. We extend this to general $\xi_0\in L^2(\mathbb{T}^d)$ through approximation. Choose a Borel representative of $\xi_0$ and set $\xi(t)=\xi_0\circ X_t^{-1}$.
Since $C^2(\mathbb{T}^d)$ is dense in $L^2(\mathbb{T}^d)$, there exists a sequence $\{\xi_{0,n}\} \subset C^2(\mathbb{T}^d)$ with $\xi_{0,n} \to \xi_0$ in $L^2(\mathbb{T}^d)$. Consider the approximating equations
\[
\mathrm{d}\xi_n + \sum_{k\in\Gamma,i} \theta_k \sigma_{k,i} \cdot \nabla \xi_n \diamond \mathrm{d}Z_t^{k,i}=0, \quad \xi_n(0) = \xi_{0,n}.
\]
By \cite[Theorem 2.1]{HP23}, each equation admits a strong solution $\xi_n(t) = \xi_{0,n}\circ X_t^{-1}$.  The inverse flow and its measure-preserving property have been verified in Proposition \ref{prop-linear-flow}. This property ensures that $\mathbb{P}$-a.s.,
\[
\big\|\xi_n(t) - \xi(t)\big\|_{L^2}^2 = \int_{\mathbb{T}^d} \big|\xi_{0,n}(X_t^{-1}(x)) - \xi_0(X_t^{-1}(x))\big|^2 \mathrm{d}x = \int_{\mathbb{T}^d}\big|\xi_{0,n}(y) - \xi_0(y)\big|^2 \mathrm{d}y \to 0.
\]
Therefore, $\xi_n(t) \to \xi(t)$ in $L^2(\mathbb{T}^d)$ for each $t \in [0,T]$. In fact,
\[
\sup_{t\in[0,T]}\|\xi_n(t)-\xi(t)\|_{L^2}=\|\xi_{0,n}-\xi_0\|_{L^2}\to0.
\]
By \eqref{linear-flow-cadlag-torus}, each $\xi_n$ has c\`adl\`ag $L^2$-paths. The uniform limit $\xi$ is therefore c\`adl\`ag as well. The measure-preserving property of $X_t$ guarantees that the $L^2(\mathbb{T}^d)$ norm remains invariant: $\P$-a.s. for all $t\in [0,T]$,
\begin{equation}\label{xi-L2}
\|\xi(t)\|_{L^2} = \|\xi_0(X_t^{-1})\|_{L^2} = \|\xi_0\|_{L^2}<\infty,
\end{equation}
which implies that $\mathbb{P}$-a.s. $\xi \in L^\infty([0,T]; L^2(\mathbb{T}^d))$. The composition is independent of the Borel representative of $\xi_0$, and its spatial mean remains zero, because $X_t^{-1}$ preserves measure. Since each $\xi_n$ is $(\mathcal F_t)$-adapted and, $\mathbb P$-a.s., $\xi_n(t)\to\xi(t)$ in $L^2(\mathbb T^d)$ for every $t\in[0,T]$, completeness of the filtration implies that $\xi$ is also adapted.
Since $\xi$ has c\`adl\`ag $L^2$-paths, $\xi_-=\{\xi(s-)\}_{s\in(0,T]}$ is predictable.

We now verify that $\xi(t) = \xi_0(X_t^{-1})$ satisfies the weak formulation \eqref{STE-Marcus-2}. Each approximate solution $\xi_n$ satisfies
\begin{equation}\label{sub-STE-weak}
\begin{aligned}
\langle\xi_n(t),\phi\rangle &= \langle\xi_{0,n},\phi\rangle + \sum_{k\in\Gamma,i} \int_0^t \int_{|z|\le 1} \big\langle \xi_n(s-), \phi\big(\varphi_{z\theta_k}^{k,i}\big) -\phi \big\rangle \, \tilde N_{k,i}(\mathrm{d}z,\mathrm{d}s) \\
&\quad+ \sum_{k\in\Gamma,i} \int_0^t \int_{|z|\le 1} \big\langle \xi_n(s), \phi\big(\varphi_{z\theta_k}^{k,i}\big) -\phi - z \theta_k\sigma_{k,i} \cdot\nabla \phi \big\rangle\, \nu(\mathrm{d}z)\mathrm{d}s.
\end{aligned}
\end{equation}
We now address the convergence of each term in \eqref{sub-STE-weak} as $n \to \infty$. The convergence of the initial data term follows directly from the assumed $L^2$-convergences of $\xi_{0,n} \to \xi_0$. Similarly, the convergence $\xi_n(t) \to \xi(t)$ in $L^2(\mathbb{T}^d)$ implies that $\langle \xi_n(t), \phi \rangle \to \langle \xi(t), \phi \rangle$.

For the martingale term, Doob's inequality, the independence of the Poisson random measures $N_{k,i}$ and the It\^o isometry yield
\begin{equation*}
\begin{aligned}
I_n &:= \mathbb{E}\sup_{t\le T}
\bigg|\sum_{k\in\Gamma,i} \int_0^t \int_{|z|\le 1}
\big\langle \xi_n(s-)-\xi(s-),
\phi\big(\varphi_{z\theta_k}^{k,i}\big)-\phi \big\rangle
\,\tilde N_{k,i}(\mathrm{d}z,\mathrm{d}s)\bigg|^{2}\\
&\le 4\sum_{k\in\Gamma,i}\mathbb{E}\int_0^T\int_{|z|\le 1}
\big\langle \xi_n(s)-\xi(s),
\phi\big(\varphi_{z\theta_k}^{k,i}\big)-\phi \big\rangle^{2}
\,\nu(\mathrm{d}z)\mathrm{d}s.
\end{aligned}
\end{equation*}
Apply \eqref{prelim-summed-noise} to $\xi_n(s)-\xi(s)$, and by the measure-preserving property of $X_s^{-1}$, we obtain
\begin{equation}\label{11}
\begin{aligned}
I_n
&\le 4\overline{C}_d\|\nabla\phi\|_{L^\infty}^2
\mathbb E\biggl(\int_0^T
\|\xi_n(s)-\xi(s)\|_{L^2}^2\,\mathrm{d}s\biggr)
\int_{|z|\le1}z^2\,\nu(\mathrm{d}z)\\
&=4\overline{C}_dT\|\nabla\phi\|_{L^\infty}^2
\|\xi_{0,n}-\xi_0\|_{L^2}^2
\int_{|z|\le1}z^2\,\nu(\mathrm{d}z)
\to0.
\end{aligned}
\end{equation}
Regarding the last term in \eqref{sub-STE-weak}, we examine
\begin{equation*}
\begin{aligned}
J_n
&:=\mathbb E\sup_{t\le T}\biggl|
\sum_{k\in\Gamma,i}\int_0^t\int_{|z|\le1}
\bigl\langle\xi_n(s)-\xi(s),
\phi\bigl(\varphi_{z\theta_k}^{k,i}\bigr)-\phi
-z\theta_k\sigma_{k,i}\cdot\nabla\phi\bigr\rangle
\,\nu(\mathrm{d}z)\,\mathrm{d}s
\biggr|\\
&\le \mathbb E
\sum_{k\in\Gamma,i}\int_0^t\int_{|z|\le1}
\big|\bigl\langle\xi_n(s)-\xi(s),
\phi\bigl(\varphi_{z\theta_k}^{k,i}\bigr)-\phi
-z\theta_k\sigma_{k,i}\cdot\nabla\phi\bigr\rangle\big|
\,\nu(\mathrm{d}z)\,\mathrm{d}s.
\end{aligned}
\end{equation*}
By  \eqref{prelim-summed-compensator}, applied to $\xi_n(s)-\xi(s)$, together with the measure-preserving property of $X_s^{-1}$, we obtain
\begin{equation}\label{12}
\begin{aligned}
J_n
&\le \overline{C}_d\|\nabla^2\phi\|_{L^\infty}
\mathbb E\biggl(\int_0^T
\|\xi_n(s)-\xi(s)\|_{L^2}\,\mathrm{d}s\biggr)
\int_{|z|\le1}z^2\,\nu(\mathrm{d}z)\\
&=\overline{C}_dT\|\nabla^2\phi\|_{L^\infty}
\|\xi_{0,n}-\xi_0\|_{L^2}
\int_{|z|\le1}z^2\,\nu(\mathrm{d}z)
\to0.
\end{aligned}
\end{equation}
Summarizing these results, we can pass to the limit in \eqref{sub-STE-weak} to show that
\eqref{STE-Marcus-2} holds $\mathbb P$-a.s. for all $t\in[0,T]$.
Together with \eqref{xi-L2} and the properties above,
this shows that $\xi$ is a solution in the sense of
Definition~\ref{def1}.
\end{proof}

\subsection{Scaling limit of Marcus stochastic transport equations} \label{subs-scaling-SLTE}

This section is devoted to the proof of Theorem \ref{main sca}. Recall that we are given a sequence of noise coefficients $\theta^n$ satisfying Hypotheses \ref{hypo-coefficients}, and now consider a sequence of stochastic transport equations \eqref{STE-Marcus-n}, which are formulated in the weak sense as follows: for $\phi\in C^\infty(\T^d)$, $\P$-a.s. for all $t\in[0,T]$,
\begin{equation}\label{STE-weak-n}
\begin{aligned}
\langle \xi^{n}(t),\phi\rangle&= \langle \xi_0,\phi\rangle+ \sum_{k,i} \int_0^t\! \int_{|z|\le 1} \big\langle \xi^{n}(s-), \phi\big(\varphi_{z\theta_k^n}^{k,i} \big) -\phi \big\rangle   \, \tilde N_{k,i}(\mathrm{d}z,\mathrm{d}s) \\
&\quad + \sum_{k,i}\int_0^t\! \int_{|z|\le 1} \big\langle \xi^n(s) , \phi\big(\varphi_{z\theta_k^n}^{k,i} \big) -\phi - z \theta^n_k\sigma_{k,i} \cdot\nabla \phi \big\rangle  \, \nu(\mathrm{d}z)\mathrm{d}s.
\end{aligned}
\end{equation}
Here $\varphi^{k,i}$ is the flow defined in \eqref{characteristics}, and $\varphi_{z\theta_k^n}^{k,i}(x)=x+z\theta_k^n\sigma_{k,i}(x)$ by \eqref{flow-expression}.
The solution constructed in Theorem \ref{cunzaixing} admits an explicit representation in terms of the inverse flow:
$$\xi^n(t)=\xi_0\big((X_t^n)^{-1}\big),$$
where $X_t^n$ solves the Marcus stochastic equation
\begin{equation}\label{X_t^n(x)}
\mathrm{d}X_t^n = \sum_{k,i} \theta_k^n \sigma_{k,i}\big(X_t^n\big) \diamond \mathrm{d}Z_t^{k,i},\quad X_0^n(x)=x.
\end{equation}
By Theorem \ref{cunzaixing}, $\xi^n$ preserves the $L^2(\mathbb{T}^d)$-norm: for all $n\geq1$,
\begin{equation}\label{xi^n(t)}
\P\mbox{-a.s.} \quad \|\xi^n(t)\|_{L^2}=\|\xi_0\|_{L^2} \quad \mbox{for all } t\in[0,T],
\end{equation}
which ensures that $\xi^n \in L^\infty\big(\Omega,L^\infty([0,T];L^2(\mathbb{T}^d))\big)$.

\begin{remark}
Under the conditions of coefficients $\{\theta^n\}$ (see Hypotheses \ref{hypo-coefficients}), one can show that  the generator of  $X^n$  converges to $\kappa \Delta$. Indeed,  \eqref{X_t^n(x)} is understood as \eqref{linear-characteristic-ito}, and thus the generator of $X^n$ has the form
\[
(\mathcal A^n f)(y)
=
\sum_{k,i}\int_{0<|z|\leq1}
\left[
f\big(\varphi_{z\theta_k^n}^{k,i}(y)\big)-f(y)
-z\theta_k^n\sigma_{k,i}(y)\cdot\nabla f(y)
\right]\nu(\mathrm{d}z),
\]
where $f\in C^2(\mathbb T^d)$. Taylor's formula with integral remainder and
\eqref{eq-important-equality} give
\[
\bigl\|\mathcal A^n f-\kappa\Delta f\bigr\|_{L^\infty}
\le
(d-1)
\sup_{|h|\le\sqrt2\,\|\theta^n\|_{\ell^\infty}}
\bigl\|\nabla^2f(\cdot+h)-\nabla^2f\bigr\|_{L^\infty}
\int_{|z|\le1}z^2\,\nu(\mathrm{d}z)
\to0,
\]
where the supremum is taken over $h\in\mathbb T^d$. The convergence follows from the uniform continuity of $\nabla^2f$, Hypothesis~\ref{hypo-coefficients}(4), and $\int_{|z|\le1}z^2\,\nu(\mathrm{d}z)<\infty$.
Since $C^\infty(\mathbb T^d)$ is a core for $\kappa\Delta$,
the estimate above establishes convergence of the generators
on this core. The standard convergence theorem for Feller
processes therefore yields convergence of $X^n(x)$ to the
Brownian motion on $\mathbb T^d$ with generator $\kappa\Delta$
and initial point $x$; see \cite[Chapter~4]{Dspace}.
This is a one-particle statement only and does not describe
the joint motion of particles starting from distinct points.
\end{remark}

We present two lemmas estimating the last two terms in \eqref{STE-weak-n}.
\begin{lemma}\label{lem-levy}
For every $\phi\in C^\infty(\mathbb T^d)$, the difference between the last term in \eqref{STE-weak-n} and
\[
\kappa\int_0^t\langle\xi^n(s),\Delta\phi\rangle\,\mathrm{d}s
\]
converges to zero in $L^2(\Omega;C([0,T]))$ as $n\to\infty$, where $\kappa=\frac12 C_d\int_{|z|\le1}z^2\,\nu(\mathrm{d}z)$ with the same dimensional constant $C_d$ as in \eqref{eq-important-equality}.
\end{lemma}
\begin{proof}
 Denote
 $$\Phi_z^n (x):= \sum_{k,i} \big[\phi\big(\varphi_{z\theta_k^n}^{k,i}(x) \big) -\phi(x) - z \theta_k^n(\sigma_{k,i} \cdot\nabla \phi)(x) \big], \quad x\in \mathbb{T}^d. $$
By \eqref{2-taylor} with $t=z\theta_k^n$, we have
\begin{equation}\label{Taylor-n-2}
\begin{aligned}
\phi\big(\varphi_{z\theta_k^n}^{k,i}(x) \big) - \phi(x)
&= z\theta_k^n \sigma_{k,i}(x) \cdot \nabla\phi(x) + \frac{z^2{(\theta_k^n)}^2}{2} \nabla^2\phi(\eta_{k,i}^n(x)) : (\sigma_{k,i} \otimes \sigma_{k,i})(x),
\end{aligned}
\end{equation}
where $\eta_{k,i}^n(x)$ belongs to the line segment linking $x$ and $x+ z \theta_k^n \sigma_{k,i}(x)$. Therefore,
  $$
  \begin{aligned}
  \Phi_z^n (x) &= \frac{z^2}2 \sum_{k,i} (\theta_k^n)^2\nabla^2\phi(\eta_{k,i}^n) : (\sigma_{k,i}\otimes \sigma_{k,i})(x) \\
  &= \frac{z^2}2 \sum_{k,i}  (\theta_k^n)^2 \nabla^2\phi(x) : (\sigma_{k,i}\otimes \sigma_{k,i})(x) \\
  &\quad + \frac{z^2}2 \sum_{k,i} (\theta_k^n)^2\big[\nabla^2\phi(\eta_{k,i}^n)- \nabla^2\phi(x)\big] : (\sigma_{k,i}\otimes \sigma_{k,i})(x).
 \end{aligned}
 $$
Let $\Psi_z^n(x)$ be the last term; recalling the identity \eqref{eq-important-equality}, we obtain
 $$\Phi_z^n (x)= \frac12 C_d z^2 \Delta \phi(x) + \Psi_z^n(x).$$
This decomposition yields
\begin{equation}\label{levy-con-1}
  \begin{aligned}
  &  \sum_{k,i} \int_0^t \int_{|z|\le 1} \big\langle \xi^n(s) , \phi\big( \varphi_{z\theta_k^n}^{k,i} \big) -\phi - z \theta^n_k\sigma_{k,i} \cdot\nabla \phi \big\rangle\, \,\nu(\mathrm{d}z) \, \mathrm{d}s \\
  &\quad= \frac12 C_d \bigg( \int_{|z|\le 1} z^2 \,\nu(\mathrm{d}z) \bigg) \int_0^t \langle \xi^n(s) , \Delta \phi \rangle \, \mathrm{d}s +\int_0^t \int_{|z|\le 1} \langle \xi^n(s) , \Psi_z^n \rangle \,\nu(\mathrm{d}z) \, \mathrm{d}s.
  \end{aligned}
\end{equation}

 Comparing \eqref{levy-con-1} with the statement of the lemma, it remains to show that the error term $\Psi_z^n (x)$ vanishes as $n\to \infty$. Indeed, recalling that $\eta_{k,i}^n$ belongs to the line segment linking $x$ and $x+ z \theta_k^n \sigma_{k,i}(x)$, we have
  \begin{equation}
     \begin{aligned}\label{1}
          |\nabla^2\phi(\eta_{k,i}^n)- \nabla^2\phi(x)| &\le \|\nabla^3\phi \|_{L^\infty}\, |\eta_{k,i}^n -x| \\&\le \|\nabla^3\phi \|_{L^\infty}\, |\theta_k^n|\, |z|\, \|\sigma_{k,i}\|_{L^\infty} \\
          &\le \sqrt{2} \|\nabla^3\phi \|_{L^\infty} \|\theta ^n\|_{\ell^\infty}\, |z|,
     \end{aligned}
  \end{equation}
where the last inequality used the bound $\|\sigma_{k,i}\|_{L^\infty}\leq \sqrt{2}.$
Also note that $\|\theta^n\|_{\ell^2}=1$, thus
  \begin{equation}\label{2}
      \sum_{k,i}  (\theta_k^n)^2 |(\sigma_{k,i}\otimes \sigma_{k,i})(x)| \le 2(d-1)\sum_{k}  (\theta_k^n)^2  ={2}(d-1).
  \end{equation}
Combining \eqref{1} and \eqref{2}, we obtain the uniform estimate
  $$|\Psi_z^n(x)|\le C'_d |z|^3\, \|\nabla^3\phi \|_{L^\infty} \|\theta^n \|_{\ell^\infty},\quad x\in \mathbb{T}^d, $$
where $C'_d $ is a constant depending only on $d$.
Hence,
\begin{equation}\label{levy-con-3}
  \begin{aligned}
  \sup_{t\in[0,T]}\bigg|\int_0^t \int_{|z|\le 1} \big\langle \xi^n(s) , \Psi^n_z \big\rangle \,\nu(\mathrm{d}z)\mathrm{d}s \bigg|
  &\le C'_d \|\theta^n\|_{\ell^\infty}\|\nabla^3\phi\|_{L^\infty} \bigg( \int_{|z|\le 1} z^2 \,\nu(\mathrm{d}z) \bigg)\int_0^T \|\xi^n(s) \|_{L^2} \, \mathrm{d}s \\
  &\le C'_d  \|\theta^n\|_{\ell^\infty}\|\nabla^3\phi\|_{L^\infty} \bigg( \int_{|z|\le 1} z^2\,\nu(\mathrm{d}z) \bigg) T \|\xi_0\|_{L^2},
  \end{aligned}
\end{equation}
thanks to the uniform bound $\|\xi^n(s) \|_{L^2} \le \|\xi_0 \|_{L^2}$.
Thus, this term vanishes uniformly for $t\in[0,T]$ since $\|\theta^n \|_{\ell^\infty} \to 0$ as $n\to\infty$. Combining this with \eqref{levy-con-1} completes the proof.
\end{proof}

Now we turn to dealing with the martingale term in \eqref{STE-weak-n}.
\begin{lemma}\label{lem-mart}
   The martingale term $M^n$ in \eqref{STE-weak-n} satisfies $\mathbb E\sup_{t\in[0,T]}|M_t^n|^2\to0$ as $n\to\infty$.
\end{lemma}

\begin{proof}
    Consider the martingale term:
 $$M_t^n:= \sum_{k,i}\int_0^t \int_{|z|\le 1} \big\langle \xi^n(s-), \phi\big(\varphi_{z\theta_k^n}^{k,i} \big) -\phi \big\rangle \, \tilde N_{k,i}(\mathrm{d}z,\mathrm{d}s). $$
By Doob's inequality and the compensated-Poisson isometry (see \cite[Theorem~1.14]{Eberle15}), we have
  \begin{equation}\label{M}
  \begin{aligned}
  \mathbb{E}\sup_{t\in[0,T]}|M_t^n|^2
  \le 4\sum_{k,i} \mathbb{E} \int_0^T \int_{|z|\le 1} \big\langle \xi^n(s), \phi\big(\varphi_{z\theta_k^n}^{k,i} \big) -\phi \big\rangle^2 \,\nu(\mathrm{d}z) \mathrm{d}s.
  \end{aligned}
  \end{equation}
By \eqref{Taylor-n-2}, we obtain the estimate
\begin{equation}\label{chaikai}
  \begin{aligned}
  \sum_{k,i} \big\langle \xi^n(s) , \phi\big(\varphi_{z\theta_k^n}^{k,i} \big) -\phi \big\rangle^2
  &\le 2 \sum_{k,i}\big\langle \xi^n(s) , z\theta^n_k \sigma_{k,i}\cdot\nabla\phi\big\rangle^2 \\
  &\quad+ 2 \sum_{k,i}\Big\langle \xi^n(s) , \frac{z^2{(\theta_k^n)}^2}{2} \nabla^2\phi(\eta_{k,i}^n) : (\sigma_{k,i} \otimes \sigma_{k,i})\Big\rangle^2.
  \end{aligned}
  \end{equation}
  We now analyze each term separately. For the first term, define
 \begin{equation}\label{estimate-0}
  \begin{aligned}
      J^n_1&:=\sum_{k,i}\big\langle \xi^n(s) , z\theta^n_k \sigma_{k,i} \cdot\nabla\phi\big\rangle^2
      \le \|\theta^n \|_{\ell^\infty}^2z^2  \sum_{k,i}\big\langle \xi^n(s) \nabla\phi, \sigma_{k,i}\big\rangle^2.
  \end{aligned}
  \end{equation}
 Due to the fact that $\{\sigma_{k,i}\}_{k\in \mathbb{Z}^d_0, i = 1,\ldots, d - 1}$ is an orthonormal system of divergence free vector fields in $L^2(\mathbb{T}^d)$, it follows that
  $$\sum_{k,i}\big\langle \xi^n(s) \nabla\phi, \sigma_{k,i}\big\rangle^2 \le\|\xi^n(s) \nabla\phi\|_{L^2}^2\le \|\xi^n(s)\|_{L^2}^2 \|\nabla\phi\|_{L^\infty}^2 \le \|\xi_0\|_{L^2}^2 \|\nabla\phi\|_{L^\infty}^2,$$
where in the last step we have used \eqref{xi^n(t)}. Substituting this into \eqref{estimate-0} yields
  \begin{equation*}
  \begin{aligned}
J^n_1\le \|\theta^n\|_{\ell^\infty}^2z^2\|\xi_0\|_{L^2}^2 \|\nabla\phi\|_{L^\infty}^2.
 \end{aligned}
  \end{equation*}
 For the second term, we have
\begin{equation*}
  \begin{aligned}
J^n_2&:= \sum_{k,i}\bigg\langle \xi^n(s), \frac{z^2{(\theta_k^n)}^2}{2} \nabla^2\phi(\eta_{k,i}^n) : (\sigma_{k,i} \otimes \sigma_{k,i})\bigg\rangle^2\\
&\le \frac{z^4}{4} \sum_{k,i}  \big(\theta^n_k\big)^4 \| \xi^n(s) \|_{L^2}^2 \|\nabla^2\phi\big(\eta^n_{k,i}\big):(\sigma_{k,i} \otimes \sigma_{k,i}) \|_{L^2}^2.
 \end{aligned}
  \end{equation*}
By \eqref{xi^n(t)} and $\|\theta^n\|_{\ell^2}=1 $ for any $n\ge1$, we obtain
 \begin{equation*}
  \begin{aligned}
J^n_2&\le Cz^4 \|\theta^n\|_{\ell^\infty}^2\Big(\sum_{k,i}  \big(\theta^n_k\big)^2\Big)\|\xi_0 \|_{L^2}^2 \big\|\nabla^2\phi \big\|_{L^\infty}^2\\
&\le \overline{C}_dz^4 \|\theta^n\|_{\ell^\infty}^2 \|\xi_0 \|_{L^2}^2 \big\|\nabla^2\phi \big\|_{L^\infty}^2.
  \end{aligned}
  \end{equation*}
Substituting the estimates on $J^n_1$ and $J^n_2$ into \eqref{chaikai} yields
  \begin{equation*}
  \begin{aligned}
  \sum_{k,i} \big\langle \xi^n(s) , \phi\big(\varphi_{z\theta_k^n}^{k,i} \big) -\phi \big\rangle^2
  &\le 2J^n_1+2J^n_2\le C(d)\|\theta^n\|_{\ell^\infty}^2(z^2+z^4)\|\xi_0 \|_{L^2}^2 \|\phi \|_{C^2}^2,
  \end{aligned}
  \end{equation*}
where $\|\phi \|_{C^2}= \sum_{m=0}^2 \|\nabla^m \phi \|_{L^\infty}$. Combining this with \eqref{M} we obtain
  $$\mathbb{E}\sup_{t\in[0,T]}|M_t^n|^2\le CT\|\theta^n \|_{\ell^\infty}^2 \|\xi_0 \|_{L^2}^2 \|\phi \|_{C^2}^2 \int_{|z|\le 1} z^2 \,\,\nu(\mathrm{d}z), $$
which vanishes as $n\to \infty$ due to item (4) in Hypotheses \ref{hypo-coefficients}.
\end{proof}

Finally, we can complete the

\begin{proof}[Proof of Theorem~\ref{main sca}] Let $\xi(t)={\rm e}^{\kappa t\Delta}\xi_0$ be the solution of deterministic heat equation. We prove that, for every $\varepsilon>0$,
\begin{equation}\label{linear-strong-negative}
\mathbb E\sup_{t\le T}
\|\xi^n(t)-\xi(t)\|_{H^{-\varepsilon}}^2
\to0.
\end{equation}
Fix $j\in\mathbb Z_0^d$. Recall from \eqref{ek} that $\Delta e_j=-4\pi^2|j|^2e_j$. Let $\Psi_z^{n,j}$ and $M_t^{n,j}$ denote respectively the quantities $\Psi_z^n$ and $M_t^n$ defined in the proofs of Lemmas~\ref{lem-levy} and \ref{lem-mart} with $\phi=e_j$. The weak formulation of \eqref{limit-heat-eq}, tested against $e_j$, reads
\[
\langle\xi(t),e_j\rangle
=\langle\xi_0,e_j\rangle
-4\pi^2\kappa|j|^2
\int_0^t\langle\xi(s),e_j\rangle\,\mathrm{d}s.
\]
Using \eqref{levy-con-1} and subtracting this identity from \eqref{STE-weak-n} gives
\[
\begin{aligned}
\langle\xi^n(t)-\xi(t),e_j\rangle
&=-4\pi^2\kappa|j|^2
\int_0^t\langle\xi^n(s)-\xi(s),e_j\rangle\,\mathrm{d}s+M_t^{n,j}
+\int_0^t\!\int_{|z|\le1}
\langle\xi^n(s),\Psi_z^{n,j}\rangle\,\nu(\mathrm{d}z)\,\mathrm{d}s.
\end{aligned}
\]
Taking the supremum over time and applying Gronwall's inequality, we obtain
\[
\begin{aligned}
\sup_{t\le T}|\langle\xi^n(t)-\xi(t),e_j\rangle|
&\le e^{4\pi^2\kappa|j|^2T}
\Biggl(
\sup_{t\le T}|M_t^{n,j}|+\sup_{t\le T}\biggl|
\int_0^t\!\int_{|z|\le1}
\langle\xi^n(s),\Psi_z^{n,j}\rangle\,\nu(\mathrm{d}z)\,\mathrm{d}s
\biggr|
\Biggr).
\end{aligned}
\]

 By \eqref{levy-con-3} and the estimate at the end of the proof of Lemma~\ref{lem-mart}, we obtain
\[
\mathbb E\sup_{t\le T}
|\langle\xi^n(t)-\xi(t),e_j\rangle|^2
\le C(d,j,T,\nu)\|\xi_0\|_{L^2}^2
\|\theta^n\|_{\ell^\infty}^2
\to0.
\] Here $C$ is independent of $n$.
  By \eqref{xi^n(t)} and the estimate $\|\xi(t)\|_{L^2}\le\|\xi_0\|_{L^2}$ for the heat equation, for every $N\ge1$,
\[
\begin{aligned}
\mathbb E\sup_{t\le T}\|\xi^n(t)-\xi(t)\|_{H^{-\varepsilon}}^2
&\le\sum_{0<|j|\le N}(1+4\pi^2|j|^2)^{-\varepsilon}\mathbb E\sup_{t\le T}|\langle\xi^n(t)-\xi(t),e_j\rangle|^2\\
&\quad+4\|\xi_0\|_{L^2}^2(1+4\pi^2N^2)^{-\varepsilon}.
\end{aligned}
\] First letting $n\to\infty$ for fixed $N$, and then $N\to\infty$, proves \eqref{linear-strong-negative}.
 This finishes the proof of Theorem \ref{main sca}.
\end{proof}

\section{Weak existence for stochastic 2D Euler equations \eqref{stoch-2D-Euler-eq}}\label{sect:see}
This section is devoted to proving existence of weak solutions to \eqref{stoch-2D-Euler-eq}, using the classical Galerkin approximation and compactness arguments.
Recall from Section~\ref{subs-struc-noise} that $\theta$ is radially symmetric, $\|\theta\|_{\ell^2}=1$, and $\Gamma=\{k\in\mathbb Z_0^2:\theta_k\neq0\}$ is finite. Throughout this section, $\theta$ and $\Gamma$ are fixed, while $n$ is the Galerkin approximation index.
The proof proceeds through the following steps:
\begin{enumerate}  
    \item \textbf{Galerkin approximation}: we approximate \eqref{stoch-2D-Euler-eq} by projecting it onto finite-dimensional spaces and then we derive uniform bounds for the approximate solutions.
    \item \textbf{Stochastic compactness}: we show that the laws of approximate solutions are tight in suitable Skorohod spaces.
    \item \textbf{Existence of weak solution}: we prove the existence of  solution by applying the Skorohod representation theorem.
\end{enumerate}
We will perform these steps in the following subsections.

\subsection{Galerkin approximation}\label{Faedo-Galerkin}

This subsection constructs finite-dimensional approximations of \eqref{stoch-2D-Euler-eq} using the Faedo--Galerkin method.

Let $H_n := \mathrm{span}\{e_k: 0<|k|\le n\} \subset L^2(\mathbb{T}^2)$ denote the finite-dimensional Fourier-Galerkin subspace, where $\{e_k\}_{k\in\mathbb{Z}_0^2}$ is the orthonormal Fourier basis defined in \eqref{ek}. The  orthogonal projection $\Pi_n: L^2(\mathbb{T}^2) \to H_n$ satisfies
\[
\langle \Pi_n f, g \rangle_{L^2} = \langle f, \Pi_ng \rangle_{L^2}, \quad \forall f,g \in L^2(\mathbb{T}^2).
\]
Projecting \eqref{stoch-2D-Euler-eq} onto  $H_n$ leads to the Galerkin approximation:
  \begin{equation}\label{eq:main_eq}
  \mathrm{d}\xi_n + \Pi_n\big(u_n \cdot \nabla \xi_n\big) \mathrm{d}t + \sum_{k\in\Gamma} \theta_k A_k^n(\xi_n)\diamond \mathrm{d}Z^k_t = 0,
  \end{equation}
where $\xi_n := \Pi_n \xi$ represents the projection of the vorticity onto $H_n$ and $u_n := K\ast \xi_n= \Pi_n(K\ast \xi)$ denotes the corresponding velocity field. For notational convenience, we introduce the operator
  $$A_k^n : H_n \rightarrow H_n,\quad \xi_n \mapsto A_k^n(\xi_n)= \Pi_n (\sigma_k \cdot \nabla \xi_n).$$
By the definition of $\sigma_k$ and expressing $\xi_n $ in terms of the basis of $H_n$, it is easy to see that $A_k^n$ can be represented as a matrix and
$$A_k^n(\xi_n) = A_k^n \xi_n,$$
where the right-hand side denotes the product of matrix and vector. Since $\sigma_k$ is divergence-free, one can show that $A_k^n$ is  skew-adjoint; indeed, for any $\xi_n, \eta_n \in H_n$,
  $$\langle A_k^n (\xi_n),\eta_n\rangle=\langle \sigma_k \cdot \nabla \xi_n, \eta_n\rangle=-\langle\xi_n, \sigma_k \cdot \nabla \eta_n\rangle=-\langle\xi_n,A_k^n(\eta_n)\rangle .$$

Following \cite{KPP95}, equation \eqref{eq:main_eq} can be expressed in It\^o integral form as
\begin{equation}\label{Galerkin-eq}
    \begin{aligned}
        \xi_n(t)
&= \xi_n(0) - \int_0^t \Pi_n \big(u_n(s) \cdot \nabla \xi_n(s)\big)  \,\mathrm{d}s - \sum_{k\in\Gamma}\theta_k \int_0^t A_k^n \big(\xi_n(s-)\big) \, \mathrm{d} Z_s^k \\
&\quad + \sum_{k\in\Gamma} \sum_{0<s\leq t} \big[ \varphi\big(-\Delta Z^k_s\theta_k  A_k^n, \xi_n(s-)\big) - \xi_n(s-) + \Delta Z^k_s\theta_k  A_k^n (\xi_n(s-)) \big],
    \end{aligned}
\end{equation}
where $\xi_n(0)=\Pi_n \xi_0$ for some $\xi_0\in L^2(\T^2)$, $\Delta Z^k_s= Z^k_s- Z^k_{s-}$ and the mapping $\varphi$ is given by
\[\varphi \big(-z \theta_k A_k^n, f\big) := \Phi \big(1;-z \theta_k A_k^n, f\big),\quad f\in H_n,\]
namely, the value at $t=1$ of the solution \(\Phi=\big\{\Phi(t;-z \theta_k A_k^n, f)\big\}_{t\in \mathbb{R}}\) to the differential equation on $H_n$:
  \[ \frac{\mathrm{d}\Phi}{\mathrm{d}t} = -z \theta_k A_k^n (\Phi), \quad \Phi(0) = f. \]
It is well known that the solution can be expressed in terms of the exponential of matrix:
\[ \varphi(-z \theta_k A_k^n, f) = {\rm e}^{-z \theta_k A_k^n} f.\]
In light of the definition of the integral with respect to a Poisson random measure, \eqref{Galerkin-eq} can be recast as
\begin{equation*}
\begin{aligned}
\xi_n(t) &= \xi_n(0) - \int_0^t \Pi_n \big(u_n(s) \cdot \nabla \xi_n(s)\big)  \,\mathrm{d}s - \sum_{k\in\Gamma} \theta_k \int_0^t\! \int_{|z| \leq 1} z A_k^n (\xi_n(s-))  \,\tilde{N}_k (\mathrm{d}z, \mathrm{d}s) \\
&\quad + \sum_{k\in\Gamma} \int_0^t\! \int_{|z| \leq 1} \big[ {\rm e}^{-z \theta_k A_k^n} \xi_n(s-) - \xi_n(s-) + z \theta_k A_k^n (\xi_n(s-)) \big] \,{N}_k (\mathrm{d}z, \mathrm{d}s).
\end{aligned}
\end{equation*}
Using the relation $\tilde{N}_k(\mathrm{d}z, \mathrm{d}s) = N_k(\mathrm{d}z, \mathrm{d}s) - \nu(\mathrm{d}z) \mathrm{d}s$, the equation becomes
\begin{equation}\label{Hn-equ-2}
\begin{aligned}
\xi_n(t) &= \xi_n(0) - \int_0^t \Pi_n \big(u_n(s) \cdot \nabla \xi_n(s)\big) \, \mathrm{d}s \\
&\quad + \sum_{k\in\Gamma} \int_0^t\! \int_{|z| \leq 1} \big[ {\rm e}^{-z \theta_k A_k^n} \xi_n(s-) - \xi_n(s-)\big] \tilde{N}_k (\mathrm{d}z, \mathrm{d}s) \\
&\quad + \sum_{k\in\Gamma} \int_0^t\! \int_{|z| \leq 1} \big[ {\rm e}^{-z \theta_k A_k^n} \xi_n(s) - \xi_n(s) + z \theta_k A_k^n (\xi_n(s)) \big] \,{\nu} (\mathrm{d}z)\mathrm{d}s.
\end{aligned}
\end{equation}

The following result provides the existence of solutions to \eqref{Hn-equ-2} and the energy estimate for $\xi_n$, which is necessary for deriving compactness results when passing to the limit as $n \to \infty$.

\begin{proposition}\label{xi-Linf-L2}
For any initial  $\xi_0\in L^2(\mathbb{T}^2)$, there exists a unique $(\mathcal{F}_t)$-adapted c\`adl\`ag process $\xi_n $ with path in  $ D([0,T];H_n)$ $\mathbb{P}$-a.s. satisfying \eqref{Hn-equ-2}. Moreover, $\xi_n$  satisfies the energy bound
\begin{equation}\label{uniform}
\mathbb{P}\text{-a.s.}\quad\sup_{t\in[0,T]}\|\xi_n(t)\|_{L^2}\leq\|\xi_0\|_{L^2}.
\end{equation}
\end{proposition}

\begin{proof}
Note that the drift term $\Pi_n(u_n\cdot \nabla\xi_n)$ in \eqref{Hn-equ-2} is in fact a quadratic polynomial mapping on $H_n$, the existence of a unique local solution follows from \cite[Corollary~3.3]{KPP95}, while the global existence is an immediate consequence of the  energy estimate derived below. We apply the It\^o formula to  $ \|\xi_n\|_{L^2}^2$ for the Marcus equation \eqref{eq:main_eq}:

\begin{align*}
\mathrm{d}\|\xi_n\|_{L^2}^2
&=2\big\langle\xi_n,-\Pi_n(u_n\cdot\nabla\xi_n)\big\rangle \mathrm{d}t -2\sum_{k\in\Gamma}\theta_k\big\langle\xi_n,\Pi_n(\sigma_k\cdot\nabla\xi_n)\big\rangle\diamond \mathrm{d}Z_t^k.
\end{align*}
The key observation is that both terms on the right-hand side vanish due to the fact that $u_n$ and $\sigma_k$ are both divergence-free. Combining this with the self-adjointness of $\Pi_n$ and integrating by parts, we obtain
\begin{align*}
\big\langle \xi_n, \Pi_n(u_n \cdot \nabla \xi_n) \big\rangle &=\big\langle \xi_n, u_n \cdot \nabla \xi_n \big\rangle = \frac{1}{2} \big\langle u_n, \nabla |\xi_n|^2 \big\rangle = 0, \\
\big\langle \xi_n, \Pi_n(\sigma_k \cdot \nabla \xi_n) \big\rangle &= \frac{1}{2} \big\langle \sigma_k, \nabla |\xi_n|^2 \big\rangle = 0.
\end{align*}
Therefore, $\mathrm{d}\|\xi_n\|_{L^2}^2 =0,$ which implies $\|\xi_n(t) \|_{L^2}^2 = \|\xi_n(0) \|_{L^2}^2\le \|\xi_0\|_{L^2}^2$ for all $t\in[0,T]$, therefore we obtain the desired energy bound.
\end{proof}

\subsection{Tightness results}\label{com-tight}

We first derive estimates for the Galerkin operators and Marcus flows. These estimates will be used to show joint tightness of the Galerkin approximations and their driving processes.
\begin{lemma}\label{lem-double-derivative}
There exists a constant $C>0$, independent of $n$ and $k$,
such that, for every $n\ge1$, $k\in\mathbb Z^2_0$
and $f\in H^2(\mathbb T^2)$,
\begin{equation}\label{l2}
\|\sigma_k\cdot\nabla\Pi_n(\sigma_k\cdot\nabla f)\|_{L^2}
+\|(A_k^n)^2\Pi_n f\|_{L^2}
\le C\|\nabla^2 f\|_{L^2}.
\end{equation}
\end{lemma}
\begin{proof}
By the definition in Section~\ref{subs-struc-noise},
\[
\sigma_k=a_ke_k,
\quad
|a_k|=1,
\quad
a_k\cdot k=0,
\]
and hence $a_k\cdot\nabla e_k=0$.
Moreover, the operator $a_k\cdot\nabla$
commutes with the Fourier projection $\Pi_n$:
\[
a_k\cdot\nabla\Pi_n g
=\Pi_n\bigl(a_k\cdot\nabla g\bigr),
\quad g\in H^1(\mathbb T^2).
\]
Therefore, for every $f\in H^2(\mathbb T^2)$,
the product rule gives
\[
\begin{aligned}
\sigma_k\cdot\nabla\Pi_n(\sigma_k\cdot\nabla f)
&=e_k(a_k\cdot\nabla)\Pi_n
\bigl(e_k(a_k\cdot\nabla f)\bigr)=e_k\Pi_n\Bigl(
a_k\cdot\nabla\bigl(e_k(a_k\cdot\nabla f)\bigr)
\Bigr)\\
&=e_k\Pi_n\Bigl(
\bigl(a_k\cdot\nabla e_k\bigr)
(a_k\cdot\nabla f)
+e_k(a_k\cdot\nabla)^2f
\Bigr)=e_k\Pi_n\bigl(e_k(a_k\cdot\nabla)^2f\bigr).
\end{aligned}
\]
Using $\|e_k\|_{L^\infty}\le\sqrt{2}$,
the $L^2$ contraction property of $\Pi_n$, and $|a_k|=1$,
we obtain
\[
\begin{aligned}
\|\sigma_k\cdot\nabla\Pi_n(\sigma_k\cdot\nabla f)\|_{L^2}
&\le \|e_k\|_{L^\infty}^2
\|(a_k\cdot\nabla)^2f\|_{L^2}\le 2\|\nabla^2f\|_{L^2},
\end{aligned}
\]
where we used the pointwise bound
\[
|(a_k\cdot\nabla)^2f|
=|(a_k\otimes a_k):\nabla^2f|
\le|\nabla^2f|.
\]
Next, recalling that $A_k^ng=\Pi_n(\sigma_k\cdot\nabla g)$
for $g\in H_n$, we can apply the preceding estimate
to $\Pi_n f$ to obtain
\[
\begin{aligned}
\|(A_k^n)^2\Pi_n f\|_{L^2}
&=\bigl\|\Pi_n\bigl(
\sigma_k\cdot\nabla\Pi_n
(\sigma_k\cdot\nabla\Pi_n f)
\bigr)\bigr\|_{L^2}\le 2\|\nabla^2\Pi_n f\|_{L^2}\le 2\|\nabla^2f\|_{L^2}.
\end{aligned}
\]
This proves~\eqref{l2}, with a constant independent of $n$ and $k$.
\end{proof}

We next use \eqref{l2} and duality to estimate the increments of ${\rm e}^{tA_k^n}$.

\begin{lemma}\label{lem-galerkin-increments}
There exists a constant $C>0$ such that, for every
$n\ge1$, $k\in\mathbb Z_0^2$, $t\in\mathbb R$
and $f\in L^2(\mathbb T^2)$,
\begin{equation}\label{galerkin-analytic-increments}
\begin{aligned}
\|{\rm e}^{tA_k^n}\Pi_n f-\Pi_n f\|_{H^{-1}}
&\le C|t|\,\|f\|_{L^2},\\
\|{\rm e}^{tA_k^n}\Pi_n f-\Pi_n f-tA_k^n\Pi_n f\|_{H^{-2}}
&\le Ct^2\|f\|_{L^2},
\end{aligned}
\end{equation}
where $C$ is independent of $n$, $k$, $t$ and $f$.
\end{lemma}

\begin{proof}
Recall that $A_k^n$ is skew-adjoint on $H_n$, hence ${\rm e}^{sA_k^n}$ is an $L^2$ isometry for every $s\in\mathbb R$.
For $\phi\in C^\infty(\mathbb T^2)$,
$\|\sigma_k\|_{L^\infty}\le\sqrt2$ and \eqref{l2} give
\[
\|A_k^n\Pi_n\phi\|_{L^2}
\le\sqrt2\|\nabla\phi\|_{L^2},
\quad
\|(A_k^n)^2\Pi_n\phi\|_{L^2}
\le C\|\nabla^2\phi\|_{L^2}.
\]
Taylor's formula with integral remainder and the
isometry property therefore yield
\begin{equation}\label{R_n-taylor}
\begin{aligned}
\|{\rm e}^{tA_k^n}\Pi_n\phi-\Pi_n\phi\|_{L^2}
&\le C|t|\,\|\nabla\phi\|_{L^2},\\
\|{\rm e}^{tA_k^n}\Pi_n\phi-\Pi_n\phi
-tA_k^n\Pi_n\phi\|_{L^2}
&\le Ct^2\|\nabla^2\phi\|_{L^2}.
\end{aligned}
\end{equation}

By orthogonality of $\Pi_n$, skew-adjointness of $A_k^n$
and \eqref{R_n-taylor} with $t$ replaced by $-t$,
\[
\begin{aligned}
\left|
\langle {\rm e}^{tA_k^n}\Pi_n f-\Pi_n f,\phi\rangle
\right|
&=
\left|
\langle\Pi_n f,
{\rm e}^{-tA_k^n}\Pi_n\phi-\Pi_n\phi\rangle
\right|
\le C|t|\,\|f\|_{L^2}\|\phi\|_{H^1},\\
\left|
\langle {\rm e}^{tA_k^n}\Pi_n f-\Pi_n f
-tA_k^n\Pi_n f,\phi\rangle
\right|
&=
\left|
\langle\Pi_n f,
{\rm e}^{-tA_k^n}\Pi_n\phi-\Pi_n\phi
+tA_k^n\Pi_n\phi\rangle
\right|
\le Ct^2\|f\|_{L^2}\|\phi\|_{H^2}.
\end{aligned}
\]
Taking the supremum over smooth $\phi$ with
$\|\phi\|_{H^1}\le1$ and $\|\phi\|_{H^2}\le1$,
respectively, and using density proves
\eqref{galerkin-analytic-increments}.
\end{proof}

The following lemma will be used when passing to the limit in the weak equation.
\begin{lemma}\label{lem-galerkin-estimates}
Let $\phi\in C^\infty(\mathbb T^2)$. For every $n\ge1$, we have
\begin{equation}\label{R-n-prime}
\begin{aligned}
\bigl\|{\rm e}^{z\theta_kA_k^n}\Pi_n\phi-\Pi_n\phi\bigr\|_{L^2}
& +\bigl\|\phi\bigl(\varphi_{z\theta_k}^k\bigr)-\phi\bigr\|_{L^2}
\le C_\phi|z\theta_k|,\\
\bigl\|{\rm e}^{z\theta_kA_k^n}\Pi_n\phi-\Pi_n\phi-z\theta_kA_k^n\Pi_n\phi\bigr\|_{L^2}
& +\bigl\|\phi\bigl(\varphi_{z\theta_k}^k\bigr)-\phi-z\theta_k\sigma_k\cdot\nabla\phi\bigr\|_{L^2}
\le C_\phi z^2\theta_k^2,
\end{aligned}
\end{equation}
where $C_\phi$ depends only on $\phi$. Moreover, for each fixed $k$ and $z$, the following convergences hold in $L^2(\mathbb T^2)$ as $n\to\infty$:
\begin{equation}\label{galerkin-kernel-convergence}
\begin{aligned}
{\rm e}^{z\theta_kA_k^n}\Pi_n\phi-\Pi_n\phi
&\to\phi\bigl(\varphi_{z\theta_k}^k\bigr)-\phi,\\
{\rm e}^{z\theta_kA_k^n}\Pi_n\phi-\Pi_n\phi-z\theta_kA_k^n\Pi_n\phi
&\to\phi\bigl(\varphi_{z\theta_k}^k\bigr)-\phi-z\theta_k\sigma_k\cdot\nabla\phi.
\end{aligned}
\end{equation}
\end{lemma}
\begin{proof}
The estimates \eqref{R-n-prime} follow from \eqref{R_n-taylor} and \eqref{flow-L2-increments}, with $t=z\theta_k$.  Together with $\Pi_n\phi\to\phi$ in $H^1$ and the $L^2$ contraction property of $\Pi_n$, this gives
\[
A_k^n\Pi_n\phi=\Pi_n(\sigma_k\cdot\nabla\Pi_n\phi)\to\sigma_k\cdot\nabla\phi
\quad\text{in }L^2.
\]
Combining this convergence with Lemma~\ref{C.1} and $\Pi_n\phi\to\phi$ in $L^2$ proves \eqref{galerkin-kernel-convergence}.
\end{proof}

We now turn to the tightness of the laws of $\{\xi_n\}_{n\ge1}$ in $D([0,T];H^{-\beta}(\mathbb T^2))$, where $\beta>3$. We first verify the Aldous condition using the estimates in Lemma~\ref{lem-galerkin-increments} and the energy bound \eqref{uniform}.

\begin{proposition}\label{aldous-condition}
     The sequence $\{\xi_n\}_n$ satisfies the Aldous condition in $H^{-\beta}(\mathbb{T}^2)$ for $\beta>3$.
\end{proposition}

\begin{proof}
In light of Lemma~\ref{adlous condition1}, for each $n\in \mathbb{N}$ and $\delta > 0$, we consider an arbitrary sequence $\{\tau_n\}$ of $\mathcal{F}_t$-stopping times satisfying $\tau_n+\delta\leq T$. By \eqref{Hn-equ-2}, the increment  $\xi_{n}(\tau_n+\delta)-\xi_{n}(\tau_n)$ can be decomposed as follows:
   \begin{align*}
\xi_{n}(\tau_n+\delta)-\xi_{n}(\tau_n)&=-\int_{\tau_n}^{\tau_n+\delta}\Pi_{n}(u_{n}(s)\cdot\nabla\xi_{n}(s))\,\mathrm{d}s\\
& \quad+\sum_{k \in \Gamma}\int_{\tau_n}^{\tau_n+\delta}\!\! \int_{|z|\leq1}\big[ {\rm e}^{-z \theta_k A_k^n} \xi_n(s-) - \xi_n(s-)\big]\,\tilde{N}_{k}(\mathrm{d}z,\mathrm{d}s)\\
&\quad+\sum_{k \in \Gamma}\int_{\tau_n}^{\tau_n+\delta}\!\! \int_{|z|\leq1} \big[ {\rm e}^{-z \theta_k A_k^n} \xi_n(s) - \xi_n(s) + z \theta_k A_k^n (\xi_n(s)) \big]\,\nu(\mathrm{d}z)\,\mathrm{d}s.
\end{align*}
We proceed to bound the $H^{-\beta}(\mathbb{T}^2)$-norm of each term individually.

\textbf{(1) Nonlinear drift term.}
Since $u_n$ is divergence-free, we can rewrite $u_n \cdot \nabla \xi_n = \nabla \cdot(u_n \xi_n)$, yielding
\begin{align*}
 \bigg\| \int_{\tau_n}^{\tau_n+\delta} \Pi_n(u_n(s) \cdot \nabla \xi_n(s)) \, \mathrm{d}s \bigg\|_{H^{-\beta}} &\leq  \int_{\tau_n}^{\tau_n+\delta} \| u_n(s) \cdot \nabla \xi_n(s) \|_{H^{-\beta}} \, \mathrm{d}s \\
&\leq   \int_{\tau_n}^{\tau_n+\delta} \| \nabla \cdot(u_n \xi_n)(s) \|_{H^{-\beta}} \, \mathrm{d}s\\
&\leq  \int_{\tau_n}^{\tau_n+\delta} \| u_n(s) \xi_n(s) \|_{H^{-\beta +1}} \, \mathrm{d}s.
\end{align*}
Since $\beta>3$, we have $\| u_n(s) \xi_n(s) \|_{H^{-\beta +1}}\leq\| u_n(s) \xi_n(s) \|_{H^{-\frac{1}{2}}}$. Applying Lemma \ref{fenkai}  with $s_1 = \frac{1}{2}$, $s_2 = 0$ and $d=2$ gives
\begin{equation}
\begin{aligned}\label{adlousnonlinear}
 \bigg\| \int_{\tau_n}^{\tau_n+\delta} \Pi_n(u_n(s) \cdot \nabla \xi_n(s)) \, \mathrm{d}s \bigg\|_{H^{-\beta}}
&\le C  \int_{\tau_n}^{\tau_n+\delta} \| u_n(s) \|_{H^{\frac{1}{2}}} \| \xi_n(s) \|_{L^2} \,\mathrm{d}s \\
&\le C   \int_{\tau_n}^{\tau_n+\delta} \| \xi_n(s) \|_{H^{-\frac{1}{2}}} \| \xi_n(s) \|_{L^2} \,\mathrm{d}s \\
&\le C  \delta \| \xi_0 \|_{L^2}^2,
\end{aligned}
\end{equation}
where in the last step we have used \eqref{uniform}.

\textbf{(2) L\'evy correction term.} By Minkowski's inequality,
\begin{align*}
L_n&:=\biggl\|\sum_{k\in\Gamma}\int_{\tau_n}^{\tau_n+\delta}\!\int_{|z|\le1}
\bigl[{\rm e}^{-z\theta_kA_k^n}\xi_n(s)-\xi_n(s)+z\theta_kA_k^n\xi_n(s)\bigr]\,\nu(\mathrm{d}z)\,\mathrm{d}s\biggr\|_{H^{-\beta}}\\
&\le\sum_{k\in\Gamma}\int_{\tau_n}^{\tau_n+\delta}\!\int_{|z|\le1}
\bigl\|{\rm e}^{-z\theta_kA_k^n}\xi_n(s)-\xi_n(s)+z\theta_kA_k^n\xi_n(s)\bigr\|_{H^{-\beta}}\,\nu(\mathrm{d}z)\,\mathrm{d}s.
\end{align*}
 Applying the second inequality in \eqref{galerkin-analytic-increments} with $t=-z\theta_k$ and $f=\xi_n(s)$ gives
\[
\bigl\|{\rm e}^{-z\theta_kA_k^n}\xi_n(s)-\xi_n(s)+z\theta_kA_k^n\xi_n(s)\bigr\|_{H^{-\beta}}
\le C_\beta z^2\theta_k^2\|\xi_n(s)\|_{L^2}.
\]
Combining this estimate with \eqref{uniform}, we obtain
\begin{equation}\label{Levy estimate}
\begin{aligned}
L_n
&\le C_\beta\sum_{k\in\Gamma}\theta_k^2
\int_{\tau_n}^{\tau_n+\delta}\!\int_{|z|\le1}z^2\|\xi_n(s)\|_{L^2}\,\nu(\mathrm{d}z)\,\mathrm{d}s\le C_\beta\delta\|\xi_0\|_{L^2}\int_{|z|\le1}z^2\,\nu(\mathrm{d}z),
\end{aligned}
\end{equation}
where we used $\sum_{k\in\Gamma}\theta_k^2=1$. The constant $C_\beta$ is independent of $n$, $\tau_n$, $\delta$, $\theta$ and its support $\Gamma$.

\textbf{(3) Stochastic integral term.}
The process $s\mapsto\mathbf1_{\{\tau_n<s\le\tau_n+\delta\}}$ is adapted and left-continuous, hence predictable. Since $\xi_n$ is an adapted c\`adl\`ag $H_n$-valued process, its left-limit process is also predictable. The first inequality in \eqref{galerkin-analytic-increments} and \eqref{uniform} ensure square integrability of the stopped integrands. Applying the compensated-Poisson isometry (see \cite[Theorem~1.14]{Eberle15}) in $H^{-\beta}$ to the integrands  gives
\begin{equation}\label{M0}
\begin{aligned}
M_n&:=\mathbb E\biggl\|\sum_{k\in\Gamma}\int_{\tau_n}^{\tau_n+\delta}\!\int_{|z|\le1}
\bigl[{\rm e}^{-z\theta_kA_k^n}\xi_n(s-)-\xi_n(s-)\bigr]\,\tilde N_k(\mathrm{d}z,\mathrm{d}s)\biggr\|_{H^{-\beta}}^2\\
&=\sum_{k\in\Gamma}\mathbb E\int_{\tau_n}^{\tau_n+\delta}\!\int_{|z|\le1}
\bigl\|{\rm e}^{-z\theta_kA_k^n}\xi_n(s)-\xi_n(s)\bigr\|_{H^{-\beta}}^2\,\nu(\mathrm{d}z)\,\mathrm{d}s.
\end{aligned}
\end{equation}
Since $\xi_n(s)\in H_n$, applying the first inequality in \eqref{galerkin-analytic-increments} with $t=-z\theta_k$ and $f=\xi_n(s)$,  we obtain
\[
\bigl\|{\rm e}^{-z\theta_kA_k^n}\xi_n(s)-\xi_n(s)\bigr\|_{H^{-\beta}}^2
\le C_\beta z^2\theta_k^2\|\xi_n(s)\|_{L^2}^2.
\]
By \eqref{uniform}, inserting the preceding estimate into \eqref{M0} therefore gives
\begin{equation}\label{martingale}
\begin{aligned}
M_n
&\le C_\beta\sum_{k\in\Gamma}\theta_k^2\,
\mathbb E\int_{\tau_n}^{\tau_n+\delta}\!\int_{|z|\le1}
z^2\|\xi_n(s)\|_{L^2}^2\,\nu(\mathrm{d}z)\,\mathrm{d}s\\
&\le C_\beta\delta\|\xi_0\|_{L^2}^2\int_{|z|\le1}z^2\,\nu(\mathrm{d}z),
\end{aligned}
\end{equation}
where we used $\sum_{k\in\Gamma}\theta_k^2=1$. The constant $C_\beta$ is independent of $n$, $\tau_n$, $\delta$, $\theta$ and its support $\Gamma$.

Combining \eqref{adlousnonlinear}, \eqref{Levy estimate} and \eqref{martingale}, and using $\|a+b+c\|^2\le3(\|a\|^2+\|b\|^2+\|c\|^2)$, gives
\[
\mathbb E\|\xi_n(\tau_n+\delta)-\xi_n(\tau_n)\|_{H^{-\beta}}^2
\le C_{\beta,\nu}
\Bigl(\delta^2\|\xi_0\|_{L^2}^4
+\delta^2\|\xi_0\|_{L^2}^2
+\delta\|\xi_0\|_{L^2}^2\Bigr).
\]
Lemma~\ref{adlous condition1} gives the Aldous condition.
\end{proof}

Fix $\beta>3$ and set
\[
Z:=(Z^k)_{k\in\Gamma},
\quad
\mathcal Y
:=D\bigl([0,T];H^{-\beta}(\mathbb T^2)
\times\mathbb R^{\#\Gamma}\bigr),
\]
equipped with the $J_1$ topology.
For each $n\ge1$, let $P_n$ be the joint law of
$(\xi_n,Z)$ on $\mathcal Y$. We next prove tightness of $\{P_n\}_{n\ge1}$ on $\mathcal Y$
using Proposition~\ref{aldous-condition},
the corresponding estimates for $Z$,
and the compactness criterion in Lemma~\ref{compact1}.

\begin{proposition}\label{compact2}
The family $\{P_n\}_{n\ge1}$ is tight on $\mathcal Y$.
In particular, the laws of $\{\xi_n\}_{n\ge1}$ are tight
in $D([0,T];H^{-\beta}(\mathbb T^2))$.
\end{proposition}

\begin{proof}
By Doob's inequality, the compensated-Poisson isometry
and \eqref{uniform},
\[
\begin{aligned}
\mathbb E\sup_{t\le T}
\Bigl(\|\xi_n(t)\|_{L^2}^2+|Z(t)|^2\Bigr)
&\le
\|\xi_0\|_{L^2}^2
+4T\,\#\Gamma\int_{|z|\le1}z^2\,\nu(\mathrm dz)
\le C,
\end{aligned}
\]
where $C$ is independent of $n$.
For $\varepsilon>0$, set $R_\varepsilon=\sqrt{2C/\varepsilon}$
and define
\[
B_\varepsilon:=
\Big\{
(f,g)\in D\bigl([0,T];L^2(\mathbb T^2)
\times\mathbb R^{\#\Gamma}\bigr):
\sup_{t\le T}
\bigl(\|f(t)\|_{L^2}^2+|g(t)|^2\bigr)
\le R_\varepsilon^2
\Big\}.
\]
We regard $B_\varepsilon$ as a subset of $\mathcal Y$.
Chebyshev's inequality gives
\begin{equation}\label{B}
\mathbb P\bigl((\xi_n,Z)\in B_\varepsilon\bigr)
\ge1-\frac{\varepsilon}{2}.
\end{equation}
The compact embedding
$L^2(\mathbb T^2)\hookrightarrow H^{-\beta}(\mathbb T^2)$
and finiteness of $\Gamma$ imply that
\[
\bigl\{(f(t),g(t)):(f,g)\in B_\varepsilon\bigr\}
\]
has compact closure for every $t\in[0,T]$.

We next verify the Aldous condition for the joint processes.  Let $\tau_n$ be an $(\mathcal F_t)$-stopping time and
$\delta\ge0$ satisfy $\tau_n+\delta\le T$.
The stopped compensated-Poisson isometry gives
\[
\mathbb E|Z(\tau_n+\delta)-Z(\tau_n)|^2
=
\delta\,\#\Gamma
\int_{|z|\le1}z^2\,\nu(\mathrm dz).
\]
Together with Proposition~\ref{aldous-condition}, this yields
\[
\mathbb E\Bigl[
\|\xi_n(\tau_n+\delta)-\xi_n(\tau_n)\|_{H^{-\beta}}^2
+|Z(\tau_n+\delta)-Z(\tau_n)|^2
\Bigr]
\le C'\delta,
\]
uniformly in $n$, $\tau_n$ and $\delta$.
Thus the joint processes satisfy the Aldous condition.

By Lemma~\ref{compact-Aldous-1}, there exists
$A_\varepsilon\subset\mathcal Y$ such that
\begin{equation}\label{A}
\mathbb P\bigl((\xi_n,Z)\in A_\varepsilon\bigr)
\ge1-\frac{\varepsilon}{2}
\end{equation}
and
\[
\lim_{\delta\downarrow0}
\sup_{(f,g)\in A_\varepsilon}
\omega_{[0,T],H^{-\beta}(\mathbb T^2)\times\mathbb R^{\#\Gamma}}
\bigl((f,g),\delta\bigr)=0.
\]
Hence Lemma~\ref{compact1} implies that
\[
K_\varepsilon
:=\overline{A_\varepsilon\cap B_\varepsilon}
\]
is compact. By \eqref{B} and \eqref{A},
\[
P_n(K_\varepsilon)\ge1-\varepsilon,
\quad n\ge1.
\]
This proves tightness on $\mathcal Y$.
The marginal tightness follows from continuity of the
first-coordinate projection.
\end{proof}

Since $\mathcal Y$ is Polish, Prokhorov's theorem yields
a subsequence, not relabeled, such that
\[
P_n\Rightarrow P_0
\quad\text{on }\mathcal Y
\]
for some probability measure $P_0$.

By Skorohod's representation theorem, we can find a new probability space $(\Omega', \mathcal{F}', \mathbb{P}')$, together with processes $(\xi'_n, Z'_n )$ and $(\xi', Z')$, satisfying the following properties:
\begin{enumerate}
   \item[(1)]
    $\mathcal L_{\mathbb P'}(\xi'_n,Z'_n)=P_n$ for every $n\ge1$,
    and $\mathcal L_{\mathbb P'}(\xi',Z')=P_0$;
    \item[(2)] $\mathbb{P}'$-a.s., $(\xi'_n, Z'_n ) \to (\xi', Z')$ as $n\to \infty$ in the topology of $ \mathcal{Y}$.
\end{enumerate}
In particular, for every $n\ge1$, $Z'_n$ has the same law as $Z$ and consists of independent L\'evy processes, and we have
  \begin{equation}\label{Levy-processes}
  Z_n^{\prime k}(t)= \int_0^t \int_{|z|\le 1} z\, \tilde N_k^{\prime n}(\mathrm{d}z, \mathrm{d}s),
  \end{equation}
where $\tilde N_k^{\prime n}(\mathrm{d}z, \mathrm{d}s)= N_k^{\prime n}(\mathrm{d}z, \mathrm{d}s) - \nu(\mathrm{d}z)\mathrm{d}s$ is the compensated Poisson random measure corresponding to $Z_n^{\prime k}(t)$. For notational simplicity, we omit the prime and denote these sequences simply as $(\xi_n, Z_n )$ and $(\xi, Z)$ in what follows, but we will keep the notations $\omega'\in \Omega'$ and $\mathbb{P}'$ to remind the reader that we are working on a new probability space.

We use the completed, right-continuous natural filtration of each pair $(\xi_n,Z_n)$ separately.

\begin{lemma}\label{lem-skorohod-galerkin-transfer}
For each $n$, let $(\mathcal F_t^n)$ be the completed, right-continuous natural filtration of $(\xi_n,Z_n)$. Then the coordinates $Z_n^k$, $k\in\Gamma$, are independent $(\mathcal F_t^n)$-L\'evy processes with the prescribed characteristics, and their jump measures $N_k^n$ have compensators $\nu(\mathrm{d}z)\,\mathrm{d}s$. In this filtration, $\xi_n$ satisfies \eqref{Hn-equ-2} with $\widetilde N_k$ replaced by $\widetilde N_k^n$. The equation and the energy bound \eqref{uniform} hold on some event of $\mathbb{P}'$-probability one for all $n$ and $t\in[0,T]$.
\end{lemma}
\begin{proof}
Fix $n$. Equality of joint laws gives, $\mathbb P'$-a.s.,
\begin{equation}\label{copied-galerkin-bounds}
\xi_n\in D([0,T];H_n),\quad
\xi_n(0)=\Pi_n\xi_0,\quad
\sup_{t\le T}\|\xi_n(t)\|_{L^2}\le\|\xi_0\|_{L^2}.
\end{equation}
Since $\mathcal L_{\mathbb P'}(\xi_n,Z_n)=P_n$
for every $n\ge1$, future increments of $Z_n$
remain independent of the past of $(\xi_n,Z_n)$. This independence remains valid under the usual augmentation of filtrations, by right continuity; see the proof of \cite[Proposition~2.11]{dBHO18}. Thus the stated L\'evy properties and compensators hold in $(\mathcal F_t^n)$.

The integrability assumptions of
\cite[Lemma~4.5]{dBHO18} follow from
\eqref{copied-galerkin-bounds}, finite dimensionality
and the L\'evy measure condition.
Applying this lemma and using c\`adl\`ag paths and a
countable intersection over $n$, we conclude that the
copied processes satisfy \eqref{Hn-equ-2} with
$\widetilde N_k^n$ on a common event of $\mathbb{P}'$-probability one
for all $n\ge1$ and $t\in[0,T]$.
\end{proof}

Next, we want to show that the $\mathbb P'$-a.s. convergence $\xi_n\to \xi$ in (2) can be improved from $H^{-\beta}(\mathbb{T}^2)$ to $H^{-\varepsilon}(\mathbb{T}^2)$, thanks to the uniform boundedness of $\{\xi_n\}_n$ in ${L^{\infty}([0,T]; L^2(\mathbb{T}^2))}$.

\begin{proposition}\label{qianrun}
For $\mathbb P'$-a.s. $\omega'\in \Omega'$, $\xi\in L^{\infty}([0,T]; L^2(\mathbb{T}^2))$ with $\|\xi(t)\|_{L^2}\le \|\xi_0\|_{L^2}$, and for any $0<\varepsilon<\beta$,
    $\xi_n \to \xi$ in $D([0,T]; H^{-\varepsilon}(\mathbb{T}^2)).$
\end{proposition}
\begin{proof}
For $0<\varepsilon<\beta$, we obtain the assertions by applying  Proposition~\ref{Skorohod-interpolation}-(ii) with $E_0=L^2$, $E_\alpha=H^{-\varepsilon}$, $E_1=H^{-\beta}$ and $\alpha=\varepsilon/\beta$, using Lemma~\ref{neicha} and \eqref{uniform}.
\end{proof}

The following corollary yields the convergence of the corresponding velocity fields.
\begin{corollary}\label{velocity-convergence}
Let $u_n = -\nabla^\perp (-\Delta)^{-1}\xi_n$ and $u = -\nabla^\perp (-\Delta)^{-1}\xi$ denote the corresponding velocity fields obtained via the Biot-Savart operator. Then $u_n \to u$  $\mathbb P'$-a.s.  in $D([0,T]; H^{1 - \varepsilon}(\mathbb{T}^2))$.
\end{corollary}

We also record the convergence in $L^p([0,T])$ needed below.
\begin{proposition}\label{L1L2}
$\mathbb P'$-a.s., $\xi_n\to\xi$ in $L^p([0,T];H^{-\varepsilon}(\mathbb T^2))$ for every $1\le p<\infty$.
\end{proposition}
\begin{proof}
Apply Proposition~\ref{Skorohod-interpolation}-(i) with $E=H^{-\varepsilon}(\mathbb T^2)$.
\end{proof}

\subsection{Proof of weak existence}\label{Proof of main results}

This section is devoted to providing the
\begin{proof}[Proof of Theorem~\ref{thm:weak_solution}]

By Lemma~\ref{lem-skorohod-galerkin-transfer}, the processes $({\xi}_{n}, Z_{n})$ on the new probability space $({\Omega}', {\mathcal{F}}', {\mathbb{P}'})$ satisfy equation \eqref{Hn-equ-2}; as a result, for any $\phi \in C^\infty(\mathbb{T}^2)$, ${\mathbb{P}'}$-a.s. for all $t \in [0,T]$, it holds

\begin{equation*}
\begin{aligned}
\langle \xi_{n}(t), \phi \rangle &= \langle \xi_n(0), \phi \rangle - \int_{0}^{t} \big\langle \Pi_{n}(u_{n}(s) \cdot \nabla \xi_{n}(s)), \phi \big\rangle \,\mathrm{d}s \\
&\quad + \sum_{k\in\Gamma}\int_{0}^{t}\! \int_{|z| \leq 1}  \big\langle {\rm e}^{-z \theta_k A_k^n} \xi_n(s-) - \xi_n(s-), \phi  \big\rangle \,\tilde N_k^n(\mathrm{d}z,\mathrm{d}s)\\
&\quad + \sum_{k\in\Gamma} \int_{0}^{t}\! \int_{|z| \leq 1}  \big\langle {\rm e}^{-z \theta_k A_k^n} \xi_n(s) - \xi_n(s) + z \theta_k A_k^n (\xi_n(s)), \phi  \big\rangle \,\nu(\mathrm{d}z) \,\mathrm{d}s,
\end{aligned}
\end{equation*}
where $\tilde N_k^n(\mathrm{d}z,\mathrm{d}s)$ is the Poisson martingale measure corresponding to $Z_{n}^k$, see \eqref{Levy-processes}. Integrating by parts and using the skew-adjointness of $A_k^n$, we arrive at
\begin{equation}\label{limit equation-2}
\begin{aligned}
\langle \xi_{n}(t), \phi \rangle &= \langle \xi_n(0), \phi \rangle + \int_{0}^{t} \big\langle \xi_n(s), u_{n}(s) \cdot\nabla(\Pi_{n}\phi) \big\rangle \,\mathrm{d}s \\
&\quad + \sum_{k\in\Gamma} \int_{0}^{t}\! \int_{|z| \leq 1} \big\langle \xi_n(s-),{\rm e}^{z \theta_k A_k^n} \Pi_n\phi - \Pi_n\phi  \big\rangle \,\tilde N_k^n(\mathrm{d}z,\mathrm{d}s)\\
&\quad + \sum_{k\in\Gamma} \int_{0}^{t}\! \int_{|z| \leq 1} \big\langle \xi_n(s),{\rm e}^{z \theta_k A_k^n} \Pi_n\phi - \Pi_n\phi -z \theta_k A_k^n (\Pi_n\phi) \big\rangle \,\nu(\mathrm{d}z) \,\mathrm{d}s.
\end{aligned}
\end{equation}

By the joint Skorohod convergence, a standard passage to the limit
in the independence-of-increments identities, followed by the
augmentation argument used in Lemma~\ref{lem-skorohod-galerkin-transfer},
shows that $Z$ has independent L\'evy coordinates with the prescribed
characteristics in the completed, right-continuous natural filtration
of $(\xi,Z)$. In particular,
\begin{equation}\label{Levy-processes-1}
Z^k(t)=\int_0^t\!\int_{|z|\le1}z\,\tilde N_k(\mathrm{d}z,\mathrm{d}s),
\quad
\tilde N_k(\mathrm{d}z,\mathrm{d}s)=N_k(\mathrm{d}z,\mathrm{d}s)-\nu(\mathrm{d}z)\,\mathrm{d}s,
\quad k\in\Gamma.
\end{equation}
Since $\xi$ is adapted and c\`adl\`ag in $H^{-\beta}$,
$\xi_-=\{\xi(s-)\}_{s\in(0,T]}$ is predictable. All stochastic integrals below are taken
in the respective filtrations of $(\xi_n,Z_n)$ and $(\xi,Z)$.

We now pass to the limit term by term in \eqref{limit equation-2}
in $L^1(\Omega';L^1([0,T]))$
to obtain \eqref{identified-weak-equation} below.

\textbf{Step 1: Convergence of $\xi_n$.}
 By Proposition~\ref{L1L2}, ${\mathbb{P}'}$-a.s. ${\xi}_{n}$ converges in $L^{1}([0,T]; H^{-\varepsilon}(\mathbb{T}^2))$ to ${\xi}$. Therefore,
\[
\begin{aligned}
\int_{0}^{T} \big|\langle \xi_{n}(s), \phi \rangle - \langle \xi(s), \phi \rangle\big| \,\mathrm{d}s &= \int_{0}^{T} \big|\langle \xi_{n}(s) - \xi(s), \phi \rangle\big| \,\mathrm{d}s \\
&\leq \int_{0}^{T} \|\xi_{n}(s) - \xi(s) \|_{H^{-\varepsilon}} \|\phi\|_{H^{\varepsilon}} \,\mathrm{d}s \to0.
\end{aligned}
\]
Furthermore, by Proposition \ref{xi-Linf-L2} we have
\[
|\langle {\xi}_{n}(t), \phi \rangle| \leq \|\xi_0\|_{L^2} \|\phi\|_{L^2} \quad {\mathbb{P}'}\text{-a.s. for all } t \in [0,T].
\]
The dominated convergence theorem yields
\[
\lim_{n\to\infty} \mathbb{E}_{\mathbb{P}'} \bigg[\int_0^T \big|\langle \xi_n(s) - \xi(s), \phi\rangle\big| \,\mathrm{d}s\bigg] = 0.
\]

\textbf{Step 2: Convergence of the nonlinear term.} The treatment of this is by now standard; we provide the proof here for completeness.
Consider the nonlinear term
\begin{equation*}
\begin{aligned}
J_n&:= \mathbb{E}_{\mathbb{P}'} \bigg[\int_0^T \Big|\int_0^t \langle \xi_n(s), u_n(s) \cdot \nabla (\Pi_n \phi) \rangle \,\mathrm{d}s - \int_0^t \langle \xi(s), u(s) \cdot \nabla \phi \rangle \,\mathrm{d}s\Big| \,\mathrm{d}t \bigg]\\
&\leq \mathbb{E}_{\mathbb{P}'} \bigg[\int_0^T \Big|\int_0^t \big\langle \xi_n(s), u_n(s) \cdot \nabla (\Pi_n \phi) - u(s) \cdot \nabla \phi \big\rangle \,\mathrm{d}s\Big| \,\mathrm{d}t \bigg] \\
&\quad + \mathbb{E}_{\mathbb{P}'} \bigg[\int_0^T \Big|\int_0^t \langle \xi_n(s) - \xi(s), u(s) \cdot \nabla \phi \rangle \,\mathrm{d}s\Big| \,\mathrm{d}t \bigg].
\end{aligned}
\end{equation*}
The Cauchy-Schwarz inequality provides
\begin{equation*}
\begin{aligned}
J_n &\leq T\mathbb{E}_{\mathbb{P}'}\bigg[\int_0^T  \|\xi_n(s)\|_{L^2} \|u_n(s) \cdot \nabla (\Pi_n \phi) - u(s) \cdot \nabla \phi\|_{L^2} \,\mathrm{d}s\bigg]\\
&\quad + C_\varepsilon T\mathbb{E}_{\mathbb{P}'}\bigg[\int_0^T \|\xi_n(s) - \xi(s)\|_{H^{-\varepsilon}} \|u(s) \cdot \nabla \phi\|_{H^{\varepsilon}} \,\mathrm{d}s \bigg]\\
&\leq T \mathbb{E}_{\mathbb{P}'}\bigg[ \int_{0}^{T} \|\xi_{0}\|_{L^{2}} \big(\|u_{n}(s)\|_{L^{2}} \|\nabla(\Pi_{n}\phi) - \nabla\phi\|_{L^{\infty}} + \|u_{n}(s) - u(s)\|_{L^{2}} \|\nabla\phi\|_{L^{\infty}}\big) \,\mathrm{d}s\bigg] \\
&\quad +C_\varepsilon T \mathbb{E}_{\mathbb{P}'}\bigg[ \int_{0}^{T} \|\xi_{n}(s) - \xi(s)\|_{H^{-\varepsilon}} \|u(s)\|_{H^{1-\varepsilon}} \|\nabla\phi\|_{H^{2\varepsilon}} \,\mathrm{d}s\bigg],
\end{aligned}
\end{equation*}
where the last inequality follows from  Lemma \ref{fenkai}. Note that the convergence $u_{n} \to u$ $\mathbb{P}'$-a.s. in $D([0,T]; H^{1-\varepsilon}(\mathbb{T}^2))$ implies $u_{n} \to u$ in $D([0,T]; L^{2}(\mathbb{T}^2))$, hence $u_{n} \to u$ in $L^1([0,T]; L^{2}(\mathbb{T}^2))$. By Sobolev embedding, $\|\nabla(\Pi_{n}\phi) - \nabla\phi\|_{L^{\infty}} \le \|\Pi_{n}\phi - \phi\|_{H^{2+\varepsilon}}$ which vanishes as $n \to \infty$. Combining these facts with the $\P'$-a.s. convergence of $\xi_{n}$ to $\xi$ in $L^1([0,T]; H^{-\varepsilon}(\mathbb{T}^2))$ and the uniform bound $\sup_{s\in[0,T]}\|u(s)\|_{H^{1-\varepsilon}}\le C\sup_{s\in[0,T]}\|\xi(s)\|_{L^2} \le C\|\xi_0\|_{L^2}$, we conclude that $J_n \to 0$ as $n\to \infty$. \smallskip

\textbf{Step 3: Convergence of stochastic integrals.}
Fix $k\in\Gamma$ and use the same $\beta>3$ as in the joint Skorohod representation.
We work with the respective completed, right-continuous natural filtrations of $(\xi_n,Z_n)$ and $(\xi,Z)$ introduced above. For $|z|\le1$, set
\[
G_n^k(z):=\bigl({\rm e}^{z\theta_kA_k^n}-I\bigr)\Pi_n\phi,\quad G^k(z):=\phi\bigl(\varphi_{z\theta_k}^k\bigr)-\phi,
\]
and define
\begin{equation}\label{eq-marting-0}
\begin{aligned}
M_n^k(t)&:=\int_0^t\!\int_{|z|\le1}\langle\xi_n(s-),G_n^k(z)\rangle\,\tilde N_k^n(\mathrm{d}z,\mathrm{d}s),\\
\overline M_n^k(t)&:=\int_0^t\!\int_{|z|\le1}\langle\xi_n(s-),G^k(z)\rangle\,\tilde N_k^n(\mathrm{d}z,\mathrm{d}s),\\
M^k(t)&:=\int_0^t\!\int_{|z|\le1}\langle\xi(s-),G^k(z)\rangle\,\tilde N_k(\mathrm{d}z,\mathrm{d}s).
\end{aligned}
\end{equation}
The integrands are predictable and square-integrable by Lemma~\ref{lem-skorohod-galerkin-transfer}, \eqref{copied-galerkin-bounds} and \eqref{R-n-prime}.

By \eqref{galerkin-kernel-convergence}, for every $|z|\le1$,
\[
\|G_n^k(z)-G^k(z)\|_{L^2}^2\to0.
\]
Since $\|G_n^k(z)-G^k(z)\|_{L^2}^2\le C_\phi z^2\theta_k^2$ and $\int_{|z|\le1}z^2\,\nu(\mathrm{d}z)<\infty$, dominated convergence gives
\[
\int_{|z|\le1}\|G_n^k(z)-G^k(z)\|_{L^2}^2\,\nu(\mathrm{d}z)\to0.
\]
Doob's inequality, the compensated-Poisson isometry (see \cite[Theorem~1.14]{Eberle15}) and the Cauchy--Schwarz inequality lead to
\begin{equation}\label{kernel-replacement-error}
\begin{aligned}
\mathbb E_{\mathbb P'}\sup_{t\le T}\bigl|M_n^k(t)-\overline M_n^k(t)\bigr|^2&
\le4\mathbb E_{\mathbb P'}\int_0^T\!\int_{|z|\le1}
\bigl|\langle\xi_n(s),G_n^k(z)-G^k(z)\rangle\bigr|^2\,\nu(\mathrm{d}z)\,\mathrm{d}s\\
&\le4T\|\xi_0\|_{L^2}^2\int_{|z|\le1}
\bigl\|G_n^k(z)-G^k(z)\bigr\|_{L^2}^2\,\nu(\mathrm{d}z)\to0.
\end{aligned}
\end{equation}
where in the second step we have used the uniform bound in \eqref{copied-galerkin-bounds}.

We next compare $\overline M_n^k$ and $M^k$ by truncating the small jumps. Fix $\delta\in(0,1)$ such that $\nu(\{z: |z|=\delta\})=0$; then, $\mathbb P'$-a.s., $Z^k$ has no jumps of size $\delta$ on $[0,T]$. Define
\[
\begin{aligned}
\overline M_n^{k,\delta}(t)&:=\int_0^t\!\int_{\delta<|z|\le1}\langle\xi_n(s-),G^k(z)\rangle\,\tilde N_k^n(\mathrm{d}z,\mathrm{d}s),\\
M^{k,\delta}(t)&:=\int_0^t\!\int_{\delta<|z|\le1}\langle\xi(s-),G^k(z)\rangle\,\tilde N_k(\mathrm{d}z,\mathrm{d}s).
\end{aligned}
\]
We shall show that $\overline M_n^{k,\delta} \to M^{k,\delta}$ in probability in $D([0,T];\R)$ with the $J_1$-topology. To this end, let $\tau_1^n<\cdots<\tau_{J_n^\delta}^n$ and $\tau_1<\cdots<\tau_{J^\delta}$ be the jump times in $(0,T]$ of $Z_n^k$ and $Z^k$, respectively, whose absolute jump sizes are greater than $\delta$. Since $\nu(\{z:\delta<|z|\le1\})<\infty$, $J_n^\delta$ and $J^\delta$ are $\mathbb P'$-a.s. finite, and
\[
\begin{aligned}
\overline M_n^{k,\delta}(t)
&=\sum_{j=1}^{J_n^\delta} \big\langle\xi_n(\tau_j^n-), G^k(\Delta Z_n^k(\tau_j^n)) \big\rangle\,\mathbf 1_{\{\tau_j^n\le t\}}-\int_0^t\!\int_{\delta<|z|\le1} \langle\xi_n(s),G^k(z) \rangle\,\nu(\mathrm{d}z)\,\mathrm{d}s,\\
M^{k,\delta}(t)
&=\sum_{j=1}^{J^\delta} \big\langle\xi(\tau_j-), G^k(\Delta Z^k(\tau_j)) \big\rangle\,\mathbf 1_{\{\tau_j\le t\}}-\int_0^t\!\int_{\delta<|z|\le1}\langle\xi(s),G^k(z)\rangle\,\nu(\mathrm{d}z)\,\mathrm{d}s.
\end{aligned}
\]
Now, for $\P'$-a.s. $\omega'\in \Omega'$, let $\lambda_n= \lambda_n(\omega')$ be the time changes realizing the joint Skorohod convergence, i.e. $(\xi_n,Z_n)\circ\lambda_n\to(\xi,Z)$ uniformly on $[0,T]$. Then for $t\in[0,T]$,
\begin{equation}\label{eq-marting-cut-off}
\aligned
\overline M_n^{k,\delta}(\lambda_n(t))-M^{k,\delta}(t)
&=I_n^{k,\delta}(t)-\int_0^{\lambda_n(t)}\!\int_{\delta<|z|\le1}\langle\xi_n(s)-\xi(s),G^k(z)\rangle\,\nu(\mathrm{d}z)\,\mathrm{d}s\\
&\quad-\int_t^{\lambda_n(t)}\!\int_{\delta<|z|\le1}\langle\xi(s),G^k(z)\rangle\,\nu(\mathrm{d}z)\,\mathrm{d}s,
\endaligned
\end{equation}
where
\[
\begin{aligned}
I_n^{k,\delta}(t)
:=\sum_{j=1}^{J_n^\delta} \big\langle\xi_n(\tau_j^n-), G^k(\Delta Z_n^k(\tau_j^n)) \big\rangle\,\mathbf 1_{\{\tau_j^n\le\lambda_n(t)\}} -\sum_{j=1}^{J^\delta} \big\langle\xi(\tau_j-), G^k(\Delta Z^k(\tau_j)) \big\rangle\,\mathbf 1_{\{\tau_j\le t\}}.
\end{aligned}
\]
Applying Lemma~\ref{lem-jump-matching} to $(\xi_n,Z_n^k)$, we obtain, $\mathbb P'$-a.s. for all sufficiently large $n$,
\[
J_n^\delta=J^\delta,\quad \tau_j^n=\lambda_n(\tau_j),\quad j=1,\ldots,J^\delta.
\]
Since $\lambda_n$ is strictly increasing, $\mathbf 1_{\{\tau_j^n\le\lambda_n(t)\}}=\mathbf 1_{\{\tau_j\le t\}}$. Hence, for these $n$,
\[
\begin{aligned}
I_n^{k,\delta}(t)
&=\sum_{j=1}^{J^\delta}\bigl\langle\xi_n(\tau_j^n-),G^k(\Delta Z_n^k(\tau_j^n))-G^k(\Delta Z^k(\tau_j))\bigr\rangle\,\mathbf 1_{\{\tau_j\le t\}}\\
&\quad+\sum_{j=1}^{J^\delta}\bigl\langle\xi_n(\tau_j^n-)-\xi(\tau_j-),G^k(\Delta Z^k(\tau_j))\bigr\rangle\,\mathbf 1_{\{\tau_j\le t\}}.
\end{aligned}
\]
Taking left limits in the common time-change estimate gives, $\mathbb P'$-a.s.,
\[
\|\xi_n(\tau_j^n-)-\xi(\tau_j-)\|_{H^{-\beta}}\le\sup_{s\in[0,T]}\|\xi_n(\lambda_n(s))-\xi(s)\|_{H^{-\beta}}\to0,\quad j=1,\ldots,J^\delta.
\]
Similarly, the corresponding jump sizes satisfy
\[
|\Delta Z_n^k(\tau_j^n)-\Delta Z^k(\tau_j)|\le2\sup_{s\in[0,T]}|Z_n^k(\lambda_n(s))-Z^k(s)|\to0.
\]
By the fundamental theorem of calculus,
\[
\bigl\|G^k(\Delta Z_n^k(\tau_j^n)) -G^k(\Delta Z^k(\tau_j)) \bigr\|_{L^2} \le|\theta_k|\, \|\nabla\phi\|_{L^\infty} |\Delta Z_n^k(\tau_j^n)-\Delta Z^k(\tau_j)|\to0.
\]
By \eqref{flow-expression} and the smoothness of $\phi$, the map $z\mapsto G^k(z)$ is continuous from $[-1,1]$ into $H^\beta$, and $\sup_{|z|\le1}\|G^k(z)\|_{H^\beta}<\infty$. Thus, by \eqref{copied-galerkin-bounds} and the above estimates, we obtain, $\mathbb P'$-a.s. for all sufficiently large $n$,
\[
\begin{aligned}
\sup_{t\le T}|I_n^{k,\delta}(t)|
&\le\|\xi_0\|_{L^2}\sum_{j=1}^{J^\delta}\bigl\|G^k(\Delta Z_n^k(\tau_j^n))-G^k(\Delta Z^k(\tau_j))\bigr\|_{L^2}\\
&\quad+\sum_{j=1}^{J^\delta}\|\xi_n(\tau_j^n-)-\xi(\tau_j-)\|_{H^{-\beta}}\bigl\|G^k(\Delta Z^k(\tau_j))\bigr\|_{H^\beta}\\
&\le2J^\delta\|\xi_0\|_{L^2}|\theta_k|\, \|\nabla\phi\|_{L^\infty} \sup_{s\in[0,T]}|Z_n^k(\lambda_n(s))-Z^k(s)|\\
&\quad+J^\delta\biggl(\sup_{|z|\le1}\|G^k(z)\|_{H^\beta}\biggr)\sup_{s\in[0,T]}\|\xi_n(\lambda_n(s))-\xi(s)\|_{H^{-\beta}}\to0,
\end{aligned}
\]
since, $\mathbb P'$-a.s., $J^\delta<\infty$ and both uniform errors tend to zero. Thus
\[
\mathbb P'\text{-a.s.}\quad\sup_{t\le T}|I_n^{k,\delta}(t)|\to0.
\]

Now we return to \eqref{eq-marting-cut-off}. By Sobolev duality, we have $\mathbb P'\text{-a.s.}$
\[
\begin{aligned}
\sup_{t\le T}\bigl|\overline M_n^{k,\delta}(\lambda_n(t))-M^{k,\delta}(t)\bigr|
&\le\sup_{t\le T}|I_n^{k,\delta}(t)|\\
&\quad +\bigg(\int_{\delta<|z|\le1}\|G^k(z)\|_{H^\beta} \,\nu(\mathrm{d}z)\bigg) \int_0^T\|\xi_n(s)-\xi(s)\|_{H^{-\beta}} \,\mathrm{d}s\\
&\quad +\|\xi_0\|_{L^2}\bigg( \int_{\delta<|z|\le1}\|G^k(z)\|_{L^2}\,\nu(\mathrm{d}z) \bigg)\sup_{t\le T} |\lambda_n(t)-t| \to0.
\end{aligned}
\]
 Here the convergence follows from Proposition~\ref{Skorohod-interpolation}-(i), the joint Skorohod convergence and the preceding estimate for $I_n^{k,\delta}$.
Since $\sup_{t\le T}|\lambda_n(t)-t|\to0$, Proposition~\ref{conD}-(c) implies
\begin{equation}\label{eq-marting-cut-off-1}
\mathbb P'\text{-a.s.}\quad d_{J_1}\big(\overline M_n^{k,\delta}, M^{k,\delta} \big)\to0.
\end{equation}
In particular, $\overline M_n^{k,\delta}\to M^{k,\delta}$ in probability in $D([0,T];\mathbb R)$ with the $J_1$-topology.

Recalling the definitions of $\overline M_n^k(t)$ and $M^k(t)$ in \eqref{eq-marting-0}, we have by Doob's inequality, the compensated-Poisson isometry and \eqref{R-n-prime},
\[
\begin{aligned}
&\sup_n\mathbb E_{\mathbb P'}\sup_{t\le T}\bigl|\overline M_n^k(t)-\overline M_n^{k,\delta}(t)\bigr|^2
+\mathbb E_{\mathbb P'}\sup_{t\le T}\bigl|M^k(t)-M^{k,\delta}(t)\bigr|^2\\
&\quad\le C_\phi T\|\xi_0\|_{L^2}^2\theta_k^2\int_{|z|\le\delta}z^2\,\nu(\mathrm{d}z),
\end{aligned}
\]
where the constant is independent of $n$ and $\delta$. After increasing $C_\phi$ if necessary, the triangle inequality and Markov's inequality give, for every $\varepsilon>0$,
\[
\begin{aligned}
\mathbb P'\bigl( d_{J_1}\big(\overline M_n^k, M^k \big) >\varepsilon\bigr)
&\le\mathbb P'\bigl(d_{J_1} \big( \overline M_n^{k,\delta}, M^{k,\delta} \big)>\varepsilon/3\bigr) +\frac{9}{\varepsilon^2}\mathbb E_{\mathbb P'}\sup_{t\le T}\bigl|\overline M_n^k(t)-\overline M_n^{k,\delta}(t)\bigr|^2\\
&\quad+\frac{9}{\varepsilon^2}\mathbb E_{\mathbb P'}\sup_{t\le T}\bigl|M^k(t)-M^{k,\delta}(t)\bigr|^2\\
&\le\mathbb P'\bigl(d_{J_1}\big( \overline M_n^{k,\delta}, M^{k,\delta} \big)>\varepsilon/3\bigr)
+\frac{C_\phi T\|\xi_0\|_{L^2}^2\theta_k^2}{\varepsilon^2}\int_{|z|\le\delta}z^2\,\nu(\mathrm{d}z).
\end{aligned}
\]
For fixed $\delta$, the convergence \eqref{eq-marting-cut-off-1} of the truncated integrals yields
\[
\mathbb P'\bigl(d_{J_1}(\overline M_n^{k,\delta},M^{k,\delta})>\varepsilon/3\bigr)\to0
\quad\text{as }n\to\infty.
\]
Consequently,
\[
\limsup_{n\to\infty}\mathbb P'\bigl(d_{J_1} \big(\overline M_n^k, M^k \big)>\varepsilon\bigr)
\le\frac{C_\phi T\|\xi_0\|_{L^2}^2\theta_k^2}{\varepsilon^2}\int_{|z|\le\delta}z^2\,\nu(\mathrm{d}z).
\]
Since $\int_{|z|\le1}z^2\,\nu(\mathrm{d}z)<\infty$, the integral on the right tends to zero as $\delta\downarrow0$. Letting $\delta\downarrow0$ through values satisfying $\nu(\{z:|z|=\delta\})=0$, we obtain
\[
\overline M_n^k\to M^k
\quad\text{in probability in }D([0,T];\mathbb R)\text{ with }J_1\text{ topology}.
\]
Together with \eqref{kernel-replacement-error}, this yields
\[
d_{J_1}(M_n^k,M^k)
\le\sup_{t\le T} \big|M_n^k(t)- \overline M_n^k(t) \big|+d_{J_1}\big(\overline M_n^k, M^k \big)
\to0\quad\text{in probability}.
\]

By Proposition~\ref{Skorohod-interpolation}-(i),
\[
M_n^k\to M^k\quad\text{in probability in }L^1([0,T]).
\]
Moreover, Doob's inequality, the compensated-Poisson isometry, \eqref{R-n-prime} and \eqref{copied-galerkin-bounds} give
\[
\sup_n\mathbb E_{\mathbb P'}\sup_{t\le T}|M_n^k(t)|^2
+\mathbb E_{\mathbb P'}\sup_{t\le T}|M^k(t)|^2
\le C_\phi T\|\xi_0\|_{L^2}^2\theta_k^2\int_{|z|\le1}z^2\,\nu(\mathrm{d}z)<\infty.
\]
Consequently,
\[
\begin{aligned}
&\sup_n\mathbb E_{\mathbb P'}\Biggl[
\biggl(\int_0^T|M_n^k(t)-M^k(t)|\,\mathrm{d}t\biggr)^2\Biggr]\\
&\quad\le 2 T^2 \biggl(\sup_n\mathbb E_{\mathbb P'}\sup_{t\le T}|M_n^k(t)|^2
+\mathbb E_{\mathbb P'}\sup_{t\le T}|M^k(t)|^2\biggr)<\infty.
\end{aligned}
\]
Thus $\int_0^T|M_n^k(t)-M^k(t)|\,\mathrm{d}t$ is uniformly integrable. Together with its convergence in probability to zero, this yields
\begin{equation}\label{stochastic-integral-convergence}
\mathbb E_{\mathbb P'}\int_0^T|M_n^k(t)-M^k(t)|\,\mathrm{d}t
\to0.
\end{equation}
Since $\Gamma$ is fixed and finite, summing over $k\in\Gamma$ proves the required convergence of the stochastic terms in $L^1(\Omega';L^1([0,T]))$.

\textbf{Step 4: Convergence of the L\'evy correction term $V_n$.} We want to show that the last term in \eqref{limit equation-2} converges to the last one in \eqref{def-weak-solution.1}. Define
\[
\begin{aligned}
R_n^{k,z}&:={\rm e}^{z\theta_kA_k^n}\Pi_n\phi-\Pi_n\phi-z\theta_kA_k^n\Pi_n\phi,\\
R^{k,z}&:=\phi\bigl(\varphi_{z\theta_k}^k\bigr)-\phi-z\theta_k\sigma_k\cdot\nabla\phi\, ;
\end{aligned}
\]
recall that $\phi\big(\varphi^k_{z\theta_k} \big) = {\rm e}^{z\theta_k \sigma_k\cdot\nabla} \phi$ by \eqref{define e}. We have the decomposition
\[
\begin{aligned}
V_n&:=\mathbb E_{\mathbb P'}\sup_{t\le T}\biggl|\sum_{k\in\Gamma}\int_0^t\!\int_{|z|\le1}
\Bigl[\langle\xi_n(s),R_n^{k,z}\rangle-\langle\xi(s),R^{k,z}\rangle\Bigr]\,\nu(\mathrm{d}z)\,\mathrm{d}s\biggr|\\
&\le\mathbb E_{\mathbb P'}\biggl[\sum_{k\in\Gamma}\int_0^T\!\int_{|z|\le1}
\bigl|\langle\xi_n(s)-\xi(s),R^{k,z}\rangle\bigr|\,\nu(\mathrm{d}z)\,\mathrm{d}s\biggr]\\
&\quad+\mathbb E_{\mathbb P'}\biggl[\sum_{k\in\Gamma}\int_0^T\!\int_{|z|\le1}
\bigl|\langle\xi_n(s),R_n^{k,z}-R^{k,z}\rangle\bigr|\,\nu(\mathrm{d}z)\,\mathrm{d}s\biggr]\\
&=:V_n^{11}+V_n^{12}.
\end{aligned}
\]

By Lemma~\ref{lem-galerkin-estimates}, specifically the second
convergence in \eqref{galerkin-kernel-convergence},
\begin{equation}\label{R}
    R_n^{k,z}\to R^{k,z} \quad\text{in }L^2(\mathbb T^2) \quad\text{for every fixed }k,z.
\end{equation}
Moreover, the second estimate in \eqref{R-n-prime}, obtained from
\eqref{R_n-taylor} and \eqref{flow-L2-increments}, gives
\[
\|R_n^{k,z}\|_{L^2}+\|R^{k,z}\|_{L^2}
\le C_\phi z^2\theta_k^2,
\]
where $C_\phi$ is independent of $n,k,z$.

\medskip
\emph{Convergence of $V_n^{12}$.}
The Cauchy--Schwarz inequality and the energy bound \eqref{copied-galerkin-bounds}
give
\[
V_n^{12}\le T\|\xi_0\|_{L^2}
\sum_{k\in\Gamma}\int_{|z|\le1}
\|R_n^{k,z}-R^{k,z}\|_{L^2}\,\nu(\mathrm{d}z).
\]
For each fixed $k,z$, the integrand tends to zero by \eqref{R}, and it is bounded uniformly in $n$:
\[
\|R_n^{k,z}-R^{k,z}\|_{L^2}
\le\|R_n^{k,z}\|_{L^2}+\|R^{k,z}\|_{L^2}
\le C_\phi z^2\theta_k^2.
\]
Since $\Gamma$ is finite and
$\int_{|z|\le1}z^2\,\nu(\mathrm{d}z)<\infty$, dominated convergence yields
\begin{equation}\label{V12}
\lim_{n\to\infty}V_n^{12}=0.
\end{equation}

\medskip
\emph{Convergence of $V_n^{11}$.}
By Proposition~\ref{qianrun}, we have $\mathbb P'$-a.s. $\xi_n\to\xi$ in $D([0,T];H^{-\varepsilon}(\mathbb T^2))$ for every $0<\varepsilon<\beta$. Hence, by Proposition~\ref{D-con-0}, for a.e. $s\in[0,T]$, $\xi_n(s)$ converges to $\xi(s)$ in $H^{-\varepsilon}(\mathbb T^2)$. Combining this fact with the uniform boundedness of $\{\xi_n(s)\}_n$ in $L^2(\mathbb T^2)$, we deduce that, for a.e. $s\in[0,T]$, $\xi_n(s)$ converges weakly to $\xi(s)$ in $L^2(\mathbb T^2)$.

For all $k\in\Gamma$ and $|z|\le1$, $R^{k,z}$ belongs to $L^2(\mathbb T^2)$. Thus, $\mathbb P'$-a.s., for a.e. $s\in[0,T]$,
\[
\lim_{n\to\infty}\langle\xi_n(s)-\xi(s),R^{k,z}\rangle=0.
\]
Again by the Cauchy--Schwarz inequality,  \eqref{copied-galerkin-bounds} and \eqref{R-n-prime}, we have
\[
\bigl|\langle\xi_n(s)-\xi(s),R^{k,z}\rangle\bigr|
\le2\|\xi_0\|_{L^2}C_\phi z^2\theta_k^2.
\]
Since $\Gamma$ is finite and $\int_{|z|\le1}z^2\,\nu(\mathrm{d}z)<\infty$, the dominated convergence theorem yields
\[
\lim_{n\to\infty}V_n^{11}=0.
\]
Combining this limit with \eqref{V12} and $V_n\le V_n^{11}+V_n^{12}$, we finally conclude that
\begin{equation}\label{Vn1L}
\lim_{n\to\infty}V_n=0.
\end{equation}

As the correction terms have absolutely continuous paths, \eqref{Vn1L} also gives convergence of their difference to zero in $L^1(\Omega';C([0,T]))$. Its $L^1(\Omega';L^1([0,T]))$ norm is bounded by $TV_n$, so it converges to zero.

\textbf{Step 5: Passage to the limit.}
Fix $\phi\in C^\infty(\mathbb T^2)$. Since
$\xi_n(0)=\Pi_n\xi_0\to\xi_0$ in $L^2(\mathbb T^2)$,
we have
$\langle\xi_n(0),\phi\rangle\to\langle\xi_0,\phi\rangle$.
Combining all the convergence results above and passing
to a further subsequence if necessary, we can let
$n\to\infty$ in \eqref{limit equation-2} and obtain,
$\mathbb P'$-a.s. for a.e. $t\in[0,T]$,
\begin{equation}\label{identified-weak-equation}
\begin{aligned}
\langle\xi(t),\phi\rangle
&=\langle\xi_0,\phi\rangle
+\int_0^t\langle\xi(s),u(s)\cdot\nabla\phi\rangle\,\mathrm{d}s\\
&\quad+\sum_{k\in\Gamma}\int_0^t\!\int_{|z|\le1}
\langle\xi(s-),\phi(\varphi_{z\theta_k}^k)-\phi\rangle
\,\tilde N_k(\mathrm{d}z,\mathrm{d}s)\\
&\quad+\sum_{k\in\Gamma}\int_0^t\!\int_{|z|\le1}
\langle\xi(s),
\phi(\varphi_{z\theta_k}^k)-\phi
-z\theta_k\sigma_k\cdot\nabla\phi\rangle
\,\nu(\mathrm{d}z)\,\mathrm{d}s.
\end{aligned}
\end{equation}
The c\`adl\`ag property of both sides and the
$J_1$ convergence extend \eqref{identified-weak-equation} to all $t\in[0,T]$,
which completes the proof.
\end{proof}

\section{Scaling limit of stochastic 2D Euler equations}

Recall that we consider a sequence of stochastic 2D Euler equations \eqref{yilie2D}, in which the coefficients $\{\theta^n\}_n\subset\ell^2(\Z_0^2)$ satisfy Hypotheses~\ref{hypo-coefficients}. According to Theorem~\ref{thm:weak_solution}, for every $n\ge1$, equation \eqref{yilie2D} has a weak solution $\bigl(\xi^n,\{Z_k^n\}_k\bigr)$, where $\{Z_k^n\}_k$ is a family of independent L\'evy processes with the same L\'evy measure $\nu$. Moreover, by \eqref{estimate-2}, $\P$-a.s., $\|\xi^n(t)\|_{L^2}\le\|\xi_0\|_{L^2}$ for all $t\in[0,T]$. These weak solutions may be defined on different probability spaces, but for simplicity of notation we do not distinguish $\Omega$, $\P$ and $\E$.

By Definition~\ref{def-weak-solution}, for each $n\ge1$ and $\phi\in C^\infty(\mathbb T^2)$, $\P$-a.s. for all $t\in[0,T]$,
\begin{equation}\label{scaling equ}
\begin{aligned}
\langle\xi^n(t),\phi\rangle
&=\langle\xi_0,\phi\rangle+\int_0^t\langle\xi^n(s),u^n(s)\cdot\nabla\phi\rangle\,\mathrm{d}s\\
&\quad+\sum_{k\in\Gamma_n}\int_0^t\!\int_{|z|\le1}
\bigl\langle\xi^n(s-),\phi\bigl(\varphi_{z\theta_k^n}^k\bigr)-\phi\bigr\rangle\,\tilde N_k^n(\mathrm{d}z,\mathrm{d}s)\\
&\quad+\sum_{k\in\Gamma_n}\int_0^t\!\int_{|z|\le1}
\bigl\langle\xi^n(s),\phi\bigl(\varphi_{z\theta_k^n}^k\bigr)-\phi-z\theta_k^n\sigma_k\cdot\nabla\phi\bigr\rangle\,\nu(\mathrm{d}z)\,\mathrm{d}s,
\end{aligned}
\end{equation}
where $\Gamma_n=\{k\in \mathbb Z^2_0:\theta_k^n\neq0\}$ is the support of $\theta^n$.

Following the methodology of Section~\ref{com-tight}, we can derive in a similar way tightness results for the sequence of solutions $\big\{\big(\xi^n,\{Z_k^n\}_k\big)\big\}_n$. The arguments are summarized as follows. First, we prove that $\{\xi^n\}_n$ verifies the Aldous condition.

\begin{proposition}\label{prop-aldous-condition}
The sequence $\{\xi^n\}_n$ satisfies the Aldous condition in $H^{-\beta}(\mathbb{T}^2)$ for $\beta>3$.
\end{proposition}

\begin{proof}
Fix any $n\in \mathbb N$ and $\delta\in (0,1)$, let $\{\tau_n\}_n$ be a sequence of $\mathcal F_t$-stopping times such that $\tau_n +\delta\le T$. We have, by the definition of Sobolev norm,
  \begin{equation}\label{prop-aldous-condition.0}
  \E\big[\|\xi^n(\tau_n +\delta) - \xi^n(\tau_n) \|_{H^{-\beta}}^2 \big] = \sum_l(1+4\pi^2|l|^2)^{-\beta} \E \big\<\xi^n(\tau_n +\delta) - \xi^n(\tau_n), e_l \big\>^2 .
  \end{equation}
Substituting $\phi= e_l$ in \eqref{scaling equ}, we obtain
  \begin{equation}\label{prop-aldous-condition.1}
  \begin{aligned}
  \big\<\xi^n(\tau_n +\delta) - \xi^n(\tau_n), e_l \big\> &= \int_{\tau_n}^{\tau_n +\delta} \langle  \xi^{n}(s), u^{n}(s) \cdot \nabla e_l \rangle\, \mathrm{d}s \\
  &\quad + \sum_{k\in\Gamma_n} \int_{\tau_n}^{\tau_n +\delta}\!\! \int_{|z| \leq 1}  \big\langle \xi^n(s-), e_l\big(\varphi_{z\theta_k^n}^k \big) - e_l \big\rangle \, \tilde N_{k}^n(\mathrm{d}z,\mathrm{d}s)\\
  & \quad + \sum_{k\in\Gamma_n} \int_{\tau_n}^{\tau_n +\delta}\!\! \int_{|z| \leq 1} \big\langle \xi^{n}(s), e_l\big(\varphi_{z\theta_k^n}^k\big) - e_l -  z\theta^n_k \sigma_{k} \cdot\nabla e_l \big\rangle\, \nu(\mathrm{d}z)\mathrm{d}s,
  \end{aligned}
  \end{equation}
where the three terms on the right-hand side of \eqref{prop-aldous-condition.1} will be denoted as $I_i$, $i=1,2,3$.

By \eqref{estimate-2}, the Biot--Savart estimate and $\|\nabla e_l\|_{L^\infty}\le C|l|$, a.s.,
\[
|I_1|
\le\int_{\tau_n}^{\tau_n+\delta}\|\xi^n(s)\|_{L^2}\|u^n(s)\|_{L^2}\|\nabla e_l\|_{L^\infty}\,\mathrm{d}s
\le C\delta\|\xi_0\|_{L^2}^2\,|l|.
\]
We apply Corollary~\ref{lem-paired-jump-estimates} with $d=2$, $\phi=e_l$ and $\theta$ replaced by $\theta^n$. By \eqref{prelim-summed-compensator} and \eqref{estimate-2},
\[
\begin{aligned}
|I_3|
&\le C\|\nabla^2e_l\|_{L^\infty}
\biggl(\int_{|z|\le1}z^2\,\nu(\mathrm{d}z)\biggr)
\int_{\tau_n}^{\tau_n+\delta}\|\xi^n(s)\|_{L^2}\,\mathrm{d}s\\
&\le C\delta\|\xi_0\|_{L^2}\,|l|^2\int_{|z|\le1}z^2\,\nu(\mathrm{d}z).
\end{aligned}
\]

For $I_2$, the stopped integrands are predictable, since $\xi^n(s-)$ is predictable and $\mathbf1_{\{\tau_n<s\le\tau_n+\delta\}}$ is adapted and left-continuous. The compensated-Poisson isometry (see \cite[Theorem~1.14]{Eberle15}), orthogonality of distinct noise coordinates and \eqref{prelim-summed-noise} give
\[
\begin{aligned}
\E|I_2|^2
&=\E\int_{\tau_n}^{\tau_n+\delta}\sum_{k\in\Gamma_n}\int_{|z|\le1}
\bigl|\bigl\langle\xi^n(s),e_l\bigl(\varphi_{z\theta_k^n}^k\bigr)-e_l\bigr\rangle\bigr|^2\,\nu(\mathrm{d}z)\,\mathrm{d}s\\
&\le\|\nabla e_l\|_{L^\infty}^2
\biggl(\int_{|z|\le1}z^2\,\nu(\mathrm{d}z)\biggr)
\E\int_{\tau_n}^{\tau_n+\delta}\|\xi^n(s)\|_{L^2}^2\,\mathrm{d}s\\
&\le C\delta\|\xi_0\|_{L^2}^2\,|l|^2\int_{|z|\le1}z^2\,\nu(\mathrm{d}z).
\end{aligned}
\]
Since $\|\theta^n\|_{\ell^2}=1$, the constants in these estimates are independent of $n$, $l$, $\tau_n$, $\delta$ and the support $\Gamma_n$.

Combining the above estimates with \eqref{prop-aldous-condition.1}, we arrive at
  $$\E \big\<\xi^n(\tau_n +\delta) - \xi^n(\tau_n), e_l \big\>^2 \le C_\nu \delta \big(\|\xi_0 \|_{L^2}^4 + \|\xi_0 \|_{L^2}^2\big) |l|^4,  $$
where $C_\nu$ depends on $\int_{|z| \leq 1} z^2\, \nu(\mathrm{d}z)$. Inserting this estimate into \eqref{prop-aldous-condition.0}, we obtain
  $$\E\big[\|\xi^n(\tau_n +\delta) - \xi^n(\tau_n) \|_{H^{-\beta}}^2 \big] \le C_\nu \delta \big(\|\xi_0 \|_{L^2}^4 + \|\xi_0 \|_{L^2}^2\big) \sum_l \frac1{|l|^{2\beta-4}} \le C'\delta, $$
where $C'$ is a finite constant since $\beta>3$. We complete the proof by applying Lemma \ref{adlous condition1}.
\end{proof}

\begin{remark}
Alternatively, one may proceed as follows. Making the change of variables $y=\varphi_{z\theta_k^n}^k(x)$ and integrating by parts, we can rewrite \eqref{scaling equ} informally as follows:
\begin{equation*}
\begin{aligned}
\langle \xi^{n}(t), \phi \rangle &= \langle \xi_{0}, \phi \rangle - \int_{0}^{t} \langle u^{n}(s) \cdot \nabla\xi^{n}(s),\phi \rangle\, \mathrm{d}s \\
&\quad + \sum_{k\in\Gamma_n} \int_{0}^{t}\! \int_{|z| \leq 1}  \big\langle \xi^n\big(s-, \varphi_{-z\theta_k^n}^k \big) -\xi^n(s-), \phi \big\rangle \, \tilde N_{k}^n(\mathrm{d}z,\mathrm{d}s)\\
&\quad + \sum_{k\in\Gamma_n} \int_{0}^{t}\! \int_{|z| \leq 1} \big\langle \xi^{n}\big(s, \varphi_{-z\theta_k^n}^k \big) -\xi^n(s) + z\theta^n_k \sigma_{k} \cdot\nabla\xi^n(s), \phi \big\rangle\, \,\nu(\mathrm{d}z)\mathrm{d}s.
\end{aligned}
\end{equation*}
By the arbitrariness of $\phi\in C^\infty(\T^2)$, we deduce that the equation below holds in the weak sense:
\begin{equation*}
\begin{aligned}
\xi^{n}(t) &=\xi_{0} - \int_{0}^{t} u^{n}(s) \cdot \nabla\xi^{n}(s) \, \mathrm{d}s \\
&\quad + \sum_{k\in\Gamma_n} \int_{0}^{t}\! \int_{|z| \leq 1} \big[ \xi^n\big(s-, \varphi_{-z\theta_k^n}^k \big) -\xi^n(s-) \big] \, \tilde N_{k}^n(\mathrm{d}z,\mathrm{d}s)\\
&\quad + \sum_{k\in\Gamma_n} \int_{0}^{t}\! \int_{|z| \leq 1} \big[ \xi^{n}\big(s, \varphi_{-z\theta_k^n}^k \big) -\xi^n(s) + z\theta^n_k \sigma_{k} \cdot\nabla\xi^n(s) \big]\, \,\nu(\mathrm{d}z)\mathrm{d}s.
\end{aligned}
\end{equation*}
Using this equation, we can follow the arguments in Proposition \ref{aldous-condition} to show that $\{\xi^n \}_n$ satisfies the Aldous condition.
\end{remark}

With Proposition \ref{prop-aldous-condition} in hand, we can repeat the arguments in Proposition \ref{compact2} to show that the family of laws of $\{\xi^n\}_n$ is tight in $D([0,T]; H^{-\beta}(\mathbb{T}^2))$. An application of Skorohod's representation theorem then yields a new probability space $\big( \overline{\Omega}, \overline{\mathcal{F}}, \overline{\mathbb{P}} \big)$ supporting processes $\bar{\xi}^n$, $\bar{\xi}$, and a sequence $\{\bar{Z}^n\}_{n\ge 1}$ where each $\bar Z^n$ consists of independent L\'evy processes $\{\bar Z^n_k\}_k$ with the same L\'evy  measure $\nu$, such that
\begin{enumerate}
    \item[(1)] $\big(\bar{\xi}^n, \bar Z^n \big) \overset{\mathcal L}{=} (\xi^{n}, Z^n)$ for each $n \in \mathbb{N}$;
    \item[(2)] $\overline{\mathbb{P}}$-a.s., $\big(\bar{\xi}^n, \bar Z^n \big) \to (\bar{\xi},\bar{Z})$ as $n\to \infty$ in $D\bigl([0,T];H^{-\beta}(\mathbb T^2)\times\R^{\Z_0^2}\bigr)$.
\end{enumerate}

Furthermore, the $\overline\P$-a.s. convergence of $\bar\xi^n$ to $\bar\xi$ can be strengthened to $D([0,T];H^{-\varepsilon}(\mathbb T^2))$ for every $\varepsilon>0$, and hence to $L^p([0,T];H^{-\varepsilon}(\mathbb T^2))$ for every $1\le p<\infty$, by Proposition~\ref{Skorohod-interpolation} and the uniform $L^2$-bound. As in Proposition~\ref{qianrun}, the limit inherits the same $L^2$-bound. Applying the same interpolation argument to compact sets intersected with the common energy ball also gives tightness of the solution laws in $D([0,T];H^{-\varepsilon}(\mathbb T^2))$ for every $\varepsilon>0$.

Set $\bar u^n=K\ast \bar{\xi}^n$ and $\bar u=K\ast \bar{\xi}$. Let $\tilde {\bar N}^n= \big\{\tilde{\bar N}_k^n \big\}_k$ be the compensated Poisson random measures associated to $\bar Z^n=\{\bar Z_k^n\}_k$. The equality of joint laws
$\mathcal L_{\overline{\mathbb P}}(\bar\xi^n,\bar Z^n)=\mathcal L_{\mathbb P}(\xi^n,Z^n)$, together with the argument of Lemma~\ref{lem-skorohod-galerkin-transfer} and \eqref{flow-negative-Sobolev-increments} in place of the Galerkin estimates, shows that \eqref{scaling equ} holds for each copied pair in its completed, right-continuous natural filtration, with the  compensator $\nu(\mathrm{d}z)\,\mathrm{d}s$.

We now proceed to the proof of Theorem~\ref{2d-scal}. Compared to the proof in Section~\ref{subs-scaling-SLTE} dealing with the stochastic transport equations, the stochastic 2D Euler equation contains an additional nonlinear term $\bar u\cdot\nabla\bar{\xi}$. While all other terms can be treated analogously to those in Section~\ref{subs-scaling-SLTE}, for this term we use the convergence estimate already established in Section~\ref{Proof of main results}.

\begin{lemma}\label{nonlinear-con}
The nonlinear term in \eqref{scaling equ} converges in $L^1\bigl(\overline\Omega;L^{\infty}([0,T])\bigr)$ to $\int_0^t \big\langle \bar{\xi}(s), \bar u(s)\cdot\nabla\phi \big\rangle\,\mathrm{d}s$ as $n\to\infty$.
\end{lemma}

\begin{proof}
Fix $0<\varepsilon<1/2$. The estimate in Step~2 of the proof of Theorem~\ref{thm:weak_solution}, without the Galerkin projection, gives
\[
\begin{aligned}
&\overline\E\sup_{t\le T}\biggl|\int_0^t\bigl[ \big\langle \bar{\xi}^n(s),\bar u^n(s) \cdot\nabla\phi \big\rangle - \big\langle \bar{\xi}(s),\bar u(s)\cdot\nabla\phi \big\rangle\bigr]\,\mathrm{d}s\biggr|\\
&\quad\le C_{\phi,\varepsilon}\|\xi_0\|_{L^2}\,\overline\E\int_0^T \big\|\bar{\xi}^n(s)- \bar{\xi}(s) \big\|_{H^{-\varepsilon}} \,\mathrm{d}s\to0.
\end{aligned}
\]
Here the convergence follows from Proposition~\ref{Skorohod-interpolation}-(i), the common energy bound and dominated convergence. This proves the claimed convergence, as well as convergence in $L^1\bigl(\overline\Omega;C([0,T])\bigr)$.
\end{proof}

\begin{lemma}\label{2d-levy}
The last term in \eqref{scaling equ} $\overline{\mathbb P}$-a.s. converges uniformly in $t\in[0,T]$, to $\kappa\int_0^t\langle\bar{\xi}(s),\Delta\phi\rangle\,\mathrm{d}s$ as $n\to\infty$, where $\kappa=\frac12 C_2\int_{|z|\le1}z^2\,\nu(\mathrm{d}z)$.
\end{lemma}
\begin{proof}
Fix $\phi\in C^\infty(\mathbb T^2)$, and denote by $A_n(t)$
the last term in \eqref{scaling equ}. Throughout this proof,
almost sure statements are understood with respect to
$\overline{\mathbb P}$.

Let $\Psi_z^n$ be the remainder defined in the proof of
Lemma~\ref{lem-levy}, with $d=2$, and set
\[
R_n(t)
:=
\int_0^t\!\int_{|z|\le1}
\big\langle \bar{\xi}^n(s),\Psi_z^n \big\rangle
\,\nu(\mathrm{d}z)\,\mathrm{d}s.
\]
The same Taylor expansion gives, as in \eqref{levy-con-1},
\[
A_n(t)
=
\kappa\int_0^t
\big\langle \bar{\xi}^n(s),\Delta\phi \big\rangle\,\mathrm{d}s
+R_n(t).
\]
The estimate \eqref{levy-con-3} only uses the noise structure
and the uniform $L^2$-bound, and hence also holds for  the
present processes. Indeed, by \eqref{estimate-2}, which is
preserved under the above change of probability space,
\[
\sup_{t\in[0,T]}|R_n(t)|
\le
C_{\phi,T,\nu}\|\xi_0\|_{L^2}
\|\theta^n\|_{\ell^\infty}
\to0
\quad\overline{\mathbb P}\text{-a.s.},
\]
where the constant is independent of $n$.

Fix $\varepsilon>0$. By the almost sure convergence in
$D([0,T];H^{-\varepsilon}(\mathbb T^2))$ established above
and Proposition~\ref{Skorohod-interpolation}-(i), we have
\[
\|\bar{\xi}^n-\bar{\xi}\|_{L^1([0,T];H^{-\varepsilon}(\mathbb T^2))}
\to0
\quad\overline{\mathbb P}\text{-a.s.}
\]
Since $\Delta\phi\in H^\varepsilon(\mathbb T^2)$, it follows that
\[
\sup_{t\in[0,T]}
\bigg|
A_n(t)
-\kappa\int_0^t
\langle\bar{\xi}(s),\Delta\phi\rangle\,\mathrm{d}s
\bigg|
\le
\sup_{t\in[0,T]}|R_n(t)|
+
\kappa\,\|\Delta\phi\|_{H^\varepsilon}
\int_0^T
\|\bar{\xi}^n(s)-\bar{\xi}(s)\|_{H^{-\varepsilon}}\,\mathrm{d}s.
\]
The right-hand side converges to zero
$\overline{\mathbb P}$-a.s., which proves the lemma.
\end{proof}

\begin{lemma}\label{2d-m}
The martingale term in \eqref{scaling equ} converges to $0$ in $L^2(\overline\Omega;L^{\infty}([0,T]))$  as $n\to\infty$.
\end{lemma}

The proof of Lemma~\ref{2d-m} is analogous to that of
Lemma~\ref{lem-mart} and is therefore omitted. We now present the

\begin{proof}[Proof of Theorem~\ref{2d-scal}]
By Lemmas~\ref{nonlinear-con}--\ref{2d-m}, along a common further subsequence, the right-hand side of \eqref{scaling equ} converges $\overline{\mathbb P}$-a.s. uniformly in time  for a countable family of test functions dense in $C^2(\mathbb T^2)$. Then, one can show that $\overline{\mathbb P}$-a.s. for every $t\in[0,T]$ and $\phi\in C^2(\mathbb T^2)$,
\[
\langle\bar{\xi}(t),\phi\rangle=\langle\xi_0,\phi\rangle+\int_0^t\langle\bar{\xi}(s),\bar u(s)\cdot\nabla\phi\rangle\,\mathrm{d}s+\kappa\int_0^t\langle\bar{\xi}(s),\Delta\phi\rangle\,\mathrm{d}s.
\]
The Biot--Savart law implies that the associated velocity $\bar u=K\ast \bar{\xi}$ satisfies the weak formulation of the 2D Navier--Stokes equations with viscosity $\kappa$ and initial velocity $K\ast \xi_0$.

Since $\bar{\xi}\in L^\infty([0,T];L^2(\mathbb T^2))$, we have $\bar u\in L^\infty([0,T];H^1(\mathbb T^2))$, and hence $\bar u\in L^\infty([0,T];L^2(\mathbb T^2))\cap L^2([0,T];H^1(\mathbb T^2))$.
The uniqueness result in \cite[Theorem~3.1 and Remark~3.4]{Temam95} therefore identifies $\bar u$ as the unique deterministic velocity solution, and hence $\bar{\xi}$ is the unique deterministic vorticity solution $\xi$. Thus every subsequential limit has law $\delta_\xi$. By tightness, the whole sequence converges weakly: $\forall\, \varepsilon>0$,
\[
Q^n\rightharpoonup\delta_\xi
\quad\text{on }D\bigl([0,T];H^{-\varepsilon}(\mathbb T^2)\bigr),
\]
which completes the proof.
\end{proof}

\begin{appendices}
\renewcommand{\appendixname}{Appendix}
\renewcommand{\thesection}{\Alph{section}}

\section{Convergence of the Marcus maps}\label{third}
This appendix addresses the $L^2(\mathbb T^2)$-convergence of ${\rm e}^{z\theta_kA_k^n}\Pi_n\phi$ to ${\rm e}^{z\theta_k\sigma_k\cdot\nabla}\phi=\phi(\varphi_{z\theta_k}^k)$, which is used in Section~\ref{Proof of main results}. Recall that the notation ${\rm e}^{z\theta_kA_k^n}\Pi_n\phi$ stands for $\rho_n(1,\cdot)$, where the function $\rho_n(t,x)$ solves
\begin{equation}\label{eq:pde}
\frac{\mathrm{d}}{\mathrm{d}t}\rho_n(t)=z\theta_k\Pi_n\bigl(\sigma_k\cdot\nabla\rho_n(t)\bigr),\quad \rho_n(0)=\Pi_n\phi.
\end{equation}
Similarly, ${\rm e}^{z\theta_k\sigma_k\cdot\nabla}\phi$ coincides with $\rho(1,\cdot)$, where $\rho(t)=\phi(\varphi_{tz\theta_k}^k)$ solves
\begin{equation}\label{eq:pde-2}
\frac{\mathrm{d}}{\mathrm{d}t}\rho(t)=z\theta_k\sigma_k\cdot\nabla\rho(t),\quad \rho(0)=\phi.
\end{equation}
For $L^2$ initial data, \eqref{eq:pde-2} is understood in the weak sense. We have the following strong convergence result.
\begin{lemma}\label{C.1}
Fix $k\in\mathbb Z_0^2$, $z\in\mathbb R$ and $T>0$. For every $\phi\in L^2(\mathbb T^2)$ with zero spatial mean, the functions $\rho_n$ and $\rho$ defined above satisfy
\[
\lim_{n\to\infty}\|\rho_n(t)-\rho(t)\|_{L^2}=0\quad\text{for all }t\in[0,T].
\]
\end{lemma}
\begin{proof}
Since $\rho_n(t)\in H_n$ and $\Pi_n$ is self-adjoint, differentiating the squared $L^2(\mathbb T^2)$-norm of $\rho_n(t)$ gives
\begin{equation}\label{rho_n(t)}
\begin{aligned}
\frac{\mathrm{d}}{\mathrm{d}t}\|\rho_n(t)\|_{L^2}^2
&=2\biggl\langle\frac{\mathrm{d}}{\mathrm{d}t}\rho_n(t),\rho_n(t)\biggr\rangle
=2z\theta_k\bigl\langle\Pi_n\bigl(\sigma_k\cdot\nabla\rho_n(t)\bigr),\rho_n(t)\bigr\rangle\\
&=2z\theta_k\langle\sigma_k\cdot\nabla\rho_n(t),\rho_n(t)\rangle=0,
\end{aligned}
\end{equation}
since $\sigma_k$ is divergence-free. For smooth initial data, similarly,
\begin{equation}\label{rho(t)}
\frac{\mathrm{d}}{\mathrm{d}t}\|\rho(t)\|_{L^2}^2
=2\biggl\langle\frac{\mathrm{d}}{\mathrm{d}t}\rho(t),\rho(t)\biggr\rangle
=2z\theta_k\langle\sigma_k\cdot\nabla\rho(t),\rho(t)\rangle=0.
\end{equation}
For general $\phi\in L^2$, the identity $\|\rho(t)\|_{L^2}=\|\phi\|_{L^2}$ follows from $\rho(t)=\phi(\varphi_{tz\theta_k}^k)$ and preservation of Lebesgue measure. Consequently,
\begin{equation}\label{dengyu}
\|\rho_n(t)\|_{L^2}^2=\|\rho_n(0)\|_{L^2}^2=\|\Pi_n\phi\|_{L^2}^2\to\|\phi\|_{L^2}^2=\|\rho(0)\|_{L^2}^2=\|\rho(t)\|_{L^2}^2
\end{equation}
as $n\to\infty$. This also implies that $\{\rho_n\}_n$ is uniformly bounded in $L^\infty([0,T];L^2(\mathbb T^2))$. Hence there exist a subsequence $\{\rho_{n_j}\}_j$ and a limit $\bar\rho$ such that $\rho_{n_j}\rightharpoonup^\ast\bar\rho$ in $L^\infty([0,T];L^2(\mathbb T^2))$.

For any $\psi\in C_c^1([0,T);C^1(\mathbb T^2))$, multiplying both sides of \eqref{eq:pde} by $\psi$ and integrating on $[0,T]\times\mathbb T^2$, we obtain
\[
\int_0^T\biggl\langle\frac{\mathrm{d}}{\mathrm{d}t}\rho_n(t),\psi(t)\biggr\rangle\,\mathrm{d}t
=z\theta_k\int_0^T\bigl\langle\Pi_n\bigl(\sigma_k\cdot\nabla\rho_n(t)\bigr),\psi(t)\bigr\rangle\,\mathrm{d}t.
\]
Integrating by parts leads to
\[
\langle\Pi_n\phi,\psi(0)\rangle+\int_0^T\langle\rho_n(t),\psi'(t)\rangle\,\mathrm{d}t
-z\theta_k\int_0^T\bigl\langle\rho_n(t),\sigma_k\cdot\nabla\bigl(\Pi_n\psi(t)\bigr)\bigr\rangle\,\mathrm{d}t=0.
\]
Since $\Pi_n\phi\to\phi$ in $L^2$ and $\nabla\Pi_n\psi\to\nabla\psi$ in $L^1([0,T];L^2(\mathbb T^2))$, replacing $n$ by $n_j$ and letting $j\to\infty$ gives
\[
\langle\phi,\psi(0)\rangle+\int_0^T\langle\bar\rho(t),\psi'(t)\rangle\,\mathrm{d}t
-z\theta_k\int_0^T\langle\bar\rho(t),\sigma_k\cdot\nabla\psi(t)\rangle\,\mathrm{d}t=0.
\]
Thus $\bar\rho$ is a weak solution to \eqref{eq:pde-2} with initial data $\phi$. By uniqueness of weak solutions to \eqref{eq:pde-2} in $L^\infty([0,T];L^2(\mathbb T^2))$, the entire sequence $\{\rho_n\}_n$ converges weakly-$\ast$ to $\rho$ in this space.

To show that $\|\rho_n(t)-\rho(t)\|_{L^2}\to0$ for every $t\in[0,T]$, take an arbitrary time-independent $\psi\in C^1(\mathbb T^2)$. Integrating \eqref{eq:pde} against $\psi$ gives
\[
\begin{aligned}
\langle\rho_n(t),\psi\rangle
&=\langle\rho_n(0),\psi\rangle+z\theta_k\int_0^t\bigl\langle\Pi_n\bigl(\sigma_k\cdot\nabla\rho_n(s)\bigr),\psi\bigr\rangle\,\mathrm{d}s\\
&=\langle\Pi_n\phi,\psi\rangle-z\theta_k\int_0^t\bigl\langle\rho_n(s),\sigma_k\cdot\nabla(\Pi_n\psi)\bigr\rangle\,\mathrm{d}s,
\end{aligned}
\]
where the second equality follows from integration by parts. Letting $n\to\infty$, we obtain
\[
\lim_{n\to\infty}\langle\rho_n(t),\psi\rangle
=\langle\phi,\psi\rangle-z\theta_k\int_0^t\langle\rho(s),\sigma_k\cdot\nabla\psi\rangle\,\mathrm{d}s
=\langle\rho(t),\psi\rangle,
\]
where the last equality follows from \eqref{eq:pde-2}. By density of $C^1(\mathbb T^2)$ in $L^2(\mathbb T^2)$ and the uniform $L^2$ bound, $\rho_n(t)\rightharpoonup\rho(t)$ in $L^2(\mathbb T^2)$. Combining this with the convergence of the norms in \eqref{dengyu}, we obtain
\[
\|\rho_n(t)-\rho(t)\|_{L^2}\to0\quad\text{as }n\to\infty,
\]
which completes the proof.
\end{proof}

\end{appendices}

\bigskip

\noindent\textbf{Acknowledgements.} The first author would like to thank Professor Flandoli for helpful discussions at the initial stage of this work. He is also grateful to the financial supports of National Natural Science Foundation of China (No. 12595281), and the Open Foundation of the State Key Laboratory of Mathematical Sciences, Academy of Mathematics and Systems Science (Grant No. 09).

\end{document}